%% file: arxivFile.tex
\documentclass[10pt,a4paper]{amsart}
\usepackage{amsmath,amsthm,verbatim,amscd,amssymb,setspace,enumitem,hyperref}
\usepackage{cleveref}
\usepackage{exscale,color}
\usepackage[colorinlistoftodos,prependcaption]{todonotes}
\usepackage{graphicx}
\usepackage{tikz}
\usetikzlibrary{patterns}
\usepackage{enumitem}
\usepackage{mathrsfs}

\renewcommand{\epsilon}{\varepsilon}
\newcommand{\N}{\mathbb{N}}

\newcommand{\R}{\mathbb{R}}

\renewcommand{\Re}{\operatorname{Re}}

\newcounter{mtheorem}
\newtheorem{mtheorem}[mtheorem]{Theorem}

\newtheorem{mcor}[mtheorem]{Corollary}

\newcommand{{\vol}}{\rm vol}

\newcommand{\Ric}{\operatorname{Ric}}

\newcommand{\Rm}{\operatorname{Rm}}

\def\tr{\operatorname{tr}}

\def\Rm{\operatorname{Rm}}

\newtheoremstyle{fancy}{}{}{\itshape}{}{\textbf\bgroup}{.\egroup}{ }{}
\newtheoremstyle{fancy2}{}{}{\rm}{}{\textbf\bgroup}{.\egroup}{ }{}

\theoremstyle{fancy}
\newtheorem{theorem}{Theorem}[section]
\newtheorem{lemma}[theorem]{Lemma}
\newtheorem{corollary}[theorem]{Corollary}
\newtheorem{defn-prop}[theorem]{Definition \& Proposition}
\newtheorem{prop}[theorem]{Proposition}

\crefname{maintheorem}{Theorem}{Theorems}
\Crefname{maintheorem}{Theorem}{Theorems}

\theoremstyle{fancy2}
\newtheorem{definition}[theorem]{Definition}

\newtheorem{remark}[theorem]{Remark}

\theoremstyle{definition}

\setlist{leftmargin=*}

\numberwithin{equation}{section}

\begin{document}
\title{Expanding Soliton Models for K\"ahler-Ricci Flow Near Conical Singularities}

\author{Longteng Chen}
\address[Longteng Chen]{Université Paris-Saclay, CNRS, Laboratoire de Mathématiques d'Orsay, 91405 Orsay, France }
\email{longteng.chen@universite-paris-saclay.fr}
\author{Max Hallgren}
\address[Max Hallgren]{University of Arizona, Department of Mathematics, 617 N. Santa Rita Ave., Tucson, AZ 85721}
\email{mhallgren@arizona.edu}
\author{Lucas Lavoyer}
\address[Lucas Lavoyer]{Mathematisches Institut, Universit\"at M\"unster, 48149 M\"unster, Germany}
\email{lucas.lavoyer@uni-muenster.de}

\begin{abstract}

Let $(Y,g_0)$ be a compact analytic space with a finite number of singular points, where the metric at each singular point is modeled on a K\"ahler cone with smooth canonical model. We show that the K\"ahler-Ricci flow with such initial data satisfies a $C/t$ curvature bound, and that the flow near each singular point is modeled on the unique K\"ahler-Ricci expander asymptotic to the corresponding cone. Our motivation is to give a geometric description of the K\"ahler--Ricci flow emerging from singularities arising in the analytic minimal model program.

\end{abstract}
\maketitle
\section{Introduction}

Given a smooth and closed Riemannian manifold $(M^n ,g_0)$, it was shown by Hamilton \cite{Ham3D} that a solution to the Ricci flow
\begin{equation*}
\frac{\partial}{\partial t}g(t)=-2\operatorname{Ric}(g(t)) \quad \textnormal{for  } t\in [0,T)
\end{equation*}
with initial data $g(0)=g_0$ exists uniquely for a short time. Since then, there has been substantial interest in extending the theory to allow singular initial data. In dimension three, Simon and Topping \cite{Simon1,Simon2, Simon3,simon-topping2} established the existence of smooth Ricci flows starting from metric spaces that arise as limits of noncollapsed manifolds with Ricci curvature bounded below. For other directions in flowing out of singular spaces, see \cite{DeruelleGuedjGuenanciaZeriahi} for the K\"ahler–Ricci flow from metric spaces and \cite{SimonSurvey} for a survey in the Riemannian setting.

Near singularities of $g_0$, the behavior of the Ricci flow is subtle and often difficult to describe geometrically. One situation where a description is available is when $g_0$ arises as a singular time-slice of a singular Ricci flow in dimension three (in the sense of \cite{KleinerLott,BamlerKleiner}) and in certain higher-dimensional symmetric settings \cite{minimallyinvasive,Hasl,Butts}. In these settings, the class of geometries modeling high-curvature regions is classified, and as a result, the flow out of such a singularity is known to be modeled on a steady Ricci soliton constructed by Bryant \cite{Bryant2005}. 

In dimensions four and higher, isolated conical singularities are expected to arise frequently at first singular times \cite{BamlerSurvey}. In this direction, Gianniotis and Schulze \cite{Gianniotis-Schulze} constructed Ricci flows emerging from metrics with isolated cone singularities, assuming the cone models have positive curvature (see also \cite{lavoyer} for the case where the singularities occur along a closed curve). This result suggests a way to continue the flow past the singular time in a way that it is immediately smooth again (see also \cite{FIK,SW1} for an example of this behavior). 
Their construction glues in expanding gradient Ricci solitons asymptotic to the cones, combining an existence result of Deruelle \cite{DeruelleGAFA} with a stability result due to Deruelle and Lamm \cite{Deruelle-Lamm}. This yields a detailed geometric description of the flow emanating from singular points, where the short-time behavior corresponds to desingularization by the expander. See also \cite{ChanLaiManChun} for a strengthening of this result, and \cite{ChodoshHolgateSchulze} for related work in mean curvature flow.

In this work, we consider this problem in the context of K\"ahler--Ricci flow. The formulation of the K\"ahler-Ricci flow as a parabolic complex Monge-Amp\`ere equation allowed Song and Tian \cite{songtian} to prove existence and uniqueness results for K\"ahler-Ricci flow with singular initial data in great generality. These results have been extended in later works \cite{BouGue,DiLu,GZTwisted}. On the other hand, the geometric properties of these solutions near singularities remains largely unknown. Therefore, our main goal is to give a good description of the behavior of the solution as it desingularizes the initial data.

Motivated by the Riemannian work of \cite{Gianniotis-Schulze}, we construct our solutions via a gluing and limiting procedure, where the initial data is obtained by inserting asymptotically conical (AC) expanding K\"ahler–Ricci solitons at small scales. A key difference is that, in contrast to the Riemannian setting, we do not assume any curvature sign condition on the cone models. This is possible because existence and uniqueness of asymptotically conical K\"ahler expanders is more robust \cite{ConlonDeruelleJDG,ConlonDeruelleSun}. However, the weak stability theory used in \cite{Gianniotis-Schulze} is not available in this generality, and a central difficulty in our work is to control the gluing without such stability assumptions.

We start by defining the type of initial data we will be considering. 

\begin{definition}\label{defn of sing space}
    Suppose $(S_i,g_{S_i})$ are compact Sasaki manifolds. We say that $(Y,g_0)$ is a compact analytic space with isolated conical singularities at $\{y_i\}_{i=1}^Q\subset Y$ modeled on the K\"ahler cones $(\mathcal{C}(S_i),g_{\mathcal{C}_i})$ if $Y$ is a compact analytic space and the following hold;
   \begin{enumerate}
       \item  \label{defn:singspace1} $\left(Y\setminus\{y_1,\ldots,y_Q\},g_0\right)$ is a smooth K\"ahler manifold.
       \item \label{defn:singspac3}
       There exist open holomorphic embeddings $\phi_i:\mathcal{C}(S_i)\supseteq \{r<r_0\}\to Y$ with pairwise disjoint images such that $\phi_i(o_i)=y_i$, where $o_i \in \mathcal{C}(S_i)$ denotes the $i$th vertex.
         Moreover, there exist $u_i \in C^{\infty}(\{0<r<r_0\})\cap C^0(\{r<r_0\})$
       such that $\phi_i^*\omega_0-\omega_{\mathcal{C}_i}=\sqrt{-1}\partial \bar\partial u_i$ and
       \begin{equation} \label{eq:estimatesforu1}
           \begin{split}
               r^{j-2}|(\nabla^{g_{\mathcal{C}_i}})^ju_i|_{g_{\mathcal{C}_i}}\le k_j(r)\quad \forall j\in\N,
           \end{split}
       \end{equation}
       where $k_j: (0,r_0]\to (0,1)$ is an increasing function satisfying $\lim_{r\to 0^+}k_j(r)=0$.
   \end{enumerate}
\end{definition}

Given the definition above, our main result of this paper is the following. 
\begin{mtheorem} \label{main theorem}
    Let $(Y,g_0)$ be a compact analytic space with isolated conical singularities at $\{y_i\}_{i=1}^Q\subset Y$, each modeled on a K\"ahler cone $(\mathcal{C}(S_i),g_{\mathcal{C}_i})$ with smooth canonical model in the sense of Definition \ref{defn smooth canonical model}. 
    
    Then there exists a compact K\"ahler manifold $M$, a smooth K\"ahler--Ricci flow $(g(t))_{t\in (0,T]}$ on $M$ and a constant $C_M \in (0,\infty)$ such that the following hold:
    \begin{enumerate}
        \item \label{thm:mainthm1} $(M,d_{g(t)})\to (Y,d_Y)$ as $t\to0^+$ in the Gromov--Hausdorff topology, where $(Y,d_Y)$ is the metric completion of $\left(Y\setminus\{y_1,...,y_Q\}, g_0\right)$.
        \item \label{thm:mainthm2} There exists a K\"ahler resolution $\pi: M\to Y$ such that $\lim_{t\to 0^+}[\omega(t)]=\pi^{\ast}[\omega_0]$ and $\pi_*g(t)$ converges to $g_0$ locally smoothly uniformly away from $\{y_i\}_{i=1}^Q$, as $t\to 0^+$.
        \item \label{thm:mainthm3} $\max_M|\Rm(g(t))|_{g(t)}\le\frac{C_M}{t}$ for all $t\in (0,T]$.
        \item \label{thm:mainthm4} For any $p\in \pi^{-1}(y_i)$, there exists $q\in E_i$ such that
        \begin{equation*}
             \left(M,\tau^{-1}g(\tau t),\ p\right)_{t\in (0,\tau^{-1}T]}
            \to \left(E_i,\ g_{E_i}(t),\ q\right)_{t\in (0,\infty)}
        \end{equation*}
        in the smooth pointed Cheeger--Gromov topology as $\tau\to 0^+$, where  
        $(E_i, g_{E_i}(t))$ is the self-similar K\"ahler--Ricci flow induced by the unique AC K\"ahler-Ricci expander $(E_i,g_{E_i})$ asymptotic to $(\mathcal{C}(S_i),g_{\mathcal{C}_i})$.

           \item \label{thm:mainthm5} Assume that $(\widetilde{M},\widetilde g(t))_{t\in (0,T]}$ is a K\"ahler-Ricci flow on a compact K\"ahler manifold $\widetilde{M}$, and suppose there is a modification $\widetilde{\pi}:\widetilde{M}\to Y$  with $\widetilde{\pi}^{\ast}[\omega_0]= \lim_{t\to 0^+}[\widetilde{\omega}(t)]$ and such that $\widetilde{\pi}_*\tilde g(t)$ converges to $g_0$ locally smoothly away from the singular set. If the scalar curvature of this flow is bounded from above by $\frac{C}{t}$, then there is a biholomorphism $H:M\to \widetilde{M}$ such that $\widetilde{\pi}\circ H=\pi$ and $H^{\ast}\tilde g(t)=g(t)$ for all $t\in (0,T]$.
    \end{enumerate}   
\end{mtheorem}
\begin{remark} 
    In Proposition \ref{prop:uniqueness}, we prove the uniqueness statement \ref{thm:mainthm5}, and show that our flows coincide with certain solutions constructed in \cite{songtian} and related works. Moreover, our argument implies that the scalar curvature bound can be replaced by the assumption that the flow has uniformly bounded local K\"ahler potentials. Thus, Theorem \ref{main theorem} provides conditions, stated purely in terms of the initial data, under which the solutions of \cite{songtian} satisfy the geometric properties \ref{thm:mainthm1}–\ref{thm:mainthm4}. This is in contrast to the Riemannian setting, where such uniqueness is not expected to hold in general \cite{AngenentKnopf}.
    \end{remark}
    
    \begin{remark}
        The proof of the above theorem further shows that the pointed Cheeger–Gromov convergence of \ref{thm:mainthm4} is in fact realized by biholomorphisms, and at the level of K\"ahler potentials. See Proposition \ref{prop:citablepotentialbound} for a more precise statement.
    \end{remark}

   \begin{remark} It is likely that our methods generalize to the case where the canonical model of the cone has orbifold singularities. In this case, the flow constructed in Theorem \ref{main theorem} will be a smooth orbifold K\"ahler-Ricci flow. Such flows have been shown to arise when flowing through K\"ahler-Ricci flow singularities in dimensions three and higher \cite{SW2}.
   \end{remark}

   \begin{remark} In the case where in addition each $\mathcal{C}(S_i)$ is positively curved, \cite{Gianniotis-Schulze} also gives a Ricci flow solution $(g(t))_{t\in (0,T]}$ satisfying \ref{thm:mainthm1}, \ref{thm:mainthm3}, \ref{thm:mainthm4}. However, the flow produced in that work is not known to be unique, and may not be K\"ahler, even if the initial data is K\"ahler. Theorem \ref{main theorem} implies that in this setting there is in fact a unique flow out of the singularity which is K\"ahler. 
   \end{remark}
    
One motivation for Theorem \ref{main theorem} is to describe geometrically the K\"ahler--Ricci flow emerging from singularities which arise in the analytic minimal model program. Song and Tian \cite{songtian} showed that if a projective K\"ahler manifold with rational K\"ahler class develops a finite-time singularity under the K\"ahler-Ricci flow, then the flow can be canonically continued on a new projective variety $X'$, related to $X$ by a birational map factoring through a singular space $Y$. Unlike the flow through singularities in three-dimensional Riemannian Ricci flow by Perelman \cite{Perelman2003}, the construction of this flow does not rely on a geometric description of the singularity formation, and in fact such a description is unknown for the flows constructed in \cite{songtian}. In this direction, our results identify conditions on the initial data under which the Song-Tian solutions admit a precise geometric description near isolated conical singularities.

While our main focus is to describe K\"ahler-Ricci flow through singularities, the notion of a compact analytic space with isolated conical singularities is much more general. In fact, we show that any asymptotically conical gradient K\"ahler-Ricci expander occurs as model for flowing out of such a space, contrasting sharply with the very limited number of expanders expected to arise when flowing through singularities.

\begin{mcor} \label{cor:examples}
    For any asymptotically conical gradient K\"ahler-Ricci expander $(E,g_E)$, there exists a K\"ahler-Ricci flow $(g_t)_{t\in (0,T]}$ on a projective manifold $M$ and a subvariety $D\subseteq M$ such that the following hold:

    \begin{enumerate}
        \item $\max_M |\Rm(g(t))|_{g(t)} \leq \frac{C_M}{t}$ for all $t\in (0,T]$,

        \item $g(t)$ converges smoothly to a smooth K\"ahler metric on $M\setminus D$ as $t\to 0+$,

        \item for any $y\in D$, we have
        \begin{equation*}
            (M,\tau^{-1}g(\tau t),y)_{t\in (0,\tau^{-1}T]} \to (E,g_E(t),q)_{t\in (0,\infty)}
        \end{equation*}
        in the smooth pointed Cheeger-Gromov topology as $\tau \to 0^+$ for some $q\in E$, where $(g_E(t))_{t\in (0,\infty)}$ is the self-similar flow corresponding to $g_E$. 
    \end{enumerate}
\end{mcor}

\noindent \textbf{Overview.} We now briefly explain the main steps of our proof and outline the structure of the paper. In Section \ref{preliminaries}, we present the relevant definitions and properties of K\"ahler cones and expanding gradient K\"ahler--Ricci solitons.
 
 As in \cite{Gianniotis-Schulze}, we construct our solution via a glueing procedure, which we describe in Section \ref{sectn; approx soltn}. For simplicity, we treat the case of a single singularity; since the results are local, the general case follows analogously. We remove a small neighborhood of the singular point and glue in, at a small scale $s>0,$ the expanding K\"ahler--Ricci soliton asymptotic to the corresponding K\"ahler cone. This is done at the level of K\"ahler potentials, and we show this yields a smooth K\"ahler manifold $(M,\omega_{s,0}).$ 

We can then study the K\"ahler--Ricci flow on $M$ with initial data $\omega_{s,0}$. Writing the flow in terms of the complex Monge--Ampère equation for $\varphi_s(t),$ $t\in [0,T_s),$ we first establish estimates in the region where the geometry is approximately conical, using pseudolocality and arguments closely following \cite{Gianniotis-Schulze}. To obtain estimates for the flow near the singular point, we study the complex Monge--Ampère equation with boundary data. By considering the normalized K\"ahler--Ricci flow instead, we fix the reference expander metric in time. The idea stems from the work of the first author in \cite{longteng2} and puts us in a better position to prove the desired $C^2$-estimates for our solution to the complex Monge--Ampère equation in the spirit of Yau's approach \cite{YauCalabiConjecture}. These estimates are obtained in Section \ref{sectn; uni estim on exp reg}.

The core of our analysis consists of a priori estimates for several quantities along this modified K\"ahler--Ricci flow. The techniques of this section might be of independent interest 
(we refer also to \cite{longteng2} for related computations). The price to pay for considering the modified flow is that the evolution of our new K\"ahler potential now involves a drift term coming from the soliton vector field of $(E,g_E,f_E)$ (see equation \eqref{equation of normalized complex monge ampere}). As an extra step, essentially because of this extra drift term, we need to carefully construct a barrier function (Lemma \ref{barrier function}) before proving our $C^2$-estimates. A maximum principle argument, together with good control on our barrier function, then allows us to obtain our desired control on the K\"ahler potential, and proving higher derivative bounds is then straightforward. 

In Section \ref{sectn; constructing the flow} we finish the proof of Theorem \ref{main theorem} by letting $s\searrow 0$ and showing that $\omega_s(t)$ converges to a limit K\"ahler--Ricci flow with the stated properties. We then prove the uniqueness of the solution in a natural class and relate it to the solutions of Song and Tian \cite{songtian}. Finally, we produce examples of K\"ahler-Ricci flows satisfying the assumptions of Theorem \ref{main theorem}, proving Corollary \ref{cor:examples}.\\

\textit{Acknowledgments}: The authors thank Ronan Conlon, Alix Deruelle, Hajo Hein, Jian Song, and Junsheng Zhang for useful discussions on the K\"ahler--Ricci flow, and Ronan Conlon for helpful comments and suggestions. LL is funded by the German Research Foundation (DFG) – Project-ID 427320536 – SFB 1442, and by Germany’s Excellence Strategy EXC 2044/2 390685587, Mathematics Münster: Dynamics–Geometry–Structure.\\

\textit{AI Disclosure:} The writing and mathematical content of this paper is due to the authors. ChatGPT was used only to assist with literature searches and to identify inconsistencies in notation and minor errors.

\section{Preliminaries}\label{preliminaries}

\subsection{Conventions and notation}

We collect here the basic notation and conventions used throughout the paper. For more specialized notation, we instead indicate the location in the paper where these definitions are given.

We let $\mathbb{N}$ denote the set of nonnegative integers.\\

\paragraph{\emph{K\"ahler geometry.}}
Unless otherwise stated, all manifolds are smooth and connected, and
$n$ denotes their complex dimension. If $(M,J,g)$ is K\"ahler, we denote its K\"ahler form
by
\begin{equation*}
\omega(\cdot,\cdot)=g(J\cdot,\cdot)
\end{equation*}
and its Riemannian volume form by
\begin{equation*}
dg=\frac{\omega^n}{n!}.
\end{equation*}
We use $\nabla^g$ for the Levi--Civita connection, and write
$\langle\cdot,\cdot\rangle_g$ and $|\cdot|_g$ for the corresponding
inner product and norm. When the metric is clear from the context, the
superscript or subscript $g$ may be omitted. The musical
isomorphisms corresponding to $g$ are denoted by $\sharp$ and $\flat$.

Our Laplacian is the K\"ahler Laplacian
\begin{equation*}
\Delta_\omega u
 :=
 \operatorname{tr}_{\omega}
 \bigl(\sqrt{-1}\partial\overline{\partial}u\bigr)
 =
 g^{i\overline{j}}
 \partial_i\partial_{\overline{j}}u,
\end{equation*}
which is equal to half of the Riemannian Laplace-Beltrami operator. We also write $\Delta_g$ for this operator when $g$ is the metric
associated with $\omega$. 
In particular, $-\Delta_g$ has
nonnegative spectrum on a compact manifold. 

The Ricci form and K\"ahler scalar curvature are
\begin{equation*}
\Ric(\omega)
 =
 -\sqrt{-1}\partial\overline{\partial}
 \log\omega^n,
\qquad
\operatorname{R}_\omega = \operatorname{R}_g
 =
 \operatorname{tr}_{\omega}\Ric(\omega).
\end{equation*}
We write $\Ric(g)$ and $\Rm(g)$ for the Ricci and Riemann curvature
tensors. Scalar curvature is always taken in the K\"ahler
normalization above; it is one half of the Riemannian scalar curvature.\\

\paragraph{\emph{Metric notation.}}
For a Riemannian metric $g$, we let $d_g$ denote the corresponding length metric and $B_g(x,r)$ the open geodesic ball. We similarly let
\begin{equation*}
\operatorname{Vol}_{g},
\qquad
\operatorname{diam}_{g},
\qquad
\operatorname{inj}_{g}
\end{equation*}
denote volume, diameter, and injectivity radius with respect to $g$.

For a tensor $T$ on an open set $U\subset M$, we define 
\begin{equation*}
\|T\|_{C^k(U,g)}
 :=
 \sum_{j=0}^k
 \sup_U
 \bigl|(\nabla^g)^jT\bigr|_g.
\end{equation*}
\\

\paragraph{\emph{K\"ahler cones and AC K\"ahler-Ricci expanders.}}

For a K\"ahler cone $(\mathcal{C},\omega_{\mathcal{C}})$, we let $o\in \mathcal{C}$ denote the vertex and define the radius function to be $r:=d_{g_{\mathcal{C}}}(o,\cdot)$. K\"ahler-Ricci expanders $(E,g_E,X)$ are defined in Subsection \ref{subsection:conesandexpanders}. We will always write $X:= \nabla^{g_E}f$, where $f$ is the soliton potential, normalized as in Lemma \ref{soliton indentities}. The biholomorphisms $(\Phi_t)_{t\in (-\infty,0)}$ associated to the soliton, and the corresponding time dependent metrics $g_{E,t}$ are also defined in Subsection \ref{subsection:conesandexpanders}.

The notations $\mathcal{B}(D),\overline{\mathcal{B}}(D),\mathcal{A}(D,D')$ are established in Subsection \ref{subsection:ACexpanders}.\\

\paragraph{\emph{Constants and asymptotic notation.}}
The notation
\begin{equation*}
\Psi(a_1,\ldots,a_k\mid b_1,\ldots,b_\ell)
\end{equation*}
denotes a nonnegative quantity depending on the displayed parameters
such that
\begin{equation*}
\lim_{(a_1,\ldots,a_k)\to(0,\ldots,0)}
\Psi(a_1,\ldots,a_k\mid b_1,\ldots,b_\ell)
=0
\end{equation*}
for every fixed choice of $b_1,\ldots,b_\ell$. 

\subsection{K\"ahler cones and AC K\"ahler-Ricci expanders}
\label{subsection:conesandexpanders}
We start by defining the K\"ahler cones appearing in Theorem \ref{main theorem}. The resolutions of K\"ahler cones that are consistent with admitting an expanding K\"ahler--Ricci soliton are of the following type.
\begin{definition}[cf. {\cite[Definition 8.2.4]{Ishii}}]\label{defn smooth canonical model}
Suppose $(C_0,o)$ is a pointed normal affine variety such that $C_0 \setminus \{o\}$ is smooth. A partial resolution $\pi:M\to C_{0}$ is called a \emph{canonical model} if
\begin{enumerate}[label=\textnormal{(\roman{*})}, ref=(\roman{*})]
\item $M$ has at worst canonical singularities;
\item $K_{M}$ is $\pi$-ample;
\item $\pi$ restricts to a biholomorphism $M\setminus \pi^{-1}(o)\to C_0 \setminus \{o\}$. 
\end{enumerate}
\end{definition}
\begin{remark}
    If $\pi':M'\to C_0$ is another canonical model, then there is a biholomorphism $F:M\to M'$ satisfying $\pi'\circ F=\pi$ (cf. \cite[Proposition 8.2.5]{Ishii}). 
\end{remark}
\begin{definition}
    We say $C_0$ has a smooth canonical model if there is a canonical model $\pi:M\to C_0$ such that $M$ is smooth.
\end{definition}

A triple $(E,g_E,X)$, where $(E,g_E)$ is a complete K\"ahler manifold and $X=\nabla^{g_E}f$ a gradient real-holomorphic vector field on $E$, is said to be a K\"ahler--Ricci expander if it satisfies the equation
\begin{equation*}
   \Ric(g_E)+g_E= \frac{1}{2}\mathcal{L}_Xg_E.
\end{equation*}
As a consequence, if $\omega_E$ denotes the K\"ahler form of $g_E$, the corresponding expanding soliton equation in terms of $(1,1)$ forms is stated as follows:
\begin{equation*}
    \Ric(\omega_E)+\omega_E=\sqrt{-1}\partial\bar\partial f.
\end{equation*}
In this paper, we primarily focus on K\"ahler--Ricci expanders that are asymptotic to cones at large scales.

\begin{definition} \label{def:ACexpander} We say $(E,g_E)$ is an AC K\"ahler-Ricci expander if it is noncompact and for some (hence any) $p\in E$, we have
\begin{equation} \label{eq:quadraticdecayofexpander}
\sup_{x\in E}|(\nabla^{g_E})^{k}\operatorname{Rm}(g_E)|_{g_E}(x)d_{g_E}^{2+k}(p,\,x)<\infty\quad\textrm{for all $k\in\mathbb{N}$}.
\end{equation}
\end{definition}
\begin{remark} By \cite[Theorem A]{ConlonDeruelleSun}, for any AC K\"ahler-Ricci expander $(E,g_E)$, there exists a unique K\"ahler cone $(\mathcal{C},g_{\mathcal{C}})$ such that for any $p\in E$, $(E,\lambda d_{g_E},p)\to (\mathcal{C},d_{g_{\mathcal{C}}},o)$ in the pointed Gromov-Hausdorff sense as $\lambda \to 0^+$. We call $(\mathcal{C},g_{\mathcal{C}})$ the asymptotic cone of $(E,g_E)$.
\end{remark}

The following gives a necessary and sufficient criterion for the existence of an AC K\"ahler-Ricci expander which is asymptotic to prescribed K\"ahler cone.

\begin{theorem}[{\cite[Corollary B]{ConlonDeruelleSun}}]\label{theorem of conlonderuellesun}
Let $(\mathcal{C},\,g_{\mathcal{C}})$ be a K\"ahler cone, with radial function $r$.
Then there exists an AC K\"ahler-Ricci expander $(E,\,g_E,\,X)$ with asymptotic cone
$(\mathcal{C},\,g_{\mathcal{C}})$ if and only if $\mathcal{C}$ has a smooth canonical model. When this is the case, there is a birational map $\pi:E\to \mathcal{C}$ which is a canonical model satisfying $d\pi(X)=r\partial_{r}$ and
\begin{equation*}
|(\nabla^{g_{\mathcal{C}}})^k(\pi_{*}g_E-g_{\mathcal{C}})|_{g_{\mathcal{C}}} \leq C_{k}r^{-2-k}\quad\textrm{for all $k\in\mathbb{N}$}.
\end{equation*}
\end{theorem}
\begin{remark}\label{self-similar sol}
Let $\Phi_X^{\cdot}$ denote the flow of the vector field $X$. The completeness of $X$ follows from the completeness of $g_E$ (see \cite{ZhuhongZhang}). For each $t>0$, define the biholomorphism
\begin{equation*}
    \Phi_t := \Phi_X^{-\frac{1}{2}\log t} : E \to E .
\end{equation*}
Then the self-similar solution defined as
\begin{equation*}
    g_E(t) := t\Phi_t^{*} g_E,
\end{equation*}
satisfies
\begin{equation}\label{eq: convergence rate of self-similar solution}
    \bigl|(\nabla^{g_{\mathcal{C}}})^{k}\bigl(\pi_{*} g_E(t) - g_{\mathcal{C}}\bigr)\bigr|_{g_{\mathcal{C}}}\le C_kt r^{-2-k}
\qquad \text{for all } k\in\N .
\end{equation}
This estimate implies that the self-similar solution converges locally smoothly to the conical metric $g_{\mathcal{C}}$.
\end{remark}

\begin{lemma}\label{appearance of uE}
    Let $(\mathcal{C},g_{\mathcal{C}})$ be a K\"ahler cone admitting a smooth canonical model, and let $(E,g_E,X)$ be the K\"ahler--Ricci soliton as in Theorem \ref{theorem of conlonderuellesun}. Let $\pi$ be the corresponding K\"ahler resolution. Then there exists a smooth function $u_E\in C^\infty(\mathcal{C}\setminus\{o\};\R)$ such that
    \begin{equation*}
        \pi_*\omega_E-\omega_{\mathcal{C}}+\Ric(\omega_\mathcal{C})=\sqrt{-1}\partial\bar\partial u_E,
    \end{equation*}
    holds on $\mathcal{C}\setminus\{o\}$.
    \end{lemma}
    \begin{proof}
We identify $\mathcal{C}\setminus\{o\}$ with its image in $E$ via the biholomorphism $\pi^{-1}$. Under this identification, $r\partial_r$ is mapped to $X$.  Let $(\omega_E(t))_{t>0}$ be the self-similar solution to the K\"ahler--Ricci flow as in Remark \ref{self-similar sol}. By the K\"ahler--Ricci flow equation, for any $x\in \mathcal{C}\setminus\{o\}$ we have
        \begin{equation*}
            \begin{split}
                \left(\omega_E-\omega_{\mathcal{C}}+\Ric(\omega_{\mathcal{C}})\right) (x)&=\left(\omega_E(1)-\omega_{\mathcal{C}}+\Ric(\omega_{\mathcal{C}})\right)(x)\\
                &=\int_0^1\left(\Ric(\omega_{\mathcal{C}})-\Ric(\omega_E(t))\right)(x)dt\\
                &=\int_0^1\sqrt{-1}\partial\bar\partial\log\frac{\omega_E(t)^n}{\omega_{\mathcal{C}}^n}(x)dt\\
                &=\sqrt{-1}\partial\bar\partial\int_0^1\log\frac{\omega_E(t)^n}{\omega_{\mathcal{C}}^n}(x)dt.
            \end{split}
        \end{equation*}
        It follows immediately that $\pi_*\omega_E-\omega_{\mathcal{C}}+\Ric(\omega_{\mathcal{C}})=\sqrt{-1}\partial\bar\partial u_E$ holds for $u_E:=\int_0^1\log\frac{\omega_E(t)^n}{\omega_{\mathcal{C}}^n}(x)dt$ on $\mathcal{C}\setminus\{o\}$.
    \end{proof}
The convergence rate of the self-similar solution toward the conical metric yields the following corollary.

  \begin{corollary}
      Let $(\mathcal{C},g_{\mathcal{C}})$ be a K\"ahler cone admitting a smooth canonical model and let $(E,g_E,X)$ be the expanding soliton as in Theorem \ref{theorem of conlonderuellesun}. Let $u_E$ be the function as in Lemma \ref{appearance of uE}. Then there exist constants $\{C_k>0\}_{k\in\N}$ such that on $\{r^2\ge 1\}$,
      \begin{equation}\label{decay of uE}
          \begin{split}
              &|(\nabla^{g_{\mathcal{C}}})^ku_E|_{g_{\mathcal{C}}}\le \frac{C_k}{r^k},\quad \textnormal{for all $k\in\N$}.
          \end{split}
      \end{equation}
  \end{corollary} 
  \begin{proof}
      Let $\omega_E(t)$ denote the self-similar solution. Then, by previous computation, we have 
      \begin{equation*}
          u_E=\int_0^1\log\frac{\pi_*\omega_E(t)^n}{\omega_{\mathcal{C}}^n}dt.
      \end{equation*}
     Thanks to \eqref{eq: convergence rate of self-similar solution}, on $\mathcal{C}\setminus\{o\}$, we have
     \begin{equation*}
        |( \nabla^{g_{\mathcal{C}}})^k u_E|_{g_\mathcal{C}}\le C_kr^{-k-2},\quad \textnormal{for all $k\in\N.$}
     \end{equation*} In particular, on $\{r^2\ge 1\}$,
      \begin{equation*}
          \begin{split}
              &|(\nabla^{g_{\mathcal{C}}})^ku_E|_{g_{\mathcal{C}}}\le \frac{C_k}{r^k},\quad \textnormal{for all $k\in\N$}.
          \end{split}
      \end{equation*}
  \end{proof}
 \begin{corollary}\label{decay of uE(s)}
      For all $s>0$, we define $u_E(s):=s\Phi_s^*u_E$, where $\Phi_s$ is as in Remark \ref{self-similar sol}. Then, we have
      \begin{equation*}
          \omega_E(s)-\omega_{\mathcal{C}}+s\Ric(\omega_{\mathcal{C}})=\sqrt{-1}\partial\bar\partial u_E(s).
      \end{equation*}
      Moreover, there exist constants $\{C_k>0\}_{k\in\N}$ such that for all $k\in\N,$ the following   holds on $\{r^2\ge s\}$:
      \begin{equation*}
          \begin{split}
              &|(\nabla^{g_{\mathcal{C}}})^ku_E(s)|_{g_{\mathcal{C}}}\le \frac{C_ks}{r^k}.
          \end{split}
      \end{equation*}
  \end{corollary} 
  \begin{proof}
     Recall that $\Phi_s$ is the flow of $-\frac{1}{2s}X$ for all $s>0$ and $X=r\partial_r$ on the K\"ahler cone $(\mathcal{C},g_{\mathcal{C}}).$ Thus,  $s\Phi_s^*g_{\mathcal{C}}=g_{\mathcal{C}}$ and $r^2\circ\Phi_s=\frac{r^2}{s}$ for all $s>0$. Then by previous computation, we have that
      \begin{equation*}
         \begin{split}
         &\omega_E(s)-\omega_{\mathcal{C}}+s\Ric(\omega_{\mathcal{C}})=\sqrt{-1}\partial\bar\partial u_E(s),\\
              &|(\nabla^{g_{\mathcal{C}}})^ku_E(s)|_{g_{\mathcal{C}}}\le \frac{C_ks}{r^k},\quad \textnormal{for all $k\in\N$},
          \end{split}
      \end{equation*}
     holds on $\{r^2\ge s\}$.
  \end{proof}
  
  \subsection{Further properties of AC expanders}
\label{subsection:ACexpanders}
  
  Let $(E,g_E,X)$ be an AC K\"ahler--Ricci expander with asymptotic cone $(\mathcal{C},g_{\mathcal{C}})$, let $f$ be a soliton potential satisfying $\nabla^{g_E}f=X$, and let $r$ be the radius function of the cone. We also take $\pi:E\to \mathcal{C}$ as in Theorem \ref{theorem of conlonderuellesun}. In this subsection, we recall some useful geometric properties of AC K\"ahler-Ricci expanders, beginning with some fundamental identities for gradient Ricci solitons (see \cite[Section 2 of Chapter 1]{ChowEtAl}).
  \begin{lemma}\label{soliton indentities} After adding a constant to $f$, the following identities hold:
    \begin{equation*}
        \begin{split}
            \Delta_{\omega_E} f=n+R_{\omega_E}, \qquad
            |\partial f|_{g_E}^2+R_{\omega_E}=f-n.
        \end{split}
    \end{equation*}
\end{lemma}
An importance consequence is that $f$ is an eigenfunction of the drift Laplacian $\Delta_{\omega_E,X}:=\Delta_{\omega_E}+\frac{1}{2}X$.
\begin{corollary}
    Let $f$ be the normalized soliton potential, then $f$ satisfies the following elliptic equation:
    \begin{equation}\label{f is sol to elliptic eq}
        \Delta_{\omega_E,X}f=f.
    \end{equation}
\end{corollary}
\begin{proof}
    Putting together $\nabla^{g_E}f=X$, $|\partial f|_{g_E}^2+R_{\omega_E}+n=f$ and $\Delta_{\omega_E} f=n+R_{\omega_E}$, we have that
    \begin{equation*}
        \Delta_{\omega_E,X}f=n+R_{\omega_E}+|\partial f|_{g_E}^2=f.
    \end{equation*}
\end{proof}

By Hopf’s maximum principle, we deduce that the normalized soliton potential is strictly bounded away from zero. Alternatively, this can be understood by examining the lower bound of the scalar curvature.
\begin{lemma}\label{lower bound of scalar curvature}
    There exists a constant $\varepsilon>0$ such that $R_{\omega_E}\ge-n+\varepsilon$ on $E$.
\end{lemma}
\begin{proof}
   See \cite[Corollary 2.5]{longteng1}.
\end{proof}
 Combining Lemma \ref{soliton indentities} and Lemma \ref{lower bound of scalar curvature} gives $f \ge \varepsilon$, where $\varepsilon>0$ is as in Lemma \ref{lower bound of scalar curvature}. 

We identify $\mathcal{C}\setminus\{o\}$ with the image of $\mathcal{C}\setminus\{o\}$ in $E$ via the biholomorphism $\pi^{-1}$. With this identification, the radial function $\frac{r^2}{2}$ of the K\"ahler cone arises as the limit of the soliton potential of $\omega_E(t)$ as $t\to 0$. Therefore, $r$ can be viewed as a continuous function on $E$. More precisely, we have the following comparison.
\begin{corollary}\label{coro of comparison of f and r^2}
Let $r$ denote the radial function of K\"ahler cone $(\mathcal{C},g_{\mathcal{C}})$. Let $\Phi_t$ be the flow of $-\frac{1}{2t}X$ for $t>0$.
    There exists a uniform constant $A>0$ such that for all $t>0$, on $\mathcal{C}\setminus\{o\}$,
    \begin{equation}\label{comparison of f and r^2}
       \frac{r^2}{2}\le t\Phi_t^*f\le \frac{r^2}{2}+At.
    \end{equation}
\end{corollary}
\begin{proof}
    On the one hand, by the soliton identities we have
    \begin{equation*}
        t\Phi_t^*f=\frac{1}{2}g_E(t)(X,X)+t(\Phi_t^*R_{\omega_E}+n).
    \end{equation*}
    Let $t\to 0^+$, since the scalar curvature is bounded, we have on $C\setminus\{o\}$
    \begin{equation*}
        \lim_{t\to 0^+}t\Phi_t^*f=\lim_{t\to 0^+}\frac{1}{2}g_E(t)(X,X)=\frac{r^2}{2}.
    \end{equation*}
    On the other hand, we compute
    \begin{equation*}
        \frac{d}{dt}(t\Phi_t^*f)=\Phi_t^*\left(f-\frac{1}{2}X\cdot f\right)=\Phi_t^*(R_{\omega_E}+n)
    \end{equation*}
    Since $R_{\omega_E}+n>0$ and let $A=\sup_E (R_{\omega_E}+n)$, we have that for all $0<s<t$
    \begin{equation*}
      s\Phi_s^*f\le  t\Phi_t^*f\le s\Phi_s^*f+A(t-s).
    \end{equation*}
    Let $s\to 0^+$, then \eqref{comparison of f and r^2} holds.
\end{proof}

\begin{definition} For any $D,D'\in (0,\infty)$, we define 
\begin{equation*}
    \mathcal{B}(D):= \mathcal{B}_E(D):=\pi^{-1}(\{x\in \mathcal{C}\: | \: r(x)<D\}), \qquad  \overline{\mathcal{B}}(D):= \overline{\mathcal{B}}_E(D):=\pi^{-1}(\{x\in \mathcal{C}\: | \: r(x)\leq D\}),
\end{equation*}
\begin{equation*}
    \overline{\mathcal{A}}(D,D') := \pi^{-1}(\{x \in \mathcal{C} \: | \: D\leq r(x) \leq D'\}).
\end{equation*}
\end{definition}

A last property that will be useful to us is the fact that, on the asymptotically conical gradient K\"ahler--Ricci expander $(E,g_E,X)$, the injectivity radius grows linearly. 
\begin{prop}\label{linear growth of injectivity radius}
    There exists a constant $\delta_0>0$ such that for all $x\in E$
    \begin{equation*}
        r_{\textnormal{inj}}^{g_E}(x)\ge \delta_0\sqrt{f(x)+1}.
    \end{equation*}
\end{prop}
\begin{proof}
   Fix $x\in E$.  By the Cheeger-Gromov-Taylor injectivity-radius estimate \cite[Theorem 4.7]{CheegerGromovTaylor}, applied after rescaling $(f(x)+1)^{-1}$, the uniform lower volume at scale $\sqrt{f(x)+1}$ implies
  \begin{equation*}
       r_{\textnormal{inj}}^{g_E}(x)\ge \delta_0\sqrt{f(x)+1},
  \end{equation*}
  for some uniform constant $\delta_0>0$.
\end{proof}

\section{The approximating solutions}\label{sectn; approx soltn}

Throughout this section, we fix a singular space $(Y,\omega_0)$ satisfying the hypotheses of Theorem \ref{main theorem}. We also fix open holomorphic embeddings 
\begin{equation*}
    \phi_i:B_{g_{\mathcal{C}(S_i)}}(o_i,r_0) \to Y
\end{equation*}
and relative K\"ahler potentials $u_i$ on $B_{g_{\mathcal{C}(S_i)}}(o_i,r_0)$ as in Definition \ref{defn of sing space}, so that $\phi_i^{\ast}\omega_0=\omega_{\mathcal{C}(S_i)}+\sqrt{-1}\partial \overline{\partial}u_i$, where Definition \ref{defn of sing space} \ref{defn:singspac3} holds.

\subsection{Construction of the approximating solutions}
In this subsection, we construct a family $(M,\omega_{s,0})_{s\in (0,s_0]}$ of smooth K\"ahler metrics which approximate $\omega_0$ as $s\to 0^+$. We begin by constructing the underlying complex manifold $M$, which will be a resolution of the complex space $Y$.

\begin{prop}\label{prop: existence of Kahler resolution}
    Let $(Y,g_0)$ be a compact K\"ahler space with isolated conical singularities modeled on K\"ahler cones with smooth canonical models. Then there exists a resolution of singularities $\pi: M\to Y$ such that $K_M$ is $\pi$-ample.
\end{prop}
\begin{proof}
After shrinking $r_0>0$, we may assume $\phi_i(B_{g_{\mathcal{C}(S_i)}}(o_i,r_0))\cap \phi_j(B_{g_{\mathcal{C}(S_j)}}(o_j,r_0))=\emptyset$ for every $i\neq j$. For each cone $\mathcal{C}(S_i)$, there exists a resolution $\pi_i:E_i\to \mathcal{C}(S_i)$. Now define
  \begin{equation*}
      M=\frac{\left(Y\setminus\{y_i\}_{i=1}^Q\right)\bigsqcup\left(\bigcup_{i=1}^Q \mathcal{B}_{E_i}(r_0)\right)}{\{y=(\phi_i\circ\pi_i)^{-1}(y)\ \textnormal{for some $i$}\}}.
  \end{equation*}
Then, $M$ is a smooth complex manifold. The maps $\operatorname{id}_{Y\setminus \{y_i\}_{i=1}^Q}$ and $\phi_i\circ \pi_i$ glue to produce the desired holomorphic map $\pi:M\to Y$. 
\end{proof}
From now on, we identify $M\setminus\cup_{i=1}^Q \{\pi^{-1}(y_i)\}$ with $Y\setminus\{y_i\}_{i=1}^Q$ via $\pi$. On $M,$ we define a continuous function $r_M$ (smooth on $\pi^{-1}(Y\setminus\{y_i\}_{i=1}^Q)$) by
\begin{equation}\label{radial function}
    r_M(x):=\begin{cases}

        r(x) \quad& \textnormal{ if } x \in \overline{\mathcal{B}}_{E_i}(r_0) \textnormal{ for some } i,\\
        r_0 \quad   &\: \textnormal{otherwise.}   \end{cases}
\end{equation}
We are now ready to define the family of approximating metrics $\omega_{s,0},$ for $s>0.$ Since the construction and the arguments are local, we only define this around one singular point $y_1\in Y,$ with all of our arguments extending easily to the case of multiple singular points. To further simplify our notation, we write $E_1=E$  and $\mathcal{C}:=\mathcal{C}(S_1)$. Let $\chi: \mathbb{R}^+\to [0,1]$ be a fixed, increasing cut-off function such that $\chi_{|_{[0,1]}}\equiv 0$ and $\chi_{|_{[2,\infty)}}\equiv 1.$
For any $j\in\N$, we then have
\begin{align} \label{eq:cutoffestimates}
    \left|(\nabla^{g_0})^j \left(\chi \left( \frac{r_M}{s^{\frac{1}{4}}} \right)\right)  \right| \leq C_js^{-\frac{j}{4}}.
\end{align}
Recalling $u_1$ from Definition \ref{defn of sing space}, we set $u_{1,s}:=u_1+s\log\frac{\omega_0^n}{\omega_{\mathcal{C}}^n}$.

In terms of the \textit{approximation parameter} $s>0$, we can then define

\begin{align}\label{defn approx metrics}
\omega_{s,0} := \omega_E(s) + \sqrt{-1}\partial \overline{\partial} \left(\chi\left(\frac{r_M(\cdot)}{s^{\tfrac{1}{4}}}\right) (u_{1,s} -u_E(s))\right),\end{align}
for $r_M\in [s^{\frac{1}{4}},2s^{\frac{1}{4}}]$, so that $\omega_{s,0}$ extends to a smooth closed $(1,1)$-form on $M$ satisfying $\omega_{s,0}=\omega_E(s)$ on $\{r_M \leq s^{\frac{1}{4}}\}$ and $\omega_{s,0}=\omega_0-s\Ric(\omega_0)$ on $\{r_M\geq 2s^{\frac{1}{4}}\}$. Here, $u_{E}(s)$ is defined as in Corollary \ref{decay of uE(s)}. The idea of including the Ricci form in the approximating metrics comes from the work of Conlon--Deruelle \cite{ConlonDeruelleJDG}.
Computations on the cohomology class give
\begin{equation}\label{eq: cohomological class of w_s,0}
    [\omega_{s,0}]=[\omega_0]-sc_1(M).
\end{equation}

\begin{remark}
    From the definition of $\omega_{s,0},$ we observe that 
    \[\lim_{s\to 0} \omega_{s,0}(x)=\omega_0(x)\] for all $x \in M\setminus \pi^{-1}(y_1).$ 
\end{remark}
The following proposition shows that, by fixing a sufficiently small upper bound $s_0>0$ for the approximation parameter $s$, one can ensure the positivity of $\omega_{s,0}$.
\begin{prop}\label{gluing metric; prop; 2}
    There exists $s_0>0$ such that for any $s\in (0,s_0]$, $\omega_{s,0}$ is a K\"ahler metric on $M\setminus\{r_M^2\le 4s^{1/2}\}$ satisfying
\begin{align}\label{eq: outside the glueing region}
\sup_{\{ r_M \geq 2s^{\frac{1}{4}}\}} r_M^k|(\nabla^{g_{0}})^k(g_{s,0}-g_0)|_{g_0} \leq C_k s^{1/2}.
\end{align}
\end{prop}
\begin{proof}
    It suffices to verify \eqref{eq: outside the glueing region} on $\{r_0\ge r_M\ge 2s^{1/4}\}$. On this region, we have
    \begin{equation*}
        r_M^{k+2}|(\nabla^{g_0})^k\Ric(g_0)|_{g_0}\le C_k.
    \end{equation*}
    Therefore,
     \begin{equation*}
       \sup_{\{ r_M \geq 2s^{\frac{1}{4}}\}} r_M^k|(\nabla^{g_{0}})^k(g_{s,0}-g_0)|_{g_0}= \sup_{\{ r_M \geq 2s^{\frac{1}{4}}\}} r_M^ks|(\nabla^{g_{0}})^k\Ric(g_0)|_{g_0}\leq C_k s^{1/2}.
    \end{equation*}
\end{proof}
\begin{prop}\label{gluing metric; prop}
 {There exists $s_0>0$ such that for any $s\in (0,s_0]$, $\omega_{s,0}$ is a K\"ahler metric satisfying
\begin{align} \label{eq:initialtimeannulusclose}
\sup_{\{s^{\frac{1}{4}} \leq r_M \leq 2s^{\frac{1}{4}}\}} r_M^k|(\nabla^{g_{\mathcal{C}}})^k(g_{s,0}-g_{\mathcal{C}})|_{g_{\mathcal{C}}} \leq C_k \left( \sum_{j=0}^{k+4} k_j(2s^{\frac{1}{4}})+s^{\frac{1}{2}} \right).
\end{align}}

\end{prop}

\begin{proof}
Since $u_{1,s}=u_1+s\log\frac{\omega_0^n}{\omega_{\mathcal{C}}^n}$, on $\{s^{\frac{1}{4}} \leq r_M \leq 2s^{\frac{1}{4}}\}$, for all $j\in\N$, we have
\begin{equation}\label{eq: decay of u1s}
   \begin{split}
        r_M^{j-2}|(\nabla^{g_{\mathcal{C}}})^j u_{1,s}|_{g_{\mathcal{C}}}&\le r_M^{j-2}|(\nabla^{g_{\mathcal{C}}})^j u_{1}|_{g_{\mathcal{C}}}+r_M^{j-2}s|(\nabla^{g_{\mathcal{C}}})^j \log\frac{\omega_0^n}{\omega_{\mathcal{C}}^n}|_{g_{\mathcal{C}}}\\
        &\le  k_j(2s^{\frac{1}{4}})+ s^{\frac{1}{2}}k_{j+2}(2s^{\frac{1}{4}}).
   \end{split}
\end{equation}

We can write
\begin{align*}
    \omega_{s,0}-\omega_{\mathcal{C}} =\omega_E(s)-\omega_{\mathcal{C}}+ \sqrt{-1}\partial \overline{\partial} \left( \chi ( \frac{r_M}{s^{\frac{1}{4}}} ) u_{1,s} -\chi ( \frac{r_M}{s^{\frac{1}{4}}} ) u_E(s)\right),
\end{align*}
so we may combine \eqref{eq:estimatesforu1},Corollary \ref{decay of uE(s)}, \eqref{eq: decay of u1s} and \eqref{eq:cutoffestimates} to obtain
\begin{align*}
    |g_{s,0}-g_{\mathcal{C}}|_{g_{\mathcal{C}}} \leq &|(\nabla^{g_{\mathcal{C}}})^2 u_{1,s}|_{g_{\mathcal{C}}}+\frac{C}{s^{\frac{1}{4}}}|\nabla^{g_{\mathcal{C}}}u_{1,s}|_{g_{\mathcal{C}}}+\frac{C}{s^{\frac{1}{2}}}
|u_{1,s}| +Csr_M^{-2}\\&+ |(\nabla^{g_{\mathcal{C}}})^2 u_E(s)|_{g_{\mathcal{C}}} + \frac{C}{s^{\frac{1}{4}}}|\nabla^{g_{\mathcal{C}}}u_E(s)|_{g_{\mathcal{C}}} + \frac{C}{s^{\frac{1}{2}}}|u_E(s)| \\
\leq &k_2(2s^{\frac{1}{4}})+s^{\frac{1}{2}}k_4(2s^{\frac{1}{4}})+ Ck_1(2s^{\frac{1}{4}})+Cs^{\frac{1}{2}}k_{3}(2s^{\frac{1}{4}})+Ck_0(2s^{\frac{1}{4}})+Cs^{\frac{1}{2}}k_2(2s^{\frac{1}{4}})\\
&+C_2 s^{\frac{1}{2}} + C_1 s^{\frac{1}{2}} + C_0 s^{\frac{1}{2}}.
\end{align*}
In particular, if $s_0>0$ is sufficiently small, then $g_{s,0}$ is a K\"ahler metric, and \eqref{eq:initialtimeannulusclose} holds for $k=0$. The case $k>0$ follows from similar considerations.
\end{proof}
From now on, we consider the approximation parameter $s\le s_0$ and introduce a new parameter $R>0$ which localizes our estimates. 
\begin{definition}
    We define the \textit{localization parameter} $R>0$ satisfying $R^2> 4\sqrt{s}$ and the \textit{conical region} 
    \begin{equation*}
    \{ s^{\frac{1}{4}}\leq r_M \leq R\}.
    \end{equation*}
\end{definition}

\begin{remark}\label{remark on decays of approximation metrics}
By similar computations, we can also show that there exists $R_0>0$ such that for all $R \le R_0$ and $s \le s_0$ satisfying $R^2 > 4\sqrt{s}$ we have the following. For each $k \in \N$, there exists a constant $A_k>0$, depending on $s_0$ and $R_0$, such that
\begin{equation*}
     \bigl|(\nabla^{g_{\mathcal{C}}})^k (g_{s,0} - g_{\mathcal{C}})\bigr|_{g_{\mathcal{C}}} \le A_kr_M^{-k}
\end{equation*}
holds on the \textit{conical region}. In particular, for all $\varepsilon>0$, we can find $s_0>0$ such that for all $s\le s_0,R\le \frac{1}{4}R_0$ satisfying $R^2>4\sqrt{s}$,  we have
 \begin{equation*}
    |g_{s,0}-g_{\mathcal{C}}|_{g_{\mathcal{C}}}\le \varepsilon
\end{equation*}
on the region $\{16R^2\ge r_M^2\ge\frac{1}{4}\sqrt{s}\}$.
Therefore, by taking $\varepsilon \ll 1$ so that $g_{s,0}$ and $g_{\mathcal{C}}$ are bi-Lipschitz equivalent on the conical region, we can find $s_0,R_0>0$ satisfying the relations above such that for all $s\le s_0$ we have
\begin{equation*}
     \bigl|(\nabla^{g_{s,0}})^k \Rm(g_{s,0})\bigr|_{g_{s,0}} \le C_kr_M^{-2-k}
\end{equation*}
on the conical region, where $C_k=C_k(s_0,R_0)>0$ is as above. This essentially follows from a direct application of Perelman's pseudolocality; for a detailed proof, we refer to \cite[Lemma 3.1]{Gianniotis-Schulze}, \cite[Theorem 3.2.3]{Siepmann}, \cite[Lemma 5.3.2]{ToppingBook} and \cite[Lemma A.4]{Topping}.
\end{remark}

We can then consider a solution to the K\"ahler--Ricci flow $\omega_s(t),$ with $t\in [0,T_s)$ and $\omega_s(0)=\omega_{s,0}.$ We define $T_s>0$ to be the \textit{maximal existence time} of the flow, and recall that, by the work of Hamilton, if $T_s\neq\infty$, then $T_s$ can be characterised by
 \begin{equation*}
     \limsup_{t\to T_s} \sup_{M}|\Rm(g_s(t))|_{g_s(t)}=+\infty,
 \end{equation*}
where $g_s(t)$ is the Riemannian metric with respect to $\omega_s(t)$.

\subsection{Uniform estimates on the conical region}

\label{subsection:estimatesconeregion}

The main difficulty of the approach laid out in the introduction is obtaining estimates for the flow $\omega_s(t)$ in a region $\displaystyle{\{r_M^2\le R^2\}\times [0,T_s)}$ near the singularities which are uniform in $s$.
Since these estimates are local, we work only around $y_1$. Let $\mathcal{C}$ (resp. $(E,g_E,X)$) denote the corresponding K\"ahler cone (resp. K\"ahler--Ricci expander). When the initial data is close enough to the conical metric, Perelman's pseudolocality together with Shi's estimates will control the flow.

 \begin{prop}\label{estimate on the conical region}
There exist constants $R_0, s_0, \lambda_0 > 0$ and $C_k \in (0,\infty)$ such that for any $s\in (0,s_0]$, on the region
 \begin{equation*}
\{(x,t) \in M\times [0,T_s )\ |\  R_0^2\ge r_M^2(x)\ge \lambda_0 t\},
    \end{equation*}
the following estimates hold:
    \begin{equation} \label{eq:consequenceofpseudolocality}
        |(\nabla^{g_s(t)})^{k}\Rm(g_s(t))|_{g_s(t)}(x)\le C_k r_M(x)^{-2-k}.
    \end{equation}  
\end{prop}
\begin{proof}
    The details of this standard argument are given in \cite[Lemma 3.1]{Gianniotis-Schulze}; see also \cite[Theorem 3.2.3]{Siepmann}, \cite[Lemma 5.3.2]{ToppingBook}, and \cite[Lemma A.4]{Topping}.
\end{proof}
As mentioned in the introduction, since we do not know if the expander $(E,g_E,X)$ is weakly stable in the sense of Deruelle--Lamm \cite{Deruelle-Lamm}, we need a different approach to obtain good estimates near the singular point. To introduce our approach, we start by reducing the Ricci flow equation to a complex Monge--Amp\`ere equation. 
    \begin{prop}
        There exists a smooth function $\varphi_s(t)$ with $t\in [0,T_s)$ which is defined on $\{r_M^2 \leq R^2\}$ such that $\omega_{s}(t)=\omega_{E}(t+s)+\sqrt{-1}\partial\bar\partial \varphi_s(t)$ and
        \begin{equation*}
            \frac{\partial}{\partial t}\varphi_s(t)=\log\frac{\omega_s(t)^n}{\omega_E({t+s})^n}.
        \end{equation*}
        Moreover, $\varphi_s(0)=\chi\left(\frac{r_M(\cdot)}{s^{\frac{1}{4}}}\right) (u_{1,s} -u_E(s))=:\psi_{s,0}$
    \end{prop}
    \begin{proof}
Define $$\varphi_s(t):= \varphi_s(0)+\int_0^t \log\frac{\omega_s(\tau)^n}{\omega_E({\tau+s})^n} d\tau,$$ and note that $\omega_s(0)=\omega_E(s)+\sqrt{-1}\partial \overline{\partial} \varphi_s(0)$ and
\begin{equation*}
    \frac{\partial}{\partial t} \left( \omega_s(t)-\omega_E(t+s)-\sqrt{-1}\partial \overline{\partial} \varphi_s(t) \right) =0.
\end{equation*}
    Then $\omega_{s}(t):=\omega_{E}(t+s)+\sqrt{-1}\partial\bar\partial \varphi_s(t)$ holds on $\{r_M\le R\}\times [0,T_s)$.
    \end{proof}
  For sufficiently small $R$, the region $\{r_M\le R\}$ may be identified with the region $\mathcal{B}(R)$ in the expander $E$. Recall that $\Phi_t^*r^2=\frac{r^2}{t}$ for all $t>0,$ where $\Phi_t$ is the flow of $-\frac{X}{2t}$. We now normalize this solution to the complex Monge--Amp\`ere equation to let the small scale $s=1$. To do so, we consider the biholomorphism 
    \begin{equation*}
        \Phi_{\frac{1}{s}}: \mathcal{B}(\frac{R}{\sqrt{s}})\to \mathcal{B}(R).
    \end{equation*}

    \begin{definition}\label{definition of MA equation}
        The normalized potentials and metrics are defined by \begin{equation*} \overline\varphi_s(t):=\frac{1}{s}\Phi_{\frac{1}{s}}^*\varphi_s(ts), \qquad\overline \omega_s(t):=\omega_E(1+t)+\sqrt{-1}\partial\bar\partial \overline \varphi_s(t)\end{equation*} 
        on $\mathcal{B}(\frac{R}{\sqrt{s}})\times [0,\frac{T_s}{s})$. 
    \end{definition}

    \begin{prop} The normalized potentials and metrics satisfy
        \begin{equation}\label{equation of complex monge ampere}
            \frac{\partial}{\partial t}\overline{\varphi}_s(t)=\log\frac{\overline\omega_s(t)^n}{\omega_E(1+t)^n}, \qquad \partial_t \overline{\omega}_s(t)=-\Ric_{\overline{\omega}_s(t)}
        \end{equation}
    on $\mathcal{B}(\frac{R}{\sqrt{s}})\times [0,\frac{T_s}{s})$. On $\{(x,t)\in M\times [0,\frac{T_s}{s}) \: | \: \frac{R^2}{s} \geq r^2(x) \geq \lambda t\},$ we have:
              \begin{equation} \label{eq:normalizedpseudolocality}
           |(\nabla^{\overline g_s(t)})^k\Rm(\overline g_s(t))|_{\overline g_s(t)}(x)\le \frac{C_k}{r(x)^{2+k}},\quad \textrm{for all $k\in\N$},
        \end{equation}
        where $C_k$ is the same constant as in Theorem \ref{estimate on the conical region}.
    \end{prop}
   \begin{proof} 
     By \eqref{eq:consequenceofpseudolocality} and the correspondence $\overline g_s(t)=\frac{1}{s}\Phi_{\frac{1}{s}}^*g_s(ts)$, we have
    \begin{equation*}
       | (\nabla^{\overline g_s(t)})^k\Rm(\overline g_s(t))|_{\overline g_s(t)}\le\frac{C_k}{r^{2+k}},
    \end{equation*}
    holds for all $(x,t)$ such that {$r^2(\Phi_{\frac{1}{s}}(x)) \geq \lambda ts$; that is, $r^2(x)\geq \lambda t$.}
    \end{proof}
     
    We observe that, after this normalisation, the reference metric $\omega_E(t+s)$ becomes $\omega_E(1+t)$. Figure \ref{fig:1} illustrates the un-normalized space-time and normalized space-time under consideration. By applying pseudolocality along the Ricci flow, we obtain curvature estimates for the conical region (resp. normalized conical region).
      \begin{figure}[htbp]
\centering
\begin{minipage}{1.0\textwidth}
    \centering
    \resizebox{\linewidth}{!}{\input{normalised_spacetime}}
      \caption{un-normalized and normalized space-time}
    \label{fig:1}
\end{minipage}
\end{figure}
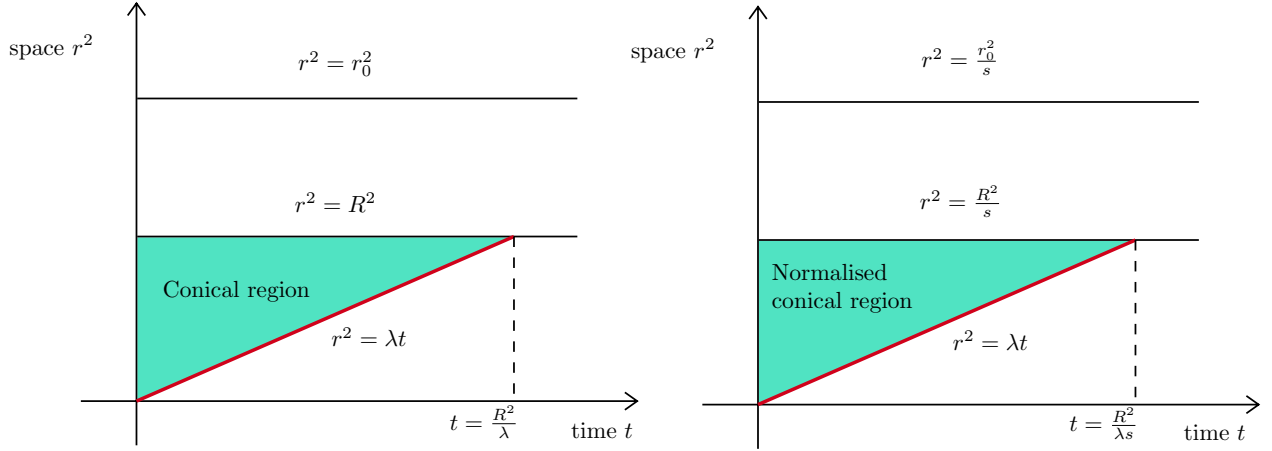

We also define $u_1(s):=\frac{1}{s}\Phi_{\frac{1}{s}}^*u_{1,s}.$ From the decay of $u_{1,s},$ on $\{\frac{1}{\sqrt{s}}\le r^2\le\frac{R^2}{s}\}$ we have:
    \begin{equation*}
       \begin{split}
            r^{j-2}|(\nabla^{g_{\mathcal{C}}})^ju_1(s)|_{g_{\mathcal{C}}}&\le k_j(r\sqrt{s})+k_{j+2}(r\sqrt{s})r^{-2}\\
            &\le k_j(r\sqrt{s})+k_{j+2}(r\sqrt{s})\sqrt{s}\le k_j(R),\quad\textnormal{for all $j\in\N$}.
       \end{split}
    \end{equation*}
    Here, $k_j(R)$ denotes a positive function which tends to 0 as $R$ goes to 0.
The initial K\"ahler potential $\overline{\varphi}_s(0)$ on $\mathcal{B}(\frac{R}{\sqrt{s}})$ is given by:
\begin{equation*}
    \overline\varphi_s(0)=\frac{1}{s}\Phi_{\frac{1}{s}}^*\varphi_s(0)=\chi\left(r(\cdot)s^\frac{1}{4}\right)(u_1(s) -u_E).
\end{equation*}
We can thus estimate the $j-$th derivative of $\overline{\varphi}_s(0)$ on $\{\sqrt{s}^{-1}\le r^2\le 4\sqrt{s}^{-1}\}$ for all $j\in\N$ by

\begin{equation*}
    r^{j-2}|(\nabla^{g_{\mathcal{C}}}
)^j \overline{\varphi}_s(0)|_{g_{\mathcal{C}}} \leq \sum_{\ell=0}^{j+2} k_{\ell}(2s^{\frac{1}{4}})+\sum_{\ell=0}^j C_{\ell}\sqrt{s} \leq \widetilde{k}_j(s),\end{equation*}
where $\lim_{s \to 0}\widetilde{k}_j(s)=0$.
On $\{\frac{R^2}{s}\ge r^2\ge 4\sqrt{s}^{-1}\}$, we can estimate the $j-$th derivative by
\begin{equation*}
     r^{j-2}|(\nabla^{g_{\mathcal{C}}}
)^j \overline{\varphi}_s(0)|_{g_{\mathcal{C}}} \leq k_j(R)+C_j\frac{1}{r^2}\leq k_j(R)+C_j\sqrt{s}.
\end{equation*}
Combining these expressions and using that $\sqrt{s}\le R^2$, we conclude on $\{r^2\le\frac{R^2}{s}\}$, we have
\begin{equation}\label{control of kahler potential at time 0}
     r^{j-2}|(\nabla^{g_{\mathcal{C}}}
)^j \overline{\varphi}_s(0)|_{g_{\mathcal{C}}} \le k_j(R),
\end{equation}
holds for some positive function $k_j(R)$ with $\lim_{R\to 0^+}k_j(R)=0$.

Let $R_0,\lambda_0,s_0$ be the constants from Proposition \ref{estimate on the conical region}. For the remainder of this subsection, we will assume 
\begin{equation} \label{eq:backgroundassumption}
    R\in (0,R_0], \qquad \lambda \in [\lambda_0,\infty), \qquad s\in \left(0, \min\left\{ \frac{R^4}{16},\frac{1}{\lambda^2},s_0 \right\} \right).
\end{equation}

 \begin{definition}[Normalized conical region]\label{conical region}
 The normalized conical region $\mathcal{C}_{R,\lambda,s}$ is defined to be
     \begin{equation*}
         \mathcal{C}_{R,\lambda,s}:=\left\{(x,t)\ |\ \lambda t\le r(x)^2\le \frac{R^2}{s},\ t\in [0,\frac{T_s}{s})\right\} .
     \end{equation*}

 \end{definition}

Our goal for the remainder of the section is to establish estimates for geometric quantities on the normalized conical region $\mathcal{C}_{R,\lambda,s}$ using \eqref{control of kahler potential at time 0} and Proposition \ref{eq:consequenceofpseudolocality}.
 
\begin{prop}\label{rough estimate on the metrics}
   For all $(x,t)\in \mathcal{C}_{R,\lambda,s}$, we have
    \begin{equation*}
        |\overline g_s(t)-g_E(1+t)|_{g_E(1+t)}(x)\le \Psi(R,\lambda^{-1}).
    \end{equation*}
\end{prop}
\begin{proof}
   By \eqref{eq:normalizedpseudolocality},
    \begin{equation*}
       e^{-\frac{Ct}{r(x)^2}}\overline{g}_s(x,0)\le \overline g_s(x,t)\le e^{\frac{Ct}{r(x)^2}}\overline{g}_s(x,0).
    \end{equation*}
    A similar argument using \eqref{eq:quadraticdecayofexpander} and $g_E(0)=g_{\mathcal{C}}$ yields $C'>0$ such that
     \begin{equation*}
     \begin{split}
       e^{-\frac{C't}{r(x)^2}}g_E(x)&\le g_E(x,1+t)\le e^{\frac{C't}{r(x)^2}}g_E(x).\\
     \end{split}
    \end{equation*}
 If $r(x) <s^{-\frac{1}{4}}$, then $\overline{g}_s(x,0)=g_E(x,1)$. If instead $r(x)  \geq s^{-\frac{1}{4}}$, then Corollary \ref{decay of uE(s)} gives
\begin{align*}
    |g_E(x)-g_{\mathcal{C}}(x)|_{g_{\mathcal{C}}(x)} \leq C_2\sqrt{s},
\end{align*}
which combines with \eqref{control of kahler potential at time 0} to give
\begin{align*}
    |\overline{g}_s(x,0)-g_E(x)|_{g_{E}(x)} \leq k_2(R)
\end{align*}
if $R>0$ is sufficiently small. Recall that we take $r(x)^2\ge\lambda t$. Combining expressions yields the claim.
\end{proof}
\begin{prop}\label{rough C^3 estimate}
     For all $(x,t)\in \mathcal{C}_{R,\lambda,s}$, we have
    \begin{equation*}
        \sum_{k=1}^2r^k(x)|(\nabla^{g_{E}(1+t)})^k\overline g_s|_{g_{E}(1+t)}(x,t)\le \Psi(R,\lambda^{-1}).
    \end{equation*}
    
\end{prop}
\begin{proof} Given \eqref{control of kahler potential at time 0}, the proof follows by the same argument as in the proof of \cite[Claim 3.9]{ConlonDeruelleSun}. 
\end{proof}

\begin{prop}\label{rough estimates on Kahler potential}
   For all $(x,t)\in \mathcal{C}_{R,\lambda,s}$, we have
    \begin{equation*}
          \sum_{k=0}^2r^{k-2}|(\nabla^{g_{E}(1+t)})^k\overline{\varphi}_s|_{g_{E}(1+t)}(x,t) \leq \Psi(R,\lambda^{-1}).
    \end{equation*}
\end{prop}

\begin{proof}
    Recall the complex Monge--Amp\`ere equation:
    \begin{equation*}
        \frac{\partial}{\partial t}\overline\varphi_s(t)=\log\frac{\overline\omega_s^n(t)}{\omega_E^n(1+t)}
    \end{equation*}
    Since $(1-\Psi(R,\lambda^{-1}))g_E(1+\rho)\le\overline g_s(\rho)\le (1+\Psi(R,\lambda^{-1}))g_E(1+\rho)$ holds on $\mathcal{C}_{R,\lambda,s}$ for all $\rho\in[0,t]$, we conclude that for all $\rho\in [0,t]$,
    \begin{equation*}
       \left|\log\frac{\overline\omega_s^n(\rho)}{\omega_E^n(1+\rho)}\right|\le \Psi(R,\lambda^{-1}).
    \end{equation*}
    By integration, we have $|\overline \varphi_s(x,t)- \overline \varphi_s(x,0)|\le\Psi(R,\lambda^{-1})t$ for all $(x,t)\in \mathcal{C}_{R,\lambda,s}$. By \eqref{control of kahler potential at time 0}, we know $|\overline \varphi_s(x,0)|\le k_0(R)r^2$. Replacing $\Psi(R,\lambda^{-1})$ with $\Psi(R,\lambda^{-1})+2k_0(R)$ and recalling that on $\mathcal{C}_{R,\lambda,s}$ we have $t\leq r^2(x),$ we get
    \begin{equation*}
        |\overline \varphi_s(x,t)|\le \Psi(R,\lambda^{-1})\left(\frac{r^2}{2}+t\right)\leq \Psi(R,\lambda^{-1})r^2.
    \end{equation*}
For the bound on $\nabla^{g_{E}(1+t)} \overline{\varphi}_s(t),$ we consider its evolution equation:
\begin{equation*}
  \begin{split}
        \partial_t \left( \nabla^{g_{E}(1+t)} \overline{\varphi}_s(t)\right) &= \nabla^{g_{E}(1+t)} \log \frac{\overline{\omega}_s^n(t)}{\omega_E^n(1+t)}+\Ric(g_E(1+t))* \nabla^{g_{E}(1+t)} \overline{\varphi}_s(t)\\
        &= \operatorname{tr}_{\overline{g}_s(t)}(\nabla^{g_{E}(1+t)}\overline{g}_s(t)) +\Ric(g_E(1+t))* \nabla^{g_{E}(1+t)} \overline{\varphi}_s(t),
  \end{split}
\end{equation*}
then 
\begin{equation*}
   \begin{split}
        \partial_t |\nabla^{g_{E}(1+t)} \overline{\varphi}_s(t)|_{g_{E}(1+t)}& \leq |\overline{g}_s^{-1}(t)|_{g_{E}(1+t)}|\nabla^{g_{E}(1+t)}\overline{g}_s(t)|_{g_{E}(1+t)}\\
        &\quad +C(n)|\Ric(g_E(1+t))|_{g_E(1+t)}|\nabla^{g_{E}(1+t)} \overline{\varphi}_s(t)|_{g_{E}(1+t)}\\
        &\leq \frac{\Psi(R,\lambda^{-1})}{r}+\frac{C}{r^2}|\nabla^{g_{E}(1+t)} \overline{\varphi}_s(t)|_{g_{E}(1+t)}.
   \end{split}
\end{equation*}
where in the last inequality we used Proposition \ref{rough C^3 estimate}. Integrating the above and using again that $t\leq r^2,$ we obtain the result for $k=1.$ The case $k=2$ follows in an analogous way.
\end{proof}
We now introduce the following function, which is a modification of the soliton potential that will be more suitable for our estimates.
\begin{definition}\label{definition of fvarphi}
Let $f$ be the normalized soliton potential of $X$ as in Lemma \ref{soliton indentities}. Let $\Phi_t$ be the flow of $-\frac{X}{2t}$ for all $t>0$.
    We define for all $t\in [0,\frac{T_s}{s})$
    \begin{equation*}
        \begin{split}
        &f_{1+t}=(1+t)f\circ\Phi_{1+t}\\
            &f_{\overline \varphi_s}(t)=f_{1+t}+\frac{X}{2}\cdot\overline\varphi_{s}(t).
        \end{split}
    \end{equation*}
\end{definition}
\begin{corollary}\label{rough estimates on X derivative}
      For all $(x,t)\in \mathcal{C}_{R,\lambda,s}$, we have
    
     \begin{equation*}
       \begin{split}
       & |X\cdot \overline \varphi_s(x,t)|\le \Psi(R,\lambda^{-1}){f_{1+t}}(x),\\
       &   |\nabla^{\overline g_s}(X\cdot\overline{\varphi}_s)|^2_{\overline g_s}(x,t)\le \Psi(R,\lambda^{-1})f_{1+t}(x),\\
            &|JX\cdot \overline \varphi_s(x,t)|\le \Psi(R,\lambda^{-1})f_{1+t}(x),\\
            &|\nabla^{\overline g_s(t)}(JX\cdot\overline{\varphi}_s)|^2_{\overline g_s}(x,t)\le \Psi(R,\lambda^{-1})f_{1+t}(x).
       \end{split}
       \end{equation*}
\end{corollary}
\begin{proof}
From $|\partial f_{1+t}|_{g_{E}(1+t)}\le\sqrt{f_{1+t}}$ and
\begin{equation*}
    |X\cdot\overline{\varphi}_s(t)|\leq |X|_{g_{E}(1+t)}|\nabla^{g_E(1+t)} \overline{\varphi}_s(t)|_{g_E(1+t)},
\end{equation*}
we can apply Proposition \ref{rough estimates on Kahler potential} with $k=1$ together with \eqref{comparison of f and r^2} and this yields the claim. The control of $|JX\cdot \overline \varphi_s|$ follows analogously. To control $\nabla^{\overline g_s}(X\cdot\varphi_s)$, we consider
\begin{equation*}
   \nabla^{g_E(1+t)}(X\cdot\overline \varphi_s)=\nabla^{g_E(1+t)}X*\nabla^{g_E(1+t)}\overline \varphi_s+X*(\nabla^{g_E(1+t)})^2\overline \varphi_s.
\end{equation*}
Then we can apply soliton equation, Proposition \ref{rough estimate on the metrics} and Proposition \ref{rough estimates on Kahler potential} with $k=1,2$ to get the results as required. We can control $|\nabla^{\overline g_s}(JX\cdot\varphi_s)|^2_{\overline g_s}$ in the same way.
\end{proof}

\begin{prop}\label{rough estimates on hamiltonian}
     For all $(x,t)\in \mathcal{C}_{R,\lambda,s}$, we have
    \begin{equation*}
     \begin{split}
         &|f_{\overline \varphi_s}(x,t)-f_{1+t}(x)|\le \Psi(R,\lambda^{-1})f_{1+t}(x)\\
         &|\nabla^{\overline g_s}f_{\overline \varphi_s}|^2_{\overline g_s}(x,t)\le  (2+\Psi(R,\lambda^{-1}))f_{\overline \varphi_s}(x,t).
     \end{split}
    \end{equation*}
\end{prop}
\begin{proof}
    Since $|f_{\overline \varphi_s}(t)-f_{1+t}|=|\frac{X}{2}\cdot\overline\varphi_{s}(t)|$, the first inequality follows directly from Corollary \ref{rough estimates on X derivative}. Now we estimate $|\nabla^{\overline g_s}f_{\overline \varphi_s}|^2_{\overline g_s}.$ Since \begin{equation*}
       \begin{split}
            |\nabla^{\overline g_s}f_{\overline \varphi_s}|_{\overline g_s}(x,t)&\le|\nabla^{\overline g_s(t)}f_{1+t}|_{\overline g_s(t)}(x)+|\nabla^{\overline g_s}(X\cdot\overline\varphi_s)|_{\overline g_s}(x,t)\\
            &\le |df_{1+t}|_{\overline g_s(t)}+\Psi(R,\lambda^{-1})\sqrt{f_{1+t}}\\
            &\le (1+\Psi(R,\lambda^{-1}))|df_{1+t}|_{g_E(1+t)}(x)+\Psi(R,\lambda^{-1})\sqrt{f_{1+t}}.
       \end{split}
    \end{equation*}
    Recalling that $|\partial f|_{g_E}^2\le|\partial f|_{g_E}^2+R_{\omega_E}+n= f$, we have
    \begin{equation*}
        |df_{1+t}|^2_{g_E(1+t)}\le 2f_{1+t}.
    \end{equation*}
    Since we also have $|f_{\overline \varphi_s}(t)-f_{1+t}|(x)\le \Psi(R,\lambda^{-1})f_{1+t}$,
    we conclude that
    \begin{equation*}
        |\nabla^{\overline g_s}f_{\overline \varphi_s}|^2_{\overline g_s}(x,t)\le (2+\Psi(R,\lambda^{-1}))f_{\overline\varphi_s}(x,t).
    \end{equation*}
\end{proof}

\section{Uniform estimates on the expander region}\label{sectn; uni estim on exp reg}
In this section, we establish uniform estimates for the complex Monge--Ampère equation corresponding to the modified K\"ahler--Ricci flow (see Definition \ref{definition of normalized complex monge ampere}) on the expander region defined below (see also figure \ref{normalized space time}). Let $R_0, s_0, \lambda_0 > 0$ be as in Proposition \ref{estimate on the conical region}. Throughout this section, we always assume that $R,\lambda,s$ satisfy \eqref{eq:backgroundassumption}.

\subsection{The modified flow}

Now we consider the following \textit{modified K\"ahler--Ricci flow} equation
    \begin{equation} \label{eq:modifiedKRF}
    \frac{\partial}{\partial\tau}\omega_{\psi_s}(\tau)=\mathcal{L}_{\frac{X}{2}}\omega_{\psi_s}(\tau)-\Ric(\omega_{\psi_s}(\tau))-\omega_{\psi_s}(\tau).
\end{equation}
The reason for considering \eqref{eq:modifiedKRF} is that it allows us to view the soliton metric $g_E$ as a \textit{fixed} reference metric. To relate this to previously considered solutions, we define the following space-time correspondence. For any $t\ge0$, let $\tau=\log(t+1)\ge0$ and consider the following biholomorphism:
\begin{equation*}
    \Phi_{e^{-\tau}}:\mathcal{B}(\frac{R}{\sqrt{s}e^{\frac{\tau}{2}}})\to \mathcal{B}(\frac{R}{\sqrt{s}}).
\end{equation*}
    \begin{definition}\label{definition of normalized complex monge ampere}
        The modified potentials and metrics are defined by
        \begin{equation*}
            \psi_s(\tau)=e^{-\tau}\Phi_{e^{-\tau}}^*\overline{\varphi}_s(e^\tau-1), \qquad \omega_{\psi_s}(\tau):=\omega_E+\sqrt{-1}\partial\bar\partial \psi_s(\tau)
        \end{equation*}
        on $\cup_{\tau \in [0,\log(1+\frac{T_s}{s}))}(\mathcal{B}(\frac{R}{\sqrt{s}e^{\frac{\tau}{2}}}) \times \{\tau\}).$ 
    \end{definition}
        
    \begin{prop}
        The modified potentials and metrics satisfy \eqref{eq:modifiedKRF}, and 
        \begin{equation}\label{equation of normalized complex monge ampere}
            \frac{\partial}{\partial\tau}\psi_s(\tau)=\log\frac{\omega^n_{\psi_s}(\tau)}{\omega^n_E}+\frac{X}{2}\psi_s(\tau)-\psi_s(\tau)
        \end{equation}
    and \eqref{eq:modifiedKRF} on $\cup_{\tau \in [0,\log(1+\frac{T_s}{s}))}(\mathcal{B}(\frac{R}{\sqrt{s}e^{\frac{\tau}{2}}}) \times \{\tau\}).$ On $ \{ (x,\tau)\in M\times [0,\log(\frac{T_s}{s}+1)) \: | \: \frac{R^2}{se^{\tau}} \geq r^2(x) \geq \frac{\lambda(e^\tau-1)}{e^\tau}\}$, one has
         \begin{equation}
           |(\nabla^{g_{\psi_s}(\tau)})^k\Rm(g_{\psi_s}(\tau))|_{ g_{\psi_s}(\tau)}(x)\le \frac{C_k}{r(x)^{2+k}},\quad \textrm{for all $k\in\N$},
        \end{equation}
        where $C_k$ is the same constant as in Theorem \ref{estimate on the conical region}.
    \end{prop}
\begin{proof}
 Since $\psi_s(\tau):=e^{-\tau}\Phi_{e^{-\tau}}^*\overline\varphi_s(e^\tau-1)$, the potential function $\psi_s$ is defined for $(x,\tau)$ such that $(\Phi_{e^{-\tau}}(x),e^\tau-1)\in \{r^2\le\frac{R^2}{s}\}\times[0,\frac{T_s}{s})$. We then need the following restriction:
 \begin{equation*}
     r(\Phi_{e^{-\tau}}(x))^2\le\frac{R^2}{s}\quad\textnormal{and }\quad 0\le e^{\tau}-1<\frac{T_s}{s},
 \end{equation*}  
 that is,
 \begin{equation*}
      r(x)^2\le\frac{R^2}{se^\tau}\quad\textnormal{and }\quad 0\le \tau<\log\left(\frac{T_s}{s}+1\right).
 \end{equation*}
 Since 
    \begin{equation*}
           | (\nabla^{\overline g_s(t)})^k\Rm(\overline g_s(t))|_{\overline g_s(t)}\le\frac{C_k}{r^{2+k}}
        \end{equation*}
        holds for $(x,t)$ satisfying $r(x)^2\ge \lambda t$, we have that
        \begin{equation*}
            |(\nabla^{g_{\psi_s}(\tau)})^k\Rm(g_{\psi_s}(\tau))|_{ g_{\psi_s}(\tau)}(x)\le \frac{C_k}{r(x)^{2+k}}
        \end{equation*}
        holds for all $(x,\tau)$ such that $r(\Phi_{e^{-\tau}}(x))^2\ge \lambda(e^\tau-1)$, that is, \(r(x)^2\ge\frac{\lambda(e^\tau-1)}{e^\tau}.\)
        \end{proof}
       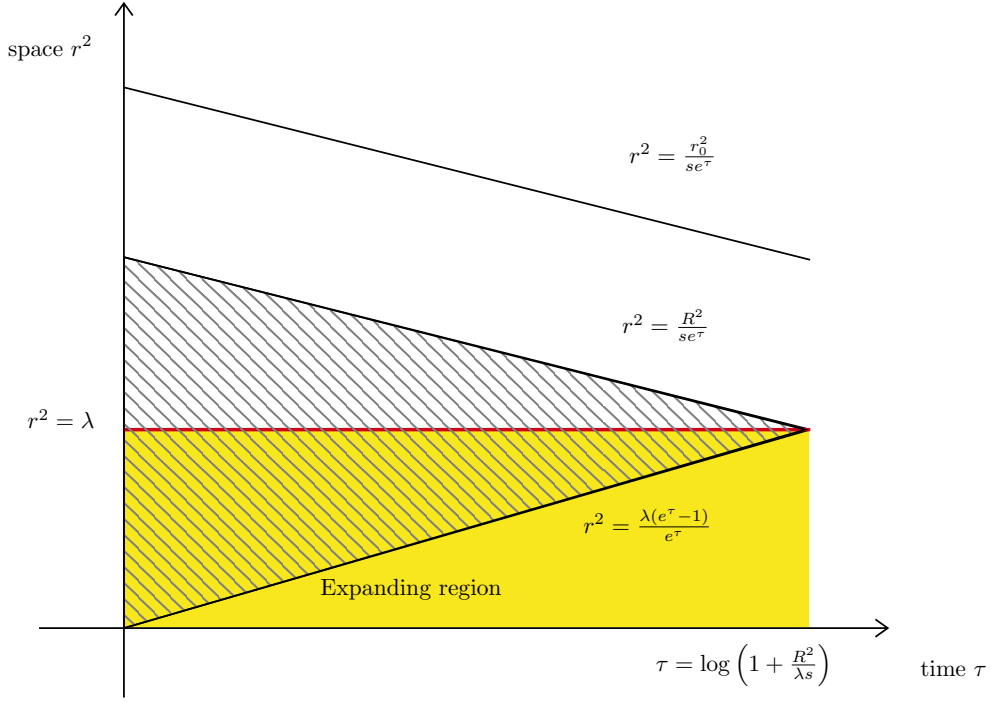
\begin{figure}[htbp]
\centering
\begin{minipage}{0.8\textwidth}
    \centering
    \resizebox{\linewidth}{!}{\input{modified_spacetime}}
      \caption{modified space-time}
    \label{normalized space time}
\end{minipage}
\end{figure}

Let $\psi_s(\tau)$ with $\tau\in [0,\log(\frac{T_s}{s}+1))$ denote the solution of the modified complex Monge--Ampère equation (see Definition~\ref{definition of normalized complex monge ampere}). In particular, $\omega_{\psi_s}(\tau)$ is a modified K\"ahler-Ricci flow satisfying
\begin{equation}
       \frac{\partial}{\partial\tau}\omega_{\psi_s}(\tau)=\mathcal{L}_{\frac{X}{2}}\omega_{\psi_s}(\tau)-\Ric(\omega_{\psi_s}(\tau))-\omega_{\psi_s}(\tau).
\end{equation}
Define $f_{\psi_s}(\tau):=e^{-\tau}\Phi_{e^{-\tau}}^*f_{\overline{\varphi}_s}(e^\tau-1)$, then by the definition of $f_{\overline{\varphi}_s}$ (see Definition \ref{definition of fvarphi}), we have $f_{\psi_s}=f+\frac{X}{2}\cdot\psi_s$.

    \begin{definition}[Expander region]
        We define the expander region $\Omega_{R,\lambda,s}\subset\{r^2\le\frac{R^2}{se^\tau}\}\times [0,\log(\frac{T_s}{s}+1))$ as follows:
        \begin{equation*}
            \Omega_{R,\lambda,s}:=\mathcal{B}(\sqrt{\lambda})\times [0,T'_s), 
        \end{equation*}
        where $T'_s=\min\{\log(\frac{T_s}{s}+1),\log(\frac{R^2 }{\lambda s})\}$. Its parabolic boundary is then defined as:
        \begin{equation*}
            \partial_\textrm{P}\Omega_{R,\lambda,s}:=(\mathcal{B}(\sqrt{\lambda})\times \{0\}) \cup (\partial \mathcal{B}(\sqrt{\lambda})\times [0,T_s')).
        \end{equation*}
    \end{definition}
The expander region is chosen so that its parabolic boundary is entirely contained in the conical region. In figure \ref{normalized space time}, the expander region is the yellow rectangle; its left and upper boundaries are contained in the modified conical region, which is dashed. This choice allows us to apply the maximum principle to extend our estimates to the entire region of interest.

\label{sectn; conical region}

Using the estimates of Subsection \ref{subsection:estimatesconeregion}, we get the following.
\begin{prop}\label{boundary data}
The following estimates hold on $\partial_\textnormal{P}\Omega_{R,\lambda,s}$:
    \begin{enumerate}
        \item \label{boundarydata1} $|g_{\psi_s}-g_E|_{g_E}\le \Psi(R,\lambda^{-1})$;
        \item \label{boundarydata2} $|\nabla^{g_E}g_{\psi_s}|^2_{g_{\psi_s}}\le \Psi(R,\lambda^{-1})f^{-1}$;
        \item \label{boundarydata3} $|\psi_s|\le \Psi(R,\lambda^{-1})f$;
        \item \label{boundarydata4} $|X\cdot \psi_s|\le \Psi(R,\lambda^{-1})f$;
        \item \label{boundarydata5} $|\dot \psi_s|\le \Psi(R,\lambda^{-1})f$;
        \item \label{boundarydata6} $|JX\cdot\psi_s|\le \Psi(R,\lambda^{-1})f$;
        \item \label{boundarydata7} $|\nabla^{g_{\psi_s}}(JX\cdot\psi_s)|^2_{g_{\psi_s}}\le \Psi(R,\lambda^{-1})f$;
        \item \label{boundarydata8} $|f_{\psi_s}-f|\le \Psi(R,\lambda^{-1})f$;
        \item \label{boundarydata9} $|\nabla^{g_{\psi_s}}f_{\psi_s}|_{g_{\psi_s}}^2\le (2+\Psi(R,\lambda^{-1}))f_{\psi_s}$.
    \end{enumerate}
\end{prop}
    \begin{proof}
On $\{(x,0)\ |\ r(x)^2\le\lambda\}$, since we choose $\lambda\le s^{-\frac{1}{2}}$, thanks to the cut-off function, all the above quantities except $|\nabla^{g_{\psi_s}}f_{\psi_s}|_{g_{\psi_s}}^2$ are zero. In this case 
\begin{equation*}|\nabla^{g_{\psi_s}}f_{\psi_s}|_{g_{\psi_s}}^2=|\nabla^{g_E}f|_{g_E}^2<2f=2f_{\psi_s}.\end{equation*}
If $(x,\tau)$ satisfies $r(x)^2=\lambda, \tau<T'_s$, then $(\Phi_{e^{-\tau}}(x),e^\tau-1)$ lies on the conical region. The above results then come from the correspondence $\psi_s(x,\tau)=e^{-\tau}\overline \varphi_s(\Phi_{e^{-\tau}}(x),e^\tau-1)$, \eqref{equation of normalized complex monge ampere} and our rough estimates on the conical region.
    \end{proof} 

\subsection{Potential estimates on the expander region}

The following maximum principle on the expander region will be essential throughout the rest of the paper. 

    \begin{lemma}[Maximum principle on expander region]\label{maximum principle on expander region}
        Let $g(\tau)_{\tau\in [0,T'_s)}$ be a smooth family of Riemannian metrics, and $Y(t)_{t\in [0,T'_s)}$ be a smooth vector field defined on $\textnormal{Int}\ \Omega_{R,\lambda,s}$.
        
        Assume that $u$ is a continuous function defined on $\Omega_{R,\lambda,s}$ such that:
       \begin{enumerate}
           \item The function $u$ is smooth on $\textnormal{Int}\ \Omega_{R,\lambda,s}$. On $\textnormal{Int}\ \Omega_{R,\lambda,s}$, there exist constants $A>0,B\ge 0$ such that
           \begin{equation*}
              \left( \frac{\partial}{\partial\tau}-\frac{1}{2}\Delta_{g(\tau)}-Y(\tau)\right)u\le B-Au.
           \end{equation*}
           \item There exists a constant $C>0$ such that on $ \partial_\textrm{P}\Omega_{R,\lambda,s}$, $u\le C$ holds.
        \end{enumerate}
       Then on $\Omega_{R,\lambda,s}$, we have $u\le\max\{C,\frac{B}{A}\}$.
    \end{lemma}
    \begin{proof}
        For any $T\in [0,T'_s)$, let us consider $\Omega_{R,\lambda,s}^T:=\{r^2\le \lambda\}\times [0,T]$ which is a compact set. Suppose that $(x_0,\tau_0)\in \Omega_{R,\lambda,s}^T$ is the maximum point, that is, $u(x_0,\tau_0)=\max_{\Omega_{R,\lambda,s}^T}u$. If $(x_0,\tau_0)\in \partial_\textnormal{P}\Omega_{R,\lambda,s}$, then we have $u(x_0,\tau_0)\le C$. If $(x_0,\tau_0)\in \textnormal{Int}\ \Omega_{R,\lambda,s}$, then by the weak maximum principle, we have
        \begin{equation*}
            0\le \left(\frac{\partial}{\partial\tau}-\frac{1}{2}\Delta_{g(\tau_0)}-Y(\tau_0)\right)u(x_0,\tau_0)\le B-Au(x_0,\tau_0).
        \end{equation*}
        Therefore, we get $u(x_0,\tau_0)\le \frac{B}{A}$, and hence $u(x_0,\tau_0)\le \max\{C,\frac{B}{A}\}$. For any $(x,\tau)\in\Omega_{R,\lambda,s}^T$, we have that 
        \begin{equation*}
            u(x,\tau)\le u(x_0,\tau_0)\le\max\{C,\frac{B}{A}\}.
        \end{equation*}
        Since $T \in [0,T_s')$ is arbitrary, we conclude that $u\le\max\{C,\frac{B}{A}\}$ holds on $\Omega_{R,\lambda,s}$.
    \end{proof}

We define the \textit{drift Laplacian} along the normalized K\"ahler--Ricci flow as 
\begin{align}
\Delta_{\omega_{\psi_s},X}=\Delta_{\omega_{\psi_s}}+\frac{X}{2}
\end{align}

    \begin{lemma}
    
    The function $f_{\psi_s}$ satisfies the following evolution equation:
    \begin{equation}\label{evolution equation of fpsi}
        \frac{\partial}{\partial \tau}f_{\psi_s}=\Delta_{\omega_{\psi_s},X}f_{\psi_s}-f_{\psi_s}.
    \end{equation}
        \end{lemma}
\begin{proof}
    First, we compute
    \begin{equation*}
        \frac{\partial}{\partial\tau}f_{\psi_s}=\frac{X}{2}\cdot\dot\psi_s=\frac{X}{2}\cdot\left(\log\frac{\omega_{\psi_s}^n}{\omega_E^n}+\frac{X}{2}\cdot\psi_s-\psi_s\right).
    \end{equation*}
    Since \begin{equation*}
        \frac{X}{2}\cdot\log\frac{\omega_{\psi_s}^n}{\omega_E^n}=\tr_{\omega_{\psi_s}}\mathcal{L}_{\frac{X}{2}}\omega_{\psi_s}-\tr_{\omega_E}\mathcal{L}_{\frac{X}{2}}\omega_E=\Delta_{\omega_{\psi_s}}f_{\psi_s}-\Delta_{\omega_E} f,
    \end{equation*} and 
    \begin{equation*}
        \frac{X}{2}\cdot\left(\frac{X}{2}\cdot\psi_s-\psi_s\right)=\frac{X}{2}\cdot\left(f_{\psi_s}-f\right)-f_{\psi_s}+f=\frac{X}{2}\cdot f_{\psi_s}-f_{\psi_s}+f-|\partial f|_{g_E}^2.
    \end{equation*}
    The soliton identity $f=\Delta_{\omega_E} f+|\partial f|_{g_E}^2$ implies that
    \begin{equation*}
        \frac{\partial}{\partial \tau}f_{\psi_s}=\Delta_{\omega_{\psi_s},X}f_{\psi_s}-f_{\psi_s},
    \end{equation*}
    holds.
\end{proof}
\begin{corollary}
  There exist $s_0,R_0,\lambda_0>0$ such that for all $s\le s_0,R\le R_0, \text{ and }\lambda\ge\lambda_0$ satisfying $R^2> 4\sqrt{s},\lambda\le\frac{1}{\sqrt{s}}$, on $\Omega_{R,\lambda,s}$, we have
  \begin{equation*}
      f_{\psi_s}\ge 0.
  \end{equation*}
\end{corollary}
\begin{proof}
On the parabolic boundary $\partial_{\textnormal{P}}\Omega_{R,\lambda,s},$ we have
    \begin{equation*}
        f_{\psi_s}\ge f-\Psi(R,\lambda^{-1})f\ge (1-\Psi(R,\lambda^{-1}))\frac{r^2}{2}.
    \end{equation*}
Here, the last inequality is ensured by Corollary \ref{coro of comparison of f and r^2} if $\Psi\le1$. Therefore, fixing $R_0,s_0>0$ such that $\Psi(R,\lambda^{-1})\le \frac{1}{2}$ for all $R\le R_0,\lambda\ge\lambda_0$ we get, on $\partial_{\textnormal{P}}\Omega_{R,\lambda,s},$
    \begin{equation*}
        f_{\psi_s}\ge 0.
    \end{equation*}
    Since $f_{\psi_s}$ satisfies the evolution equation $ \frac{\partial}{\partial \tau}f_{\psi_s}=\Delta_{\omega_{\psi_s},X}f_{\psi_s}-f_{\psi_s},$ 
    the maximum principle in Lemma \ref{maximum principle on expander region}, implies that $f_{\psi_s}\ge 0$ holds on $\Omega_{R,\lambda,s}$.
\end{proof}
The following Lemma recalls the Bochner formula along the modified K\"ahler--Ricci flow. For a proof of this, we refer to \cite[Lemma 4.10]{longteng2}.
\begin{lemma}[Bochner formula along the normalized K\"ahler--Ricci flow]\label{bochner formula;lem}
    Let $u$ be a smooth function on $\Omega_{R,\lambda,s}$ satisfying the following evolution equation:
    \begin{equation*}
        \left( \frac{\partial}{\partial\tau}-\Delta_{\omega_{\psi_s},X}\right) u = -u
    \end{equation*}
    along the modified K\"ahler--Ricci flow. Then its gradient satisfies
    \begin{equation}\label{bochner formula eq}
         \left( \frac{\partial}{\partial\tau}-\Delta_{\omega_{\psi_s},X}\right)|\nabla^{g_{\psi_s}}u|_{g_{\psi_s}}^2=-|\nabla^{g_{\psi_s}}u|_{g_{\psi_s}}^2 -|(\nabla^{{g_{\psi_s}}})^2u|_{g_{\psi_s}}^2.
    \end{equation}
\end{lemma}

As a consequence of the formula above, we obtain the following gradient estimates.
\begin{corollary}\label{uniform estimtes on gradient of fpsi}
    On $\Omega_{R,\lambda,s}$, we have
    \begin{equation*}
        |\nabla^{g_{\psi_s}}f_{\psi_s}|_{g_{\psi_s}}^2\le (2+\Psi(R,\lambda^{-1}))f_{\psi_s}.
    \end{equation*}
\end{corollary}
\begin{proof} Combining \eqref{evolution equation of fpsi} with Lemma \ref{bochner formula;lem} yields
    \begin{equation*}
        \begin{split}
            \left(\frac{\partial}{\partial\tau}-\Delta_{\omega_{\psi_s},X}\right) |\nabla^{g_{\psi_s}}f_{\psi_s}|_{g_{\psi_s}}^2&=- |\nabla^{g_{\psi_s}}f_{\psi_s}|_{g_{\psi_s}}^2- |(\nabla^{g_{\psi_s}})^2f_{\psi_s}|_{g_{\psi_s}}^2\le - |\nabla^{g_{\psi_s}}f_{\psi_s}|_{g_{\psi_s}}^2.
        \end{split}
    \end{equation*}
    On $\partial_\textnormal{P}\Omega_{R,\lambda,s}$, we have that 
    \begin{equation*}
         |\nabla^{g_{\psi_s}}f_{\psi_s}|_{g_{\psi_s}}^2\le (2+\Psi(R,\lambda^{-1}))f_{\psi_s}.
    \end{equation*}
    Now we consider $ |\nabla^{g_{\psi_s}}f_{\psi_s}|_{g_{\psi_s}}^2-(2+\Psi(R,\lambda^{-1}))f_{\psi_s}$ and compute
    \begin{equation*}
         \begin{split}
              & \left(\frac{\partial}{\partial\tau}-\Delta_{\omega_{\psi_s},X}\right) \left(|\nabla^{g_{\psi_s}}f_{\psi_s}|_{g_{\psi_s}}^2-(2+\Psi(R,\lambda^{-1}))f_{\psi_s}\right)\\
              &\le -\left(|\nabla^{g_{\psi_s}}f_{\psi_s}|_{g_{\psi_s}}^2-(2+\Psi(R,\lambda^{-1}))f_{\psi_s}\right).
         \end{split}
    \end{equation*}
The claim therefore follows from Lemma \ref{maximum principle on expander region} and Proposition \ref{boundary data} \ref{boundarydata9}.
\end{proof}
    \begin{corollary}\label{coro C^0 bound}
        On $\Omega_{R,\lambda,s}$ we have
        \begin{equation*}
            \begin{split}
                -\Psi(R,\lambda^{-1})f_{\psi_s}&\le \psi_s\le \Psi(R,\lambda^{-1})f,\\
                -\Psi(R,\lambda^{-1})f_{\psi_s}&\le\dot\psi_s\le \Psi(R,\lambda^{-1})f_{\psi_s}.
            \end{split}
        \end{equation*}
    \end{corollary}
    \begin{proof}
        We notice that $\dot\psi_s$ satisfies 
        \begin{equation*}
            \left(\frac{\partial}{\partial\tau}-\Delta_{\omega_{\psi_s},X}\right)\dot\psi_s=-\dot\psi_s.
        \end{equation*}
       Moreover, on $\partial_\textnormal{P}\Omega_{R,\lambda,s}$, Proposition \ref{boundary data} \ref{boundarydata5} gives
       \begin{equation*}
           |\dot\psi_s|\le \Psi(R,\lambda^{-1})f.
       \end{equation*}
       Since on $\partial_\textnormal{P}\Omega_{R,\lambda,s}$ it holds that $|f_{\psi_s}-f|\le \Psi(R,\lambda^{-1})f$, we obtain
       \begin{equation*}
            |\dot\psi_s|\le\Psi(R,\lambda^{-1})f_{\psi_s}.
       \end{equation*}
    Now we consider the function $\dot\psi_s-\Psi(R,\lambda^{-1})f_{\psi_s}.$ On one hand, we have
       \begin{equation*}
          \left(\frac{\partial}{\partial\tau}-\Delta_{\omega_{\psi_s},X}\right) (\dot\psi_s-\Psi(R,\lambda^{-1})f_{\psi_s})=-(\dot\psi_s-\Psi(R,\lambda^{-1})f_{\psi_s}),
       \end{equation*}
       and on the other hand, on $\partial_\textnormal{P}\Omega_{R,\lambda,s}$, we have $\dot\psi_s-\Psi(R,\lambda^{-1})f_{\psi_s}\le 0$. Hence, by Lemma \ref{maximum principle on expander region}, we have that
       \begin{equation*}
           \dot\psi_s\le \Psi(R,\lambda^{-1}) f_{\psi_s},
       \end{equation*}
       holds on $\Omega_{R,\lambda,s}$. Similarly, we can prove $\dot\psi_s\ge -\Psi(R,\lambda^{-1})f_{\psi_s}$ on $\Omega_{R,\lambda,s}$.

       Recalling the modified complex Monge--Amp\`ere equation
       \begin{equation*}
           \frac{\partial}{\partial\tau}\psi_s=\log\frac{\omega_{\psi_s}^n}{\omega_E^n}+\frac{X}{2}\cdot\psi_s-\psi_s
       \end{equation*}
       and diagonalising $\omega_{\psi_s}$ with respect to $\omega_E,$ an elementary algebraic inequality for the eigenvalues yields $\Delta_{\omega_{\psi_s}}\psi_s\le\log \frac{\omega_{\psi_s}^n}{\omega_E^n}\le\Delta_{\omega_E}\psi_s$ (see, for instance, the discussion in \cite[Chapter 3]{SongWeinkove}). We can then conclude that
       \begin{equation*}
           \begin{split}
                &\frac{\partial}{\partial\tau}\psi_s\ge \Delta_{\omega_{\psi_s},X}\psi_s-\psi_s,\\
                &\frac{\partial}{\partial\tau}\psi_s\le \Delta_{\omega_E,X}\psi_s-\psi_s.
           \end{split}
       \end{equation*}

For the same reason as before, we have, on $\partial_\textnormal{P}\Omega_{R,\lambda,s}$,
       \begin{equation*}
        -\Psi(R,\lambda^{-1})f_{\psi_s}  \le \psi_s\le \Psi(R,\lambda^{-1})f.
       \end{equation*}
       Together with the fact that $(\frac{\partial}{\partial\tau}-\Delta_{\omega_E,X})f=-f$, on the expander region $\Omega_{R,\lambda,s}$, we obtain
        \begin{equation*}
          -\Psi(R,\lambda^{-1})f_{\psi_s}\le \psi_s\le \Psi(R,\lambda^{-1})f
        \end{equation*}
       \end{proof}
\begin{lemma}
    The functions $JX\cdot\psi_s$ and $|\nabla^{g_{\psi_s}}(JX\cdot\psi_s)|_{g_{\psi_s}}^2$ satisfy the following evolution equations:
    \begin{equation*}
        \begin{split}
            \left(\frac{\partial}{\partial\tau}-\Delta_{\omega_{\psi_s},X}\right)JX\cdot\psi_s&=-JX\cdot\psi_s;\\
            \left(\frac{\partial}{\partial\tau}-\Delta_{\omega_{\psi_s},X}\right)|\nabla^{g_{\psi_s}}(JX\cdot\psi_s)|_{g_{\psi_s}}^2&=-|\nabla^{g_{\psi_s}}(JX\cdot\psi_s)|_{g_{\psi_s}}^2-|(\nabla^{g_{\psi_s}})^2(JX\cdot\psi_s)|_{g_{\psi_s}}^2.
        \end{split}
    \end{equation*}
\end{lemma}
\begin{proof}
    Since $\mathcal{L}_{JX}g_E=0$ and $\mathcal{L}_{JX}X=0$, we have
    \begin{equation*}
        \begin{split}
            \frac{\partial}{\partial\tau}JX\cdot\psi_s&=\tr_{\omega_{\psi_s}}\mathcal{L}_{JX}\omega_{\psi_s}-\tr_{\omega_E}\mathcal{L}_{JX}\omega_E+JX\cdot\frac{X}{2}\cdot\psi_s-JX\cdot\psi_s\\
            &=\Delta_{\omega_{\psi_s}}JX\cdot\psi_s+\frac{X}{2}\cdot(JX\cdot\psi_s)-JX\cdot\psi_s.
        \end{split}
    \end{equation*}
    We can then use Bochner's formula \eqref{bochner formula eq} to get the evolution of \(|\nabla^{g_{\psi_s}}(JX\cdot\psi_s)|_{g_{\psi_s}}^2.\)
\end{proof}
\begin{corollary}\label{uniform estimates on Killing part}
    On $\Omega_{R,\lambda,s}$, we have 
    \begin{equation*}
        \begin{split}
            &|JX\cdot\psi_s|\le\Psi(R,\lambda^{-1})f_{\psi_s};\\
            &|\nabla^{g_{\psi_s}}(JX\cdot\psi_s)|_{g_{\psi_s}}^2\le \Psi(R,\lambda^{-1})f_{\psi_s}.
        \end{split}
    \end{equation*}
\end{corollary}
\begin{proof}
    First, on $\partial_\textnormal{P}\Omega_{R,\lambda,s}$, we have $|JX\cdot\psi_s|\le\Psi(R,\lambda^{-1})f_{\psi_s}$. By the evolution equation of $JX\cdot\psi_s$ and Lemma \ref{maximum principle on expander region}, we have $|JX\cdot\psi_s|\le\Psi(R,\lambda^{-1})f_{\psi_s}$ on $\Omega_{R,\lambda,s}$.
    Moreover, the gradient term satisfies
     \begin{equation*}
        \begin{split}
            \left(\frac{\partial}{\partial\tau}-\Delta_{\omega_{\psi_s},X}\right)|\nabla^{g_{\psi_s}}(JX\cdot\psi_s)|_{g_{\psi_s}}^2&=-|\nabla^{g_{\psi_s}}(JX\cdot\psi_s)|_{g_{\psi_s}}^2-|(\nabla^{g_{\psi_s}})^2(JX\cdot\psi_s)|_{g_{\psi_s}}^2\\
            &\le -|\nabla^{g_{\psi_s}}(JX\cdot\psi_s)|_{g_{\psi_s}}^2.
        \end{split}
    \end{equation*}
    Since $|\nabla^{g_{\psi_s}}(JX\cdot\psi_s)|_{g_{\psi_s}}^2\le\Psi(R,\lambda^{-1})f_{\psi_s}$ on $\partial_\textnormal{P}\Omega_{R,\lambda,s}$, we have that
    \begin{equation*}
        |\nabla^{g_{\psi_s}}(JX\cdot\psi_s)|_{g_{\psi_s}}^2\le \Psi(R,\lambda^{-1})f_{\psi_s},
    \end{equation*}
    holds on $\Omega_{R,\lambda,s}$ due to Lemma \ref{maximum principle on expander region} and Proposition \ref{boundary data} \ref{boundarydata6} and \ref{boundarydata7}.
\end{proof}

For the discussion below, we fix a positive constant $A>0$ such that, on the asymptotically conical gradient K\"ahler--Ricci expander $(E,g_E,X)$, we have 
\begin{equation*}
   -A g_E\le (\nabla^{g_E})^2f\le Ag_E,
\end{equation*}
whose existence follows from curvature bounds on the expander. The following propositions show that we can compare $f_{\psi_s}$ and $f$ on the expander region uniformly as long as we take $R,\lambda^{-1}$ sufficiently small.
\begin{prop}\label{control f with fpsi}
    There exist $s_0,R_0,\lambda_0>0$ and a uniform constant $D>0$ such that for all $s\le s_0,R\le R_0,\lambda\ge\lambda_0$ with $R^2> 4\sqrt{s}, \text{ and }\lambda\le\frac{1}{\sqrt{s}}$, we have
    \begin{equation*}
        f\le D(f_{\psi_s}+1),
    \end{equation*}
holds on $\Omega_{R,\lambda,s}.$
\end{prop}

\begin{proof}
First by our initial setting, we have $\Psi(R,\lambda^{-1})\le \frac{1}{2}<1.$ We compute
    \begin{equation*}
        \begin{split}
            \left(\frac{\partial}{\partial\tau}-\Delta_{\omega_{\psi_s},X}\right)(f-A\psi_s)&=-\Delta_{\omega_{\psi_s}}f-\frac{X}{2}\cdot f-A\dot\psi_s+A\Delta_{\omega_{\psi_s}}\psi_s+A\frac{X}{2}\cdot\psi_s\\
            &=-\Delta_{\omega_{\psi_s}}f-|\partial f|^2_{g_E}-A\dot\psi_s+A(n-\tr_{\omega_{\psi_s}}\omega_E)+A(f_{\psi_s}-f).
        \end{split}
    \end{equation*}
    Since $\Delta_{\omega_{\psi_s}}f=\tr_{\omega_{\psi_s}}(\nabla^{g_E})^2f\ge -A\tr_{\omega_{\psi_s}}\omega_E$, it follows that
    \begin{equation*}
        \begin{split}
             \left(\frac{\partial}{\partial\tau}-\Delta_{\omega_{\psi_s},X}\right)(f-A\psi_s)&\le A\tr_{\omega_{\psi_s}}\omega_E-|\partial f|^2_{g_E}-A\dot\psi_s+A(n-\tr_{\omega_{\psi_s}}\omega_E)+A(f_{\psi_s}-f)\\
             &=An-|\partial f|^2_{g_E}-A\dot\psi_s+A(f_{\psi_s}-f).
        \end{split}
    \end{equation*}
    By the soliton identity $|\partial f|^2_{g_E}=f-R_{\omega_E}-n\ge f-C_1$ for some universal constant $C_1>0,$ and the fact that $\dot\psi_s\ge -\Psi(R,\lambda^{-1})f_{\psi_s}\ge -f_{\psi_s}$, there exists a universal constant $C>0$ such that
    \begin{equation*}
         \left(\frac{\partial}{\partial\tau}-\Delta_{\omega_{\psi_s},X}\right)(f-A\psi_s)\le C+2Af_{\psi_s}-(A+1)f.
    \end{equation*}
    Because we have $\psi_s\ge-\Psi(R,\lambda^{-1})f_{\psi_s}\ge -f_{\psi_s}$, it holds that
    \begin{equation*}
        \begin{split}
            & \left(\frac{\partial}{\partial\tau}-\Delta_{\omega_{\psi_s},X}\right)(f-A\psi_s)\\
            &\le C+2Af_{\psi_s}-(A+1)f\\
            &=C+2Af_{\psi_s}-(A+1)(f-A\psi_s)-A(A+1)\psi_s\\ 
            &\le C-(A+1)(f-A\psi_s)+(2A+A(A+1))f_{\psi_s}.
        \end{split}
    \end{equation*}
    Now we consider $u:=f-A\psi_s-B f_{\psi_s}$ for some positive constant $B>0$ to be determined. We compute
    \begin{equation*}
        \begin{split}
            &\left(\frac{\partial}{\partial\tau}-\Delta_{\omega_{\psi_s},X}\right)u\\
            &\le C-(A+1)(f-A\psi_s)
           +(2A+A(A+1))f_{\psi_s}+Bf_{\psi_s}\\
            &=C-(A+1)(f-A\psi_s-Bf_{\psi_s})\\
         &\quad +(2A+A(A+1))f_{\psi_s}+Bf_{\psi_s}-(A+1)Bf_{\psi_s}\\
            &=C-(A+1)u+(2A+A(A+1))f_{\psi_s}-ABf_{\psi_s}.\\
        \end{split}
    \end{equation*}
    On $\partial_{\textnormal{P}}\Omega_{R,\lambda,s}$ we have 
    \begin{equation*}
        \begin{split}
            u&=f-A\psi_s-B f_{\psi_s}\\
            &\le f+A\Psi(R,\lambda^{-1})f-Bf_{\psi_s}\\
            &\le f+A\Psi(R,\lambda^{-1})f-B(1-\Psi(R,\lambda^{-1}))f.\\
        \end{split}
    \end{equation*}
    
    Since $\Psi(R,\lambda^{-1})\le\frac{1}{2}$, we have $u\le f+Af-\frac{B}{2}f$ on $\partial_{\textnormal{P}}\Omega_{R,\lambda,s}$. Finally, we pick $B=2A+3.$ Then $u\le 0$ on $\partial_{\textnormal{P}}\Omega_{R,\lambda,s}$. Now on $\Omega_{R,\lambda,s}$, since $f_{\psi_s}\ge 0,$ we have
    \begin{equation*}
        \left(\frac{\partial}{\partial\tau}-\Delta_{\omega_{\psi_s},X}\right)u\le C-(A+1)u.
    \end{equation*}
   From Lemma \ref{maximum principle on expander region}, there exists a universal constant $C'>0$ such that
    \begin{equation*}
        f-A\psi_s-Bf_{\psi_s}=u\le C'.
    \end{equation*}
    Let us take $R_0,\lambda_0,s_0>0$ such that for all $s\le s_0,\lambda\ge\lambda_0,R\le R_0$ with $R^2> 4\sqrt{s},\lambda\le\frac{1}{\sqrt{s}}$, we have $A\Psi(R,\lambda^{-1})\le\frac{1}{2}.$ Then we get
    \begin{equation*}
        f\le A\psi_s+Bf_{\psi_s}+C'\le A\Psi(R,\lambda^{-1})f+Bf_{\psi_s}+C'\le\frac{1}{2}f+Bf_{\psi_s}+C'.
    \end{equation*}
    Taking $D=2B+2C'+2$, we have $f+1\le D(f_{\psi_s}+1)$ on $\Omega_{R,\lambda ,s}$.
\end{proof}
To control $f_{\psi_s}+1$ in terms of $f+1$, we require the following lemma, which shows that $f_{\psi_s}$ is an approximate Hamiltonian function for $X$.
\begin{lemma}
    We have
    \begin{equation*}
        \nabla^{g_{\psi_s}}f_{\psi_s}=X+\frac{J}{2}\nabla^{g_{\psi_s}}(JX\cdot\psi_{s})=: X+X_{\Delta}.
    \end{equation*}
\end{lemma}
\begin{proof}
    It suffices to show that 
\begin{equation*}
    i_X g_{\psi_s} + \frac{1}{2}g_{\psi_s}(J\nabla^{g_{\psi_s}}(JX\cdot \psi_s),\cdot)=df_{\psi_s}.
\end{equation*}
Using Cartan's formula and $\sqrt{-1}\partial \overline{\partial} =-\frac{1}{2}dJd$, we compute
\begin{align*}
    \frac{1}{2}g_{\psi_s}(J\nabla^{g_{\psi_s}}(JX\cdot \psi_s),\cdot) &= -\frac{1}{2}Jd\mathcal{L}_{JX}\psi_s=-\frac{1}{2}\mathcal{L}_{JX}(Jd\psi_s)\\
    &= -\frac{1}{2}i_{JX}dJd\psi_s - \frac{1}{2}di_{JX}Jd\psi_s\\
    &= i_{JX}(\omega_{\psi_s}-\omega_E) +\frac{1}{2} d(X \cdot \psi_s)\\
    &= -i_X g_{\psi_s} +df_{\psi_s},
\end{align*}
and the claim follows by combining expressions.
\end{proof}

\begin{corollary}
    There exists a uniform constant $D>0$ such that on $\Omega_{R,\lambda,s}$,
    \begin{equation*}
        |X\cdot f_{\psi_s}|\le Df_{\psi_s}.
    \end{equation*}
\end{corollary}
\begin{proof}
    Notice that, by Cauchy-Schwarz inequality, we have
    \begin{equation*}
        |X\cdot f_{\psi_s}|\le |\nabla^{g_{\psi_s}}f_{\psi_s}|_{g_{\psi_s}}|\nabla^{g_{\psi_s}}f_{\psi_s}-X_\Delta|_{g_{\psi_s}}\le |\nabla^{g_{\psi_s}}f_{\psi_s}|_{g_{\psi_s}}^2+|X_\Delta|_{g_{\psi_s}}|\nabla^{g_{\psi_s}}f_{\psi_s}|_{g_{\psi_s}}.
    \end{equation*}
    By Corollaries \ref{uniform estimtes on gradient of fpsi} and \ref{uniform estimates on Killing part}, there exists a uniform constant $D>2$ such that 
    \begin{equation*}
         |X\cdot f_{\psi_s}|\le Df_{\psi_s}.
    \end{equation*}
\end{proof}
\begin{prop}\label{control fpsi with f}
    There exist $s_0,R_0,\lambda_0>0$ and a uniform constant $D>0$ such that for all $s\le s_0,R\le R_0,\lambda\ge\lambda_0$ with $R^2\ge 4\sqrt{s},\lambda\le\frac{1}{\sqrt{s}}$, we have that
    \begin{equation*}
        f_{\psi_s}+1\le Df
    \end{equation*}
  holds on $\Omega_{R,\lambda,s}.$
\end{prop}
\begin{proof}
    We define $\Delta_{\omega_{\psi_s},-X}:=\Delta_{\omega_{\psi_{s}}}-\frac{X}{2}$ and compute
    \begin{equation*}
        \begin{split}
            &\left(\frac{\partial}{\partial\tau}-\Delta_{\omega_{\psi_s},-X}\right)(f+A\psi_s)\\
            &=-\Delta_{\omega_{\psi_s}}f+\frac{X}{2}\cdot f+A\dot\psi_s-A\Delta_{\omega_{\psi_s}}\psi_s+A\frac{X}{2}\cdot\psi_s\\
            &=-\Delta_{\omega_{\psi_s}}f+|\partial f|^2_{g_E}+A\dot\psi_s-A(n-\tr_{\omega_{\psi_s}}\omega_E)+Af_{\psi_s}-Af.
        \end{split}
    \end{equation*}
    Since $(\nabla^{g_E})^2f\le Ag_E$ for some $A>0$, we have
    \begin{equation*}
         \begin{split}
            &\left(\frac{\partial}{\partial\tau}-\Delta_{\omega_{\psi_s},-X}\right)(f+A\psi_s)\\
            &\ge -A\tr_{\omega_{\psi_s}}\omega_E+|\partial f|^2_{g_E}+A\dot\psi_s-A(n-\tr_{\omega_{\psi_s}}\omega_E)+Af_{\psi_s}-Af\\
            &=-An+|\partial f|^2_{g_E}+A\dot\psi_s+Af_{\psi_s}-Af\\
            &\ge -An+A\dot\psi_s+Af_{\psi_s}-Af.
        \end{split}
    \end{equation*}
    Using that $\dot\psi_s\ge -\Psi(R,\lambda^{-1})f_{\psi_s},$ where $\Psi(R,\lambda^{-1})\le\frac{1}{2}$, we have  
    \begin{equation*}
         \begin{split}
            &\left(\frac{\partial}{\partial\tau}-\Delta_{\omega_{\psi_s},-X}\right)(f+A\psi_s)\\
            &\ge-An+A\dot\psi_s+Af_{\psi_s}-Af\\
            &\ge -C+(A-A\Psi(R,\lambda^{-1}))f_{\psi_s}-Af\\
            &= -C+(A-A\Psi(R,\lambda^{-1}))f_{\psi_s}-A(f+A\psi_s)+A^2\psi_s,
        \end{split}
    \end{equation*}
    with $C>0$ a uniform constant. Since $\psi_s\ge-\Psi(R,\lambda^{-1})f_{\psi_s}$, we have 
    \begin{equation*}
          \begin{split}
            &\left(\frac{\partial}{\partial\tau}-\Delta_{\omega_{\psi_s},-X}\right)(f+A\psi_s)\\
            &\ge -C+(A-A\Psi(R,\lambda^{-1}))f_{\psi_s}-A(f+A\psi_s)+A^2\psi_s\\
            &\ge -C+(A-A\Psi(R,\lambda^{-1})-A^2\Psi(R,\lambda^{-1}))f_{\psi_s}-A(f+A\psi_s).
        \end{split}
    \end{equation*}
    Taking $R_0,s_0,\lambda_0>0$ so that $A\Psi(R,\lambda^{-1})+A^2\Psi(R,\lambda^{-1})\le \frac{A}{2}$ for all $s\le s_0,R\le R_0,\lambda\ge\lambda_0$ with $R^2> 4\sqrt{s},\lambda\le\frac{1}{\sqrt{s}},$ we obtain that
    \begin{equation*}
            \left(\frac{\partial}{\partial\tau}-\Delta_{\omega_{\psi_s},-X}\right)(f+A\psi_s)
            \ge -C+\frac{A}{2}f_{\psi_s}-A(f+A\psi_s),
    \end{equation*}
    holds for uniform constant $C>0$.
    
    Now fix $D>0$ a uniform constant such that the following inequality holds:
    \begin{equation*}
        \left(\frac{\partial}{\partial\tau}-\Delta_{\omega_{\psi_s},-X}\right)f_{\psi_s}=X\cdot f_{\psi_s}-f_{\psi_s}\le Df_{\psi_s}.
    \end{equation*}
    Taking $\alpha \in (0,1)$ to be determined later and considering $v:=-f-A\psi_s+\alpha f_{\psi_s},$ on the one hand, we have
    \begin{equation*}
        \begin{split}
             \left(\frac{\partial}{\partial\tau}-\Delta_{\omega_{\psi_s},-X}\right)v&\le \alpha Df_{\psi_s}+C-\frac{A}{2}f_{\psi_s}+A(f+A\psi_s)\\
             &=C+(\alpha D-\frac{A}{2})f_{\psi_s}+A(f+A\psi_s-\alpha f_{\psi_s})+\alpha A f_{\psi_s}\\
             &=C+(\alpha D-\frac{A}{2}+\alpha A)f_{\psi_s}-Av.
        \end{split}
    \end{equation*}
    On the other hand, on the parabolic boundary $\partial_{\textnormal P}\Omega_{R,\lambda,s}$, we have that 
    \begin{equation*}
       \begin{split}
            v&=\alpha f_{\psi_s}-f-A\psi_s \le \alpha (1+\Psi(R,\lambda^{-1}))f+A\Psi(R,\lambda^{-1})f-f\\
            &\le 2\alpha f+\frac{1}{2}f-f< 2\alpha f-\frac{1}{2}f.
       \end{split}
    \end{equation*}
    Now we fix $\alpha=\min\{\frac{1}{4},\frac{A}{2(A+D)}\}>0.$ It follows that $v\le 0$ on $\partial_\textnormal P\Omega_{R,\lambda,s},$ and 
    \begin{equation*}
          \left(\frac{\partial}{\partial\tau}-\Delta_{\omega_{\psi_s},-X}\right)v\le C-Av
    \end{equation*}
    on $\Omega_{R,\lambda,s}$. Lemma \ref{maximum principle on expander region} then implies that there exists a uniform constant $C'>0$ such that on $\Omega_{R,\lambda,s}$, we have
    \begin{equation*}
        \alpha f_{\psi_s}\le f+A\psi_s+C'\le f+A\Psi(R,\lambda^{-1})f+C'.
    \end{equation*}
    Since $\alpha>0$, we conclude that there exists a uniform constant $D>0$ such that
    \begin{equation*}
        f_{\psi_s}+1\le D(f+1)
    \end{equation*}
    holds on $\Omega_{R,\lambda,s}$.
\end{proof}

\subsection{\texorpdfstring{$C^2$}{C\^2}-estimates on the expander region}
The next theorem is the main result of this section, and the method can be traced back to Yau's celebrated $C^2$-estimate. Let $R_0, s_0 \in (0,\infty)$  and $\lambda_0 \in (0,\infty)$ be the minimums and maximum, respectively, of the quantities appearing in Proposition \ref{control f with fpsi} and Proposition \ref{control fpsi with f}. Throughout this subsection, we will choose parameters $R, s, \lambda > 0$ satisfying \eqref{eq:backgroundassumption}.

\begin{theorem}[$C^2$-estimates]\label{C^2 estimate}
    There exists a uniform constant $C>1$ such that, on $\Omega_{R,\lambda,s}$, we have
    \begin{equation*}
        \frac{1}{C}\omega_E\le\omega_{\psi_s}\le C\omega_E.
    \end{equation*}
\end{theorem} 
Before proving the theorem, we will need to introduce a barrier function to deal with the extra drift term coming from the drift Laplacian $\Delta_{\omega_{\psi_s},X}.$ See also the results in \cite[Section 4]{longteng2} for a similar approach.

\begin{lemma}[Barrier function]\label{barrier function}
    There exists a smooth, uniformly bounded barrier function $\Theta(\psi_s)$ defined on $\Omega_{R,\lambda,s}$ such that 
    \begin{equation}\label{equation of barrier function}
         \left(\frac{\partial}{\partial\tau}-\Delta_{\omega_{\psi_s},X}\right)\Theta(\psi_s)\ge \frac{1}{4(f_{\psi_s}+1)}\left(\tr_{\omega_{\psi_s}}\omega_E+\left( \log \frac{\omega_{\psi_s}^n}{\omega_E^n} \right)_+-C\right)
    \end{equation}
    holds on $\Omega_{R,\lambda,s},$ for some uniform constant $C>0$.
\end{lemma}
\begin{proof}
    We start by considering the function $\frac{\psi_s}{f_{\psi_s}+1}$ which is well-defined since $f_{\psi_s}\ge 0$. Then,
    \begin{equation*}
        \begin{split}
            \left(\frac{\partial}{\partial\tau}-\Delta_{\omega_{\psi_s},X}\right)\frac{\psi_s}{f_{\psi_s}+1}&=\frac{ \left(\frac{\partial}{\partial\tau}-\Delta_{\omega_{\psi_s},X}\right)\psi_s}{f_{\psi_s}+1}+\psi_s\left(\frac{\partial}{\partial\tau}-\Delta_{\omega_{\psi_s},X}\right)\frac{1}{f_{\psi_s}+1}\\
            &\quad -2\Re<\partial\psi_s,\bar\partial\frac{1}{f_{\psi_s}+1}>_{g_{\psi_s}}\\
            &=\frac{\dot\psi_s-\frac{X}{2}\cdot\psi_s-\Delta_{\omega_{\psi_s}}\psi_s}{f_{\psi_s}+1}+\psi_s\left(-\frac{\left(\frac{\partial}{\partial\tau}-\Delta_{\omega_{\psi_s},X}\right)f_{\psi_s}}{(f_{\psi_s}+1)^2}-2\frac{|\partial f_{\psi_s}|_{g_{\psi_s}}^2}{(f_{\psi_s}+1)^3}\right)\\
            &\quad +\frac{\nabla^{g_{\psi_s}}f_{\psi_s}\cdot\psi_s}{(f_{\psi_s}+1)^2}\\
            &=\frac{\dot\psi_s-\frac{X}{2}\cdot\psi_s+\tr_{\omega_{\psi_s}}\omega_E-n}{f_{\psi_s}+1}+\psi_s\left(\frac{f_{\psi_s}}{(f_{\psi_s}+1)^2}-\frac{|\nabla^{g_{\psi_s}} f_{\psi_s}|_{g_{\psi_s}}^2}{(f_{\psi_s}+1)^3}\right)\\
            &\quad +\frac{\nabla^{g_{\psi_s}}f_{\psi_s}\cdot\psi_s}{(f_{\psi_s}+1)^2}.\\
        \end{split}
    \end{equation*}
    Since initially we chose $R,\lambda$ such that $\Psi(R,\lambda^{-1})<1$, it follows that
    \[|\nabla^{g_{\psi_s}}f_{\psi_s}|_{g_{\psi_s}}^2\le (2+\Psi(R,\lambda^{-1}))f_{\psi_s}+\Psi(R,\lambda^{-1})\le 3(f_{\psi_s}+1).\]
Then, we have 
    \begin{equation*}
        \begin{split}
             \left(\frac{\partial}{\partial\tau}-\Delta_{\omega_{\psi_s},X}\right)\frac{\psi_s}{f_{\psi_s}+1}  &=\frac{\dot\psi_s-\frac{X}{2}\cdot\psi_s+\psi_s+\tr_{\omega_{\psi_s}}\omega_E-n}{f_{\psi_s}+1}+\psi_s\left(-\frac{1}{(f_{\psi_s}+1)^2}-\frac{|\nabla^{g_{\psi_s}} f_{\psi_s}|_{g_{\psi_s}}^2}{(f_{\psi_s}+1)^3}\right)\\
            &\quad +\frac{\nabla^{g_{\psi_s}}f_{\psi_s}\cdot\psi_s}{(f_{\psi_s}+1)^2}\\
            &\ge \frac{\dot\psi_s-\frac{X}{2}\cdot\psi_s+\psi_s+\tr_{\omega_{\psi_s}}\omega_E-n}{f_{\psi_s}+1}-\frac{4|\psi_s|}{(f_{\psi_s}+1)^2}+\frac{\nabla^{g_{\psi_s}}f_{\psi_s}\cdot\psi_s}{(f_{\psi_s}+1)^2}\\
            &=\frac{\log\frac{\omega_{\psi_s}^n}{\omega_E^n}+\tr_{\omega_{\psi_s}}\omega_E-n}{f_{\psi_s}+1}-\frac{4|\psi_s|}{(f_{\psi_s}+1)^2}+\frac{\nabla^{g_{\psi_s}}f_{\psi_s}\cdot\psi_s}{(f_{\psi_s}+1)^2}
        \end{split}
    \end{equation*}
    Recalling that $|\psi_s|\le D\Psi(R,\lambda^{-1})(f_{\psi_{s}}+1)\le D(f_{\psi_{s}}+1)$, we obtain
    \begin{equation*}
         \left(\frac{\partial}{\partial\tau}-\Delta_{\omega_{\psi_s},X}\right)\frac{\psi_s}{f_{\psi_s}+1}  \ge \frac{\log\frac{\omega_{\psi_s}^n}{\omega_E^n}+\tr_{\omega_{\psi_s}}\omega_E-n}{f_{\psi_s}+1}-\frac{4D}{f_{\psi_s}+1}-\frac{|\nabla^{g_{\psi_s}}f_{\psi_s}|_{g_{\psi_s}}|\nabla^{g_{\psi_s}}\psi_s|_{g_{\psi_s}}}{(f_{\psi_s}+1)^2}.
    \end{equation*}
   Applying the Cauchy-Schwarz inequality, we have, for any $\varepsilon>0$
    \begin{equation*}
        \begin{split}
            \left(\frac{\partial}{\partial\tau}-\Delta_{\omega_{\psi_s},X}\right)\frac{\psi_s}{f_{\psi_s}+1} 
            &\ge \frac{\log\frac{\omega_{\psi_s}^n}{\omega_E^n}+\tr_{\omega_{\psi_s}}\omega_E}{f_{\psi_s}+1}-\frac{4D+n}{f_{\psi_s}+1}-\frac{1}{4\varepsilon}\frac{|\nabla^{g_{\psi_s}}f_{\psi_s}|_{g_{\psi_s}}^2}{(f_{\psi_s}+1)^2}-\frac{\varepsilon|\nabla^{g_{\psi_s}}\psi_s|^2_{g_{\psi_s}}}{(f_{\psi_s}+1)^2}\\
            &\ge \frac{\log\frac{\omega_{\psi_s}^n}{\omega_E^n}+\tr_{\omega_{\psi_s}}\omega_E}{f_{\psi_s}+1}-\frac{4D+n}{f_{\psi_s}+1}-\frac{1}{4\varepsilon}\frac{3}{f_{\psi_s}+1}-\frac{\varepsilon|\nabla^{g_{\psi_s}}\psi_s|^2_{g_{\psi_s}}}{(f_{\psi_s}+1)^2}.
        \end{split}
    \end{equation*}
    Analogously, we consider the function $\frac{\psi_s^2}{(f_{\psi_s}+1)^2}$ and compute
    \begin{equation*}
        \begin{split}
             \left(\frac{\partial}{\partial\tau}-\Delta_{\omega_{\psi_s},X}\right)\frac{\psi_s^2}{(f_{\psi_s}+1)^2}&=\frac{\left(\frac{\partial}{\partial\tau}-\Delta_{\omega_{\psi_s},X}\right)\psi_s^2}{(f_{\psi_s}+1)^2}+\psi_s^2\left(\frac{\partial}{\partial\tau}-\Delta_{\omega_{\psi_s},X}\right)\frac{1}{(f_{\psi_s}+1)^2}\\
             &\quad -2\Re \langle\partial \psi_s^2,\bar \partial \frac{1}{(f_{\psi_s}+1)^2}\rangle_{g_{\psi_s}}\\
             &=\frac{2\psi_s\dot\psi_s-\psi_sX\cdot\psi_s-2\psi_s\Delta_{\omega_{\psi_s}}\psi_s-|\nabla^{g_{\psi_s}}\psi_s|_{g_{\psi_s}}^2}{(f_{\psi_s}+1)^2}\\
             &\quad +\psi_s^2\left(-2\frac{\left(\frac{\partial}{\partial\tau}-\Delta_{\omega_{\psi_s},X}\right)f_{\psi_s}}{(f_{\psi_s}+1)^3}-6\frac{|\partial f_{\psi_s}|_{g_{\psi_s}}^2}{(f_{\psi_s}+1)^4}\right)\\
             &\quad +4\psi_s\frac{\nabla^{g_{\psi_s}}f_{\psi_s}\cdot\psi_s}{(f_{\psi_s}+1)^3}\\
             &\le \frac{2\psi_s\dot\psi_s-\psi_sX\cdot\psi_s-2\psi_s\Delta_{\omega_{\psi_s}}\psi_s-|\nabla^{g_{\psi_s}}\psi_s|_{g_{\psi_s}}^2}{(f_{\psi_s}+1)^2}\\
             &\quad +\frac{2\psi_s^2f_{\psi_s}}{(f_{\psi_s}+1)^3}+4\psi_s\frac{\nabla^{g_{\psi_s}}f_{\psi_s}\cdot\psi_s}{(f_{\psi_s}+1)^3}\\
             &\le \frac{2\psi_s^2+2\psi_s\dot\psi_s-\psi_sX\cdot\psi_s-2\psi_s\Delta_{\omega_{\psi_s}}\psi_s-|\nabla^{g_{\psi_s}}\psi_s|_{g_{\psi_s}}^2}{(f_{\psi_s}+1)^2}\\
             &\quad +4\psi_s\frac{\nabla^{g_{\psi_s}}f_{\psi_s}\cdot\psi_s}{(f_{\psi_s}+1)^3}\\
             &=\frac{2\psi_s\log\frac{\omega_{\psi_s}^n}{\omega_E^n}-2\psi_s\Delta_{\omega_{\psi_s}}\psi_s-|\nabla^{g_{\psi_s}}\psi_s|_{g_{\psi_s}}^2}{(f_{\psi_s}+1)^2}\\
             &\quad +4\psi_s\frac{\nabla^{g_{\psi_s}}f_{\psi_s}\cdot\psi_s}{(f_{\psi_s}+1)^3}.\\
        \end{split}
    \end{equation*}
    Since $|\psi_s|\le D(f_{\psi_s}+1)$, we have
    \begin{equation*}
         \frac{2\psi_s\log\frac{\omega_{\psi_s}^n}{\omega_E^n}-2\psi_s\Delta_{\omega_{\psi_s}}\psi_s}{(f_{\psi_s}+1)^2}\le\frac{2D\left|\log\frac{\omega_{\psi_s}^n}{\omega_E^n}\right|}{f_{\psi_s}+1}+\frac{2D|\Delta_{\omega_{\psi_s}}\psi_s|}{f_{\psi_s}+1}.
    \end{equation*}
    Moreover,
    \begin{equation*}
       4\psi_s\frac{\nabla^{g_{\psi_s}}f_{\psi_s}\cdot\psi_s}{(f_{\psi_s}+1)^3}\le 4D\frac{|\nabla^{g_{\psi_s}}f_{\psi_s}\cdot\psi_s|}{(f_{\psi_s}+1)^2}.
    \end{equation*}
   Putting everything together yields
    \begin{equation*}
         \left(\frac{\partial}{\partial\tau}-\Delta_{\omega_{\psi_s},X}\right)\frac{\psi_s^2}{(f_{\psi_s}+1)^2}\le \frac{2D\left|\log\frac{\omega_{\psi_s}^n}{\omega_E^n}\right|}{f_{\psi_s}+1}+\frac{2D|\Delta_{\omega_{\psi_s}}\psi_s|}{f_{\psi_s}+1}-\frac{|\nabla^{g_{\psi_s}}\psi_s|_{g_{\psi_s}}^2}{(f_{\psi_s}+1)^2}+4D\frac{|\nabla^{g_{\psi_s}}f_{\psi_s}\cdot\psi_s|}{(f_{\psi_s}+1)^2}.
    \end{equation*}
   Again by the Cauchy-Schwarz inequality,
   \begin{equation*}
        \begin{split}
            |\nabla^{g_{\psi_s}}f_{\psi_s}\cdot\psi_s|\le |\nabla^{g_{\psi_s}}f_{\psi_s}|_{g_{\psi_s}}|\nabla^{g_{\psi_s}}\psi_s|_{g_{\psi_s}}&\le \frac{1}{8D}|\nabla^{g_{\psi_s}}\psi_s|_{g_{\psi_s}}^2+2D |\nabla^{g_{\psi_s}}f_{\psi_s}|^2_{g_{\psi_s}}\\
            &\le \frac{1}{8D}|\nabla^{g_{\psi_s}}\psi_s|_{g_{\psi_s}}^2+6D(f_{\psi_s}+1).
        \end{split}
    \end{equation*}
    Therefore, we have 
    \begin{equation*}
        \left(\frac{\partial}{\partial\tau}-\Delta_{\omega_{\psi_s},X}\right)\frac{\psi_s^2}{(f_{\psi_s}+1)^2}\le \frac{2D\left|\log\frac{\omega_{\psi_s}^n}{\omega_E^n}\right|+2D|\Delta_{\omega_{\psi_s}}\psi_s|+24D^2}{f_{\psi_s}+1}-\frac{1}{2}\frac{|\nabla^{g_{\psi_s}}\psi_s|_{g_{\psi_s}}^2}{(f_{\psi_s}+1)^2}.
    \end{equation*}
    Finally, using that $|\Delta_{\omega_{\psi_s}}\psi_s|\le n+\tr_{\omega_{\psi_s}}\omega_E$ we have
    \begin{equation*}
         \left(\frac{\partial}{\partial\tau}-\Delta_{\omega_{\psi_s},X}\right)\frac{\psi_s^2}{(f_{\psi_s}+1)^2}\le \frac{24D^2+2Dn}{f_{\psi_s}+1}+\frac{2D\left|\log\frac{\omega_{\psi_s}^n}{\omega_E^n}\right|+2D\tr_{\omega_{\psi_s}}\omega_E}{f_{\psi_s}+1}-\frac{1}{2}\frac{|\nabla^{g_{\psi_s}}\psi_s|_{g_{\psi_s}}^2}{(f_{\psi_s}+1)^2}.
    \end{equation*}
    Taking $\varepsilon=\frac{1}{8D}$ and $C=4D+n+\frac{3}{4\varepsilon}+2\varepsilon(24D^2+2Dn)$, we have
    \begin{equation*}
         \left(\frac{\partial}{\partial\tau}-\Delta_{\omega_{\psi_s},X}\right)\left(\frac{\psi_s}{f_{\psi_s}+1}-2\varepsilon\frac{\psi_s^2}{(f_{\psi_s}+1)^2}\right)\ge -\frac{C}{f_{\psi_s}+1}+\frac{\log\frac{\omega_{\psi_s}^n}{\omega_E^n}-\frac{1}{2}\left|\log\frac{\omega_{\psi_s}^n}{\omega_E^n}\right|+\frac{1}{2}\tr_{\omega_{\psi_s}}\omega_E}{f_{\psi_s}+1}.
    \end{equation*}
    We can then define $\Theta(\psi_s):=\frac{\psi_s}{f_{\psi_s}+1}-2\varepsilon\frac{\psi_s^2}{(f_{\psi_s}+1)^2}$. Since $|\psi_s|\le D(f_{\psi_s}+1)$, the boundedness of $\Theta(\psi_s)$ follows immediately. Now we show that
    \begin{equation}\label{psotive part of barrier function}
        \log\frac{\omega_{\psi_s}^n}{\omega_E^n}-\frac{1}{2}\left|\log\frac{\omega_{\psi_s}^n}{\omega_E^n}\right|+\frac{1}{2}\tr_{\omega_{\psi_s}}\omega_E\ge \frac{1}{4}\left(\tr_{\omega_{\psi_s}}\omega_E+\left(\log\frac{\omega_{\psi_s}^n}{\omega_E^n}\right)_+\right)-C(n),
    \end{equation}
    holds for some dimensional constant. Then \eqref{equation of barrier function} holds naturally. At points where $\log \frac{\omega_{\psi_s}^n}{\omega_E^n} \geq 0$, the inequality \eqref{psotive part of barrier function} follows. If instead $\log \frac{\omega_{\psi_s}^n}{\omega_E^n} \leq 0$, we can use the Arithmetic Mean-Geometric Mean inequality and the fact that $\sup_{y>0}(-\frac{1}{4}y+\frac{3n}{2}\log y) \leq C(n)$ to get
       \begin{equation*}
        \frac{3}{2}\log \frac{\omega_{\psi_s}^n}{\omega_E^n} =-\frac{3}{2}\log \frac{\omega_E^n}{\omega_{\psi_s}^n} \geq -\frac{3n}{2}\log (\frac{1}{n}\operatorname{tr}_{\omega_{\psi_s}}\omega_E) \geq -\frac{1}{4}\operatorname{tr}_{\omega_{\psi_s}}\omega_E -C(n).
    \end{equation*}
    Then \eqref{equation of barrier function} holds as expected.
    \end{proof}
\begin{lemma}\label{parabolic Schwarz lemmas}
    There exists a constant $A>0$ that only depends on $g_E$ such that
    \begin{equation*}
        \begin{split}
            &\left(\frac{\partial}{\partial\tau}-\Delta_{\omega_{\psi_s},X}\right)\log\tr_{\omega_E}\omega_{\psi_s}\le\frac{A}{f+1}\left(\tr{\omega_{\psi_s}}\omega_E+1\right),\\
             &\left(\frac{\partial}{\partial\tau}-\Delta_{\omega_{\psi_s},X}\right)\log\tr_{\omega_{\psi_s}}\omega_E\le\frac{A}{f+1}\left(\tr{\omega_{\psi_s}}\omega_E+1\right).
        \end{split}
    \end{equation*}
\end{lemma}
\begin{proof}
    The first evolution inequality is a straightforward computation, where we are crucially using that the expanding soliton is asymptotically conical and, therefore, $f|\operatorname{Rm}(g_E)|\leq C(g_E).$ For a detailed proof, see \cite[Lemma 4.18]{longteng2}.

The second inequality is nothing more than the Parabolic Schwarz Lemma adapted to this setting. Recall the modified K\"ahler--Ricci flow equation
 \begin{equation*}
     \frac{\partial}{\partial\tau}\omega_{\psi_s}=\mathcal{L}_{\frac{X}{2}}\omega_{\psi_s}-\omega_{\psi_s}-\Ric(\omega_{\psi_s}).
 \end{equation*}
  To simplify notation, below we use $g=g_E$ and $g_\psi=g_{\psi_s}.$ For the time derivative, we have
      \begin{equation}\label{time derivative of tr wpsi w}
          \frac{\partial}{\partial\tau}\tr_{\omega_{\psi}}\omega=\frac{\partial}{\partial\tau}(g_\psi^{i\bar j}g_{i\bar j})=-\mathcal{L}_{\frac{X}{2}}g_\psi^{i\bar j}g_{i\bar j}+\Ric(g_\psi)^{i\bar j}g_{i\bar j}+g_\psi^{i\bar j}g_{i\bar j}.
      \end{equation}
    On the holomorphic coordinates of $g_\psi$, the Laplacian of $\tr_{\omega_{\psi_s}}\omega$ is given by the following formula:
      \begin{equation}\label{lapalcian of tr wpsi w}
          \begin{split}
              \Delta_{\omega_{\psi}}\tr_{\omega_{\psi}}\omega&=g_\psi^{i\bar j}\partial_i\partial_{\bar j}(g_\psi^{p\bar q}g_{p\bar q})\\
              &=g_\psi^{i\bar j}g_{p\bar q}\partial_i\partial_{\bar j}g_\psi^{p\bar q}+g_\psi^{i\bar j}g_\psi^{p\bar q}\partial_i\partial_{\bar j} g_{p\bar q}\\
              &=g_{p\bar q}\Ric(g_\psi)^{p\bar q}+g^{m\bar s}g_\psi^{i\bar j}g_\psi^{p\bar q}\partial_i g_{p\bar s}\partial_{\bar j}g_{m\bar q}-g_\psi^{i\bar j}g_\psi^{p\bar q}\Rm(g)_{i\bar j p\bar q}.\\
              &=g_{p\bar q}\Ric(g_\psi)^{p\bar q}+g^{m\bar s}g_\psi^{i\bar j}g_\psi^{p\bar q}\nabla^{g_\psi}_i g_{p\bar s}\nabla^{g_\psi}_{\bar j}g_{m\bar q}-g_\psi^{i\bar j}g_\psi^{p\bar q}\Rm(g)_{i\bar j p\bar q}.
          \end{split}
      \end{equation}
      We compute 
      \begin{equation}\label{X derivative of tr wpsi w}
         \begin{split}
              \frac{X}{2}\cdot\tr_{\omega_{\psi}}\omega= \frac{X}{2}\cdot(g_\psi^{i\bar j}g_{i\bar j})&=-\mathcal{L}_{\frac{X}{2}}g_\psi^{i\bar j}g_{i\bar j}+g_\psi^{i\bar j}\mathcal{L}_{\frac{X}{2}}g_{i\bar j}\\
              &=-\mathcal{L}_{\frac{X}{2}}g_\psi^{i\bar j}g_{i\bar j}+g_\psi^{i\bar j}\Ric(g)_{i\bar j}+g_\psi^{i\bar j}g_{i\bar j}.
              \end{split}
      \end{equation}
      Combining \eqref{time derivative of tr wpsi w}, \eqref{lapalcian of tr wpsi w} and \eqref{X derivative of tr wpsi w}, we get
        \begin{equation}\label{eq tr wpsi w}
          \left(\frac{\partial}{\partial \tau}-\Delta_{\omega_{\psi},X}\right)\tr_{\omega_{\psi}}\omega=g_\psi^{i\bar j}g_\psi^{p\bar q}\Rm(g)_{i\bar j p\bar q}-g^{m\bar s}g_\psi^{i\bar j}g_\psi^{p\bar q}\nabla^{g_\psi}_i g_{p\bar s}\nabla^{g_\psi}_{\bar j}g_{m\bar q}-g_\psi^{i\bar j}\Ric(g)_{i\bar j}.
      \end{equation}
      Now, we compute
      \begin{equation*}
          \begin{split}
         \left(\frac{\partial}{\partial \tau}-\Delta_{\omega_{\psi},X}\right)\log\tr_{\omega_{\psi}}\omega& =\frac{ \left(\frac{\partial}{\partial \tau}-\Delta_{\omega_{\psi},X}\right)\tr_{\omega_{\psi}}\omega}{\tr_{\omega_{\psi}}\omega}+\frac{|\partial \tr_{\omega_{\psi}}\omega|_{g_\psi}^2}{(\tr_{\omega_{\psi}}\omega_E)^2}\\ 
         &=\frac{g_\psi^{i\bar j}g_\psi^{p\bar q}\Rm(g)_{i\bar j p\bar q}-g^{m\bar s}g_\psi^{i\bar j}g_\psi^{p\bar q}\nabla^{g_\psi}_i g_{p\bar s}\nabla^{g_\psi}_{\bar j}g_{m\bar q}-g_\psi^{i\bar j}\Ric(g)_{i\bar j}}{\tr_{\omega_{\psi}}\omega}\\
         &\quad +\frac{|\partial \tr_{\omega_{\psi}}\omega|_{g_\psi}^2}{(\tr_{\omega_{\psi}}\omega_E)^2}.
          \end{split}
      \end{equation*}
    From the proof of the Parabolic Schwarz lemma in \cite[Theorem 2.6]{SongWeinkove}, we know that the following inequality always holds:
      \begin{equation*}
          -g^{m\bar s}g_\psi^{i\bar j}g_\psi^{p\bar q}\nabla^{g_\psi}_i g_{p\bar s}\nabla^{g_\psi}_{\bar j}g_{m\bar q}+\frac{|\partial \tr_{\omega_{\psi}}\omega|_{g_\psi}^2}{\tr_{\omega_{\psi}}\omega}\le 0.
      \end{equation*}
      Therefore, we have 
      \begin{equation*}
           \left(\frac{\partial}{\partial \tau}-\Delta_{\omega_{\psi},X}\right)\log\tr_{\omega_{\psi}}\omega\le\frac{g_\psi^{i\bar j}g_\psi^{p\bar q}\Rm(g)_{i\bar j p\bar q}-g_\psi^{i\bar j}\Ric(g)_{i\bar j}}{\tr_{\omega_{\psi}}\omega}.
      \end{equation*}
    Note that in holomorphic coordinates for $g_\psi$, we have
      \begin{equation*}
          g_\psi^{i\bar j}g_\psi^{p\bar q}\Rm(g)_{i\bar j p\bar q}=\sum_{i,p}\Rm(g)_{i\bar i p\bar p}\le \frac{A}{f+1}\sum_{i p}g_{i\bar i}g_{p\bar p}=\frac{A}{f+1} (\tr_{\omega_{\psi}}\omega_E)^2,
      \end{equation*}
     where $A=A(g_E).$ Moreover,
      \begin{equation*}
          -g_{\psi}^{i\bar j}\Ric(g)_{i\bar j}\le \frac{B}{f+1}\tr_{\omega_{\psi}}\omega,
      \end{equation*}
      where $B$ is a constant such that $\Ric(g)+\frac{B}{f+1}g\ge 0$. Thus, after rescaling if necessary, it follows that
      \begin{equation*}
           \left(\frac{\partial}{\partial \tau}-\Delta_{\omega_{\psi},X}\right)\log\tr_{\omega_{\psi}}\omega\le \frac{A}{f+1}(\tr_{\omega_{\psi}}\omega+1).
    \end{equation*}
  \end{proof}

We are now ready to prove our main theorem of the section, regarding the $C^2$-estimates for our solution.
\begin{proof}[Proof of Theorem \ref{C^2 estimate}]For any $0\le T<T'_s$, consider also the compact set 
          \[\Omega_{R,\lambda,s}^T=\{(x,\tau)\ |\ (x,\tau)\in\Omega_{R,\lambda,s},\tau\in [0,T]\}.\]
    Since $f_{\psi_s}+1\le D(f+1),$ it follows from Lemma \ref{parabolic Schwarz lemmas} that 
      \begin{equation*}
        \begin{split}
            &\left(\frac{\partial}{\partial\tau}-\Delta_{\omega_{\psi_s},X}\right)\log\tr_{\omega_E}\omega_{\psi_s}\le\frac{AD}{f_{\psi_s}+1}\left(\tr{\omega_{\psi_s}}\omega_E+1\right),\\
             &\left(\frac{\partial}{\partial\tau}-\Delta_{\omega_{\psi_s},X}\right)\log\tr_{\omega_{\psi_s}}\omega_E\le\frac{AD}{f_{\psi_s}+1}\left(\tr{\omega_{\psi_s}}\omega_E+1\right).
        \end{split}
          \end{equation*}
    Along with Lemma \ref{barrier function}, this implies that 
    \begin{equation*} u:= \log \operatorname{tr}_{\omega_{\psi_s}}\omega_E-5AD\Theta(\psi_s)
    \end{equation*}
    satisfies
    \begin{equation*}
        (\partial_\tau - \Delta_{\omega_{\psi_s},X})u \leq \frac{AD}{f_{\psi_s}+1}\left( -\frac{1}{4}\operatorname{tr}_{\psi_s}\omega_E +C\right).
    \end{equation*}
    If $(x_0,\tau_0) \in \Omega_{R,\lambda,s}^T \setminus \partial_{\textnormal{P}}\Omega_{R,\lambda,s}$ is a maximum point for $u$, it then follows that
    \begin{equation*}
        \operatorname{tr}_{\psi_s}\omega_E \leq C,
    \end{equation*}
    hence $u\leq C$. Since the barrier function is uniformly bounded, we have $\omega_E \leq C\omega_{\psi_s}$ on all of $\Omega_{R,\lambda,s}^T$. Similarly, at a maximum point $\Omega_{R,\lambda,s}^T\setminus \partial_{\textnormal{P}}\Omega_{R,\lambda,s}$ of the quantity
    \begin{equation*}
        v:= \log \operatorname{tr}_{\omega_E}\omega_{\psi_s} - 5AD\Theta(\psi_s),
    \end{equation*}
   we have
   \begin{equation*}
       \operatorname{tr}_{\psi_s}\omega_E + \left( \log \frac{\omega_{\psi_s}^n}{\omega_E^n} \right)_+ \leq C,
   \end{equation*}
   so that 
   \begin{align*}
       \operatorname{tr}_{\omega_E}\omega_{\psi_s} \leq \frac{\omega_{\psi_s}^n}{\omega_E^n}(\operatorname{tr}_{\psi_s}\omega_E)^{n-1} \leq C.
   \end{align*}
    If $u$ or $v$ attains their minimum/maximum on $\partial_{\textnormal{P}}\Omega_{R,\lambda,s}$, the estimates instead follow from Proposition \ref{boundary data} \ref{boundarydata1} and the boundedness of $\Theta(\psi_s)$.
\end{proof}

\subsection{Higher order and improved estimates}
It is then standard to obtain higher order estimates for our solution. We will use these to improve our original $C^2$-estimate via an interpolation argument. 

Let $R_0, s_0, \lambda_0 > 0$ be as in Proposition \ref{control f with fpsi} and Proposition \ref{control fpsi with f}, and choose parameters $R, s, \lambda > 0$ satisfying \eqref{eq:backgroundassumption}.
\begin{theorem}[$C^3-$estimate]\label{thm: C^3 estimates}
    There exists a uniform constant $C>0$ such that on $\Omega_{R,\lambda,s}$
    \begin{equation*}
   f |\nabla^{g_E}{g_{\psi_s}}|_{g_E}^2\le C
    \end{equation*}
\end{theorem}
\begin{proof}
The proof is the same as in \cite[Proposition 4.21]{longteng2}, in light of Proposition \ref{control f with fpsi} and Proposition \ref{control fpsi with f}. 
\end{proof}
We can use the theorem above to prove the following stronger $C^2$-estimate.

\begin{theorem}\label{interpolation inequality}
    Setting $T''_s=\min\{\log\left(\frac{T_s}{s}+1\right),\log\frac{R^2}{2\lambda s}\}$, we have 
    \begin{equation}\label{eq: interpolation for C^2}
        \sup_{\tau \in [0,T_s'')} \sup_{\mathcal{B}(\frac{R}{\sqrt{s}e^{\frac{\tau}{2}}})} \sum_{j=0}^2 f^{\frac{j}{2}-1}|(\nabla^{g_E})^j \psi_s(\tau)|_{g_E}\le \Psi(R,\lambda^{-1}).
    \end{equation}
\end{theorem}
\begin{proof} Fix $\tau \in [0,T_s'')$. By Proposition \ref{boundary data}  \ref{boundarydata1}, \ref{boundarydata3} and standard Schauder estimates, it suffices to prove \eqref{eq: interpolation for C^2} on $\mathcal{B}(\frac{R}{\sqrt{2s}e^{\frac{\tau}{2}}})$. For any $\tau\in [0,T_s'')$, we set
\begin{equation}\label{eq: poisson eq}
    u:=\Delta_{g_E}\psi_s,
\end{equation}
on $\mathcal{B}(\frac{R}{\sqrt{2s}e^{\frac{\tau}{2}}})$.  By Theorem \ref{C^2 estimate} and \ref{thm: C^3 estimates}, there is a uniform constant $A>0$ such that
\begin{equation} \label{eq:interp1}
  \sup_{\mathcal{B}(\frac{R}{\sqrt{s}e^{\frac{\tau}{2}}})}  \left(|u|+\sqrt{f}|\nabla^{g_E}u|_{g_E} \right)\le A.
\end{equation}
Fix $x\in\{r^2\le \frac{R^2}{se^\tau}\}$ and consider the geodesic ball $B_{g_E}(x,\delta\sqrt{f(x)+1})$ with $\delta<\min\{\frac{1}{4},\frac{1}{2}\delta_0\}$ as in Proposition \ref{linear growth of injectivity radius}. 
    For any $y\in B_{g_E}(x,\delta\sqrt{f(x)+1})$, let $\gamma$ be the $g_E-$geodesic connecting $x$ and $y$. Then we have
    \begin{equation*}
        \begin{split}
            \left|\sqrt{f(x)+1}-\sqrt{f(y)+1}\right|&=\left|\sqrt{f(\gamma(0))+1}-\sqrt{f(\gamma(1))+1}\right|\\
            &\le\int_0^1\frac{1}{2}\frac{|\nabla^{g_E}f|_{g_E}|\dot\gamma|_{g_E}}{\sqrt{f(\gamma (t))+1}}dt\\
            &\le\frac{\sqrt{2}}{2}d_{g_E}(x,y)<\delta\sqrt{f(x)+1}.
        \end{split}
    \end{equation*}
    Therefore,
   \begin{equation*}
       \frac{1}{2}\sqrt{f(x)+1}\le\sqrt{f(y)+1}\le \frac{5}{4}\sqrt{f(x)+1}\le 2\sqrt{f(x)+1},
   \end{equation*}
   and in particular,
   \begin{equation*}
       \frac{r(y)^2}{2}\le \frac{25}{16}({f(x)+1})\le \frac{25}{16}\left(\frac{r(x)^2}{2}+C+1\right)= \frac{25}{32}r(x)^2+C',
   \end{equation*}
   where $C>0$ is such that $f\le\frac{r^2}{2}+C$ holds on $E$. By taking $\lambda>1$ sufficiently large and $s$ sufficiently small such that $\frac{25}{32}r(x)^2+C'\le\frac{R^2}{2se^\tau}$, we have that $y\in\{r^2\le\frac{R^2}{se^\tau}\}$. Hence $B_{g_E}(x,\delta\sqrt{f(x)+1})\subset\{r^2\le\frac{R^2}{se^\tau}\}$. For instance, we can take $\lambda>0$ such that $\lambda\ge 100C'$. In this case, $r(x)^2\le \lambda\le\frac{R^2}{2se^\tau}$ and $s\lambda\le R^2$, it follows that
\begin{equation*}
    \begin{split}
       & \frac{25}{32}r(x)^2+C'\le\frac{25}{32}\frac{R^2}{2se^\tau}+C'\le \frac{R^2}{2se^\tau},\\
    \end{split}
\end{equation*}
where the last inequality follows from the fact that $\frac{7}{64}\frac{R^2}{se^\tau}\ge\frac{1}{10} \frac{R^2}{se^{T_s''}}\ge\frac{\lambda}{5} $. Moreover, in this geodesic ball, for all $k\in \N$,
\begin{equation*}
    \sup_{B_{g_E}(x,\delta\sqrt{f(x)+1})} (f(x)+1)^{\frac{k}{2}+1}|(\nabla^{g_E})^k\Rm(g_E)|_{g_E}\le A_k,
\end{equation*}
holds for some uniform constant $A_k>0$. Set $\widetilde{\psi}:=(f(x)+1)^{-1}\psi_s$ and $\widetilde{g}:=(f(x)+1)^{-1}g_E$, so that \eqref{eq: poisson eq} becomes
\begin{equation}\label{eq: new poisson eq}
    \Delta_{\widetilde{g}}\widetilde{\psi}=u,
\end{equation}
and \eqref{eq:interp1} becomes
\begin{equation}\label{wq: new holder bound}
    \sup_{B_{\widetilde{g}}(x,\delta)}|u|+|\nabla^{\widetilde{g}}u|_{\widetilde{g}}\le A, \quad  \sup_{B_{\widetilde{g}}(x,\delta)}|\widetilde{\psi}|\le\Psi(R,\lambda^{-1})
\end{equation}
On the ball $B_{\widetilde{g}}(x,\delta)$, we have
\begin{equation*}
    \sup_{B_{\widetilde{g}}(x,\delta)} |(\nabla^{\widetilde{g}})^k\Rm(\widetilde{g})|_{\widetilde{g}}\le A_k,\quad \textnormal{for all $k\in\N$}.
\end{equation*}
Then for any choice of $\alpha \in (0,1)$, standard interior Schauder estimates for elliptic equations (see \cite[Theorem 6.2]{GilbargTrudinger}) give
\begin{equation*}
   \|\widetilde{\psi}\|_{C^{2,\alpha}(B_{\widetilde{g}}(x,\delta/2),\widetilde{g})}\le C(\alpha,\delta,A,A_0,A_1)\left( \sup_{B_{\widetilde{g}}(x,\delta)}|\widetilde{\psi}|+ \|u\|_{C^{0,\alpha}(B_{\widetilde{g}}(x,\delta),\widetilde{g})}\right)\le C(\alpha,\delta,A,A_0,A_1),
\end{equation*}
where $\|\cdot\|_{C^{k,\alpha}(\cdot,\widetilde{g})}$ denotes the $(k,\alpha)$-H\"older norm with respect to $\widetilde{g}$. By the general interpolation inequality \cite[p.141]{GilbargTrudinger} for H\"older norms, we get
\begin{equation}
    |(\nabla^{\widetilde{g}})^2\widetilde{\psi}|_{\widetilde{g}}(x)\le C\left(\sup_{B_{\widetilde{g}}(x,\delta/2)}|\widetilde{\psi}|\right)^\theta\|\widetilde{\psi}\|_{C^{2,\alpha}(B_{\widetilde{g}}(x,\delta/2),\widetilde{g})}^{1-\theta} \leq C(\alpha,\delta,A,A_0,A_1)\Psi(R,\lambda^{-1})^\theta,
\end{equation}
where $\theta=\frac{\alpha}{2+\alpha}$. By rewriting this in terms of $g_E$ and taking $\alpha=\frac{2}{3}$, we get that
\begin{equation*}
     \sum_{j=0}^2(1+f)^{\frac{j}{2}-1}|(\nabla^{g_E})^j \psi_s|_{g_E}(x)\le C\Psi(R,\lambda^{-1})^{\frac{1}{4}}
\end{equation*}
holds as expected.
\end{proof}

\subsection{Curvature control and local stability}
We now use the estimates from the previous subsections to prove space-time curvature bounds for our solutions. Let $R_0, s_0, \lambda_0 > 0$ be as in Proposition \ref{control f with fpsi} and Proposition \ref{control fpsi with f}, and choose parameters $R, s, \lambda > 0$ satisfying \eqref{eq:backgroundassumption}.

By \eqref{comparison of f and r^2} and Theorem \ref{interpolation inequality}, on the space-time region $\{r_M\le 2R\}\times (0,\min\{T_s,\frac{R^2}{2\lambda}-s\})$, we have
\begin{equation}\label{eq: initial C^2 closedness}
    \sum_{j=0}^2(r_M+\sqrt{t+s})^{j-2}|(\nabla^{g_E(t+s)})^j\varphi_s(t)|_{g_E(t+s)]}\le\Psi(R,\lambda^{-1}).
\end{equation}

 \begin{prop}\label{prop: BK bound}
  For each $k\in\N$, there exist $s_0(k),R_0(k),\lambda_0(k)>0$ such that for all $R \in (0, R_0(k)]$, $\lambda\in [\lambda_0(k),\infty)$ and $s\in (0,\min\{s_0(k),\frac{R^4}{4},\frac{1}{\lambda^2}\})$, the following holds on $\mathcal{B}(R)\times (0,\min\{T_s,\frac{R^2}{4\lambda}\})$:
    \begin{equation}\label{eq: BK induces C^k bound}
       \sum_{j=0}^k (r_M+\sqrt{t})^{j}|(\nabla^{g_E(t+s)})^j\left(g_s(t)-g_E(t+s)\right)|_{g_E(t+s)}\le \Psi(R,\lambda^{-1}|k).
    \end{equation}
\end{prop}

\begin{proof}
    The metrics $g_s(t)$ and $g_{E}(t+s)$ are K\"ahler with respect to the same fixed complex structure, which implies that the identity map between them is holomorphic and, in particular, harmonic. In other words, the K\"ahler-Ricci flow $(g_s(t))_{t\in (0,\min\{T_s,\frac{R^2}{2\lambda}-s\}) }$ is a solution to Ricci-DeTurck flow equation with reference flow $(g_E(s+t))_{t\in (0,\min\{T_s,\frac{R^2}{2\lambda}-s\}) }$. 

    Fix $x\in \mathcal{B}(R)$
    and note that $B_{g_E(s)}(x,\frac{1}{2}r(x))\subset \mathcal{B}(2R)$ if $s>0$ is sufficiently small. By \eqref{eq:quadraticdecayofexpander}, we have
    \begin{equation}\label{eq: curvature bound in prop 4.20}
        \sup_{B_{g_E(s)}(x,\frac{1}{2}r(x))}|(\nabla^{g_E(t+s)})^k\Rm(g_E(t+s))|\le A_kr(x)^{-2-k},
    \end{equation}
    for some uniform constants $A_k \in (0,\infty)$.
    Using Proposition \ref{gluing metric; prop; 2}, Proposition \ref{gluing metric; prop}, \eqref{eq: curvature bound in prop 4.20} and \cite[Lemma A.2(b)]{BKActa}, we conclude that for all $k\in\N$
\begin{equation}\label{eq: BK bound 1}
    \sup_{t\in [0,\frac{1}{2}r(x)^2]}\sum_{j=0}^kr(x)^j|(\nabla^{g_E(t+s)})^j\left(g_s(t)-g_E(t+s)\right)|_{g_E(t+s)}(x)\le C_kH_k.
\end{equation}
Here $C_k>0$ is the dimensional constant appearing in \cite[Lemma A.2]{BKActa} and 
\begin{equation*}
    H_k:=\sup_{B_{g_E(s)}(x,\frac{1}{2}r(x))}|g_s(t)-g_E(t+s)|_{g_E(t+s)}+\max_{0\le l\le k+1}\sup_{B_{g_E(s)}(x,\frac{1}{2}r(x))}r(x)^l|(\nabla^{g_E(s)})^l\left(g_{s,0}-g_E(s)\right)|_{g_E(s)}.
\end{equation*}
By Proposition \ref{gluing metric; prop; 2}, \ref{gluing metric; prop}, we have
\begin{equation*}
    H_k \leq \Psi(R,\lambda^{-1}|k),
\end{equation*}
so that \cite[Lemma A.2]{BamlerKleiner} applies if we choose $R,\lambda^{-1}$ small enough such that $H_k\le \eta_k$ as in \cite[Lemma A.2]{BKActa}.

Fix $t\in (0,\min\{T_s,\frac{R^2}{2\lambda}-s\})$ such that $r(x)^2\le 2t$ and consider the geodesic ball $B_{g_E(\frac{t}{2}+s)}(x, \sqrt{t})$. Then $B_{g_E(\frac{t}{2}+s)}(x, \sqrt{t})\subset\{r^2\le \lambda t\}$ if $\lambda>1$ is sufficiently large. Moreover, on $B_{g_E(\frac{t}{2}+s)}(x, \sqrt{t})$, we have
\begin{equation}\label{eq: curvature bound in prop 4.20 2}
     \sup_{B_{g_E(\frac{t}{2}+s)}(x,\sqrt{t})}|(\nabla^{g_E(t+s)})^k\Rm(g_E(t+s))|\le A_k(t+s)^{-1-\frac{k}{2}}\le A_k t^{-1-\frac{k}{2}},
\end{equation}
 holds for uniform constants $A_k \in (0,\infty)$. Applying \eqref{eq: curvature bound in prop 4.20 2} and \cite[Lemma A.2 (a)]{BamlerKleiner}, we get
 \begin{equation}\label{eq: BK bound 2}
     |(\nabla^{g_E(t+s)})^j\left(g_s(t)-g_E(t+s)\right)|_{g_E(t+s)}(x)\le C_jHt^{-\frac{j}{2}},\quad \forall j\in\N
 \end{equation}
 where $C_j \in (0,\infty)$ are dimensional constants and 
 \begin{equation*}
     H:=\sup_{B_{g_E(\frac{t}{2}+s)}(x,\sqrt{t})}|g_s(t)-g_E(t+s)|_{g_E(t+s)}\le \Psi(R,\lambda^{-1}).
 \end{equation*}
Then \eqref{eq: BK induces C^k bound} follows directly from \eqref{eq: BK bound 1} and \eqref{eq: BK bound 2}.
\end{proof}
Take $R_0(2),\lambda_0(2),s_0(2)>0$ as in Proposition \ref{prop: BK bound}, and choose parameters $R, s, \lambda > 0$ satisfying 
\begin{equation*}
     R \le R_0(2), \quad \lambda \ge \lambda_0(2), \qquad  s \le \min \left\{s_0(2),\frac{1}{\lambda^2},\frac{R^4}{16}\right\}.
\end{equation*}
Using Proposition \ref{prop: BK bound} and Shi's local derivative estimates yields the following curvature bound on a spacetime neighborhood of the singular point

\begin{prop}\label{prop: boundedness of curvature}
    There exist $C_k \in (0,\infty)$ such that for any $s\in (0,\min\{s_0,\frac{R^4}{16}\})$, we have
     \begin{equation}\label{eq: BK induces C^k bound rough}
        \sup_{t\in (0,\min\{T_s,\frac{R^2}{4\lambda}\})}\sup_{\mathcal{B}(R)}(r_M+\sqrt{t})^{2+k}|(\nabla^{g_s(t)})^k\Rm(g_s(t))|_{g_s(t)}\le C_k,\quad \forall k\in\N.
    \end{equation}
\end{prop}
\begin{proof}
  The proof is the same as in \cite[Corollary 5.6]{longteng2}
\end{proof}

\begin{prop} \label{prop:curvatureawayfromsings}
    There exist constants $C=C(R,g_0)>0$ and $T=T(R,g_0)>0$ such that
    \begin{equation*}
         |\Rm(g_{s}(t))|_{g_s(t)}(x)\le \frac{C}{t}
    \end{equation*}
    for all $(x,t)\in (M\setminus\{r_M\le R\})\times (0,\min \{T(R,g_0),T_s,\frac{R^2}{4\lambda}\})$
\end{prop}

\begin{proof}
    For the initial data, we have 
    \begin{equation*}
        \omega_{s,0}|_{M\setminus\{r_M\le R\}}=\omega_E(s)+\sqrt{-1}\partial\bar\partial\left(\chi\left(\frac{r(\cdot)}{s^{\frac{1}{4}}}\right)(u_{1,s}-u_E(s))\right)= \omega_0-s\Ric(\omega_0).
    \end{equation*}
    A similar computation to the one in Proposition \ref{gluing metric; prop} shows that the curvature of $ \omega_{s,0}$ is uniformly bounded on $M\setminus\{r^2\le R^2\}$ for all $s_0\ge s>0$.
    Then, letting $C_1:=\sup_{M\setminus\{r^2\le R^2\}}|\Rm(g_{s,0})|_{g_{s,0}}$ and recalling the evolution equation for the norm of the curvature tensor along K\"ahler--Ricci flow:
    \begin{equation*}
        \left(\frac{\partial}{\partial t}-\Delta_{\omega_s(t)}\right)|\Rm(g_{s}(t))|_{g_s(t)}\le C|\Rm(g_{s}(t))|_{g_s(t)}^2,
    \end{equation*}
    we can use the maximum principle \cite{ChowLuNi} to obtain the result, where Proposition \ref{prop: boundedness of curvature} is used to control $|\Rm(g_s(t))|_{g_s(t)}$ on the boundary.
\end{proof}

\begin{corollary}
   The maximum existence time $T_s$ satisfies $T_s > T_0:=\min\{T(R,g_0),\frac{R^2}{4\lambda}\}$.
\end{corollary}
\begin{proof}
Recall we defined $T_s$ as the maximum existence time of the K\"ahler--Ricci flow. Therefore, we must have
    \begin{equation*}
        \limsup_{t\to T_s}\sup_M|\Rm(g_s(t))|_{g_s(t)}=\infty.
    \end{equation*}
    Along with Proposition \ref{prop:curvatureawayfromsings} and Proposition \ref{prop: boundedness of curvature}, this implies $T_s\ge T_0$.
\end{proof}

\section{Flowing metrics with conical singularities: Proof of Theorem \ref{main theorem}}\label{sectn; constructing the flow}

The aim of this section is to prove Theorem \ref{main theorem} in the case of one conical singularity at $y_1\in Y$ modeled on a K\"ahler cone $(\mathcal{C},g_{\mathcal{C}})$ with smooth canonical model. Since the arguments are local, the case of more than one singularity can be treated similarly. 

Theorem \ref{main theorem}, \ref{thm:mainthm2} follows from the existence of a smooth K\"ahler resolution established in Proposition \ref{prop: existence of Kahler resolution}, together with Proposition \ref{gluing metric; prop; 2} and Proposition \ref{gluing metric; prop}. In Subsection \ref{subsection: existence of KRF}, we establish the existence of the K\"ahler--Ricci flow together with the corresponding curvature bounds appearing in Theorem \ref{main theorem}. The proofs of Theorem \ref{main theorem},\ref{thm:mainthm1} and Theorem \ref{main theorem}, \ref{thm:mainthm4} are given in Subsections \ref{subsection: GH} and \ref{subsection: CG}, respectively. Finally, in Subsection \ref{subsection: uniqueness}, we prove Theorem \ref{main theorem}, \ref{thm:mainthm5} and explain its relation to the Song-Tian solution.

\subsection{Taking the limit}\label{subsection: existence of KRF}
We start by providing an overview of the key estimates proved in the previous section. We have shown the existence of constants $R_0,T_0,C_M=C(R_0,g_0),\{C_k\}_{k\in\N},s_0>0$ such that for all $s\le s_0$, there exists a smooth K\"ahler--Ricci flow $g_s(t)_{t\in [0,T_0]}$ starting from $g_{s,0}$ with the following properties:
\begin{equation}\label{C/t bound for curvature}
    \max_{M}|\Rm(g_s(t))|_{g_s(t)}\le \frac{C_M}{t},\quad \textnormal{for $t\in[0,T_0]$},
\end{equation}
\begin{equation}\label{radial decay for curvature}
   \sup_{t\in(0,T_0]} \max_M (r_M+\sqrt{t})^{2+k}|(\nabla^{g_s(t)})^k\Rm(g_s(t))|_{g_s(t)}\le C_{k},\quad \textnormal{for $t\in[0,T_0]$, $k\in\N$},
\end{equation}
where the radial function $r$ is defined by \eqref{radial function}.

In fact, we have the following result, that says that the estimates above improve in smaller scales.
\begin{lemma}\label{closeness to expander improves in small scales}
    For every $\varepsilon>0$ and integer $k\ge 0$, there exist positive parameters $R(\varepsilon,k),s_0(\varepsilon,k)>0$ small and $\lambda(\varepsilon,k)>0$ large such that for all $R\le R(\varepsilon,k),\lambda\ge \lambda(\varepsilon,k)$, and $s\le \min \{\frac{R^4}{16},\frac{1}{\lambda^2},s_0\},$ we have
    \begin{equation*}
        (r_M+\sqrt{t})^{j}\left|(\nabla^{g_E(t+s)})^j(g_s(t)-g_E(t+s))\right|_{g_E(t+s)}\le \varepsilon\quad\textnormal{for all $j\le k$}
    \end{equation*}
    on $\{r_M\le R\} \times [0,\min\{\frac{R^2}{4\lambda},T_0\}).$
\end{lemma}
\begin{proof}
    The proof follows directly from Proposition \ref{prop: BK bound}.
\end{proof}
Moreover, as a consequence of Corollary \ref{coro C^0 bound}, Proposition \ref{prop: BK bound} and standard Schauder estimates, we obtain the following estimates for the K\"ahler potential.
\begin{corollary}\label{coro: decay of varphi_s}
    For every $\varepsilon>0$ and integer $k\ge 0$, there exist positive parameters $R(\varepsilon,k),s_0(\varepsilon,k)>0$ small and $\lambda(\varepsilon,k)>0$ large such that whenever $R\le R(\varepsilon,k),\lambda\ge \lambda(\varepsilon,k)$, and $s\le \min \{\frac{R^4}{16},\frac{1}{\lambda^2},s_0\},$ we have
    \begin{equation*}
        (r+\sqrt{t+s})^{j-2}\left|(\nabla^{g_E(t+s)})^j\varphi_s(t)\right|_{g_E(t+s)}\le \varepsilon\quad\textnormal{for all $j\le k$}
    \end{equation*}
    on $\{r^2\le \frac{R^2}{2}\} \times [0,\frac{R^2}{4\lambda}].$
\end{corollary}
\begin{proof}
   The $C^0$ estimate for $\varphi_s$ follows from Corollary~\ref{coro C^0 bound}. Lemma~\ref{closeness to expander improves in small scales} yields higher-order estimates for $\Delta_{g_E(t+s)}\varphi_s(t)$. The corresponding higher-order estimates for $\varphi_s(t)$ then follow from standard Schauder estimates.

\end{proof}
It then follows from Proposition \ref{prop: BK bound} and \ref{prop: boundedness of curvature} that the approximating solutions $g_s(t)$ satisfy
\begin{equation*}
C^{-1}\tilde g(t) \le g_s(t) \le C\tilde g(t)
\end{equation*}
on $M\times (0,T_0],$ where $C>0$ is a uniform constant and $\tilde g(t)$ is a smooth metric interpolating between $g_E(t)$ on the expander region and $g_0$ outside of it, with bounded curvature for every $t\in (0,T_0].$ By \cite[Theorem 3.2.15]{BEG}, we can therefore pass to a subsequence to ensure convergence to a limit Riemannian metric $g(t)$ after letting $s\searrow 0.$ Since we also have analogous higher derivative bounds, $g(t)$ is smooth and is a solution to the K\"ahler--Ricci flow on $M\times (0,T_0].$ Finally, since $(M,g_s(t))_{t\in(0,T_0]}$ satisfies \eqref{C/t bound for curvature}, it follows from the smooth convergence that the limit solution $g(t)$ also satisfies
\begin{equation}\label{C/t curvature bound for limit}
    |\Rm(g(t))|_{g(t)}\le\frac{C_M}{t},
\end{equation}
\begin{equation}\label{spatial decay of curvature for the limit}
   \sup_{t\in (0,T_0]} \max_M(r_M+\sqrt{t})^{2+k}|(\nabla^{g(t)})^k\Rm(g(t))|_{g(t)}\le C_k
\end{equation}
on $M\times (0,T_0]$. By Corollary~\ref{coro: decay of varphi_s}, we may assume that $g_s(t)$ converges to $g(t)$ at the level of K\"ahler potentials: that is, on $\{r_M^2\leq r_0^2\}\times(0,T_0]$, there exists a smooth function $\varphi(t)$ which is a limit of functions $(\varphi_s(t))_{t\in (0,T_0]}$ satisfying
\begin{equation*}
\omega(t)=\omega_E(t)+\sqrt{-1}\partial\bar\partial\varphi(t).
\end{equation*}
Moreover, since $[\omega_s(t)]=[\omega_0]-(s+t)c_1(M)$ by \eqref{eq: cohomological class of w_s,0}, we conclude that
\begin{equation*}
    [\omega(t)]=[\omega_0]-tc_1(M),\quad \lim_{t\to0^+}[\omega(t)]=[\omega_0].
\end{equation*}
\subsection{Convergence to the initial data}\label{subsection: GH}
In this section, we prove two kinds of convergences of $(M,g(t))$ to the singular initial data $Y.$ Firstly, we show that our limit solution, $g(t),$ converges smoothly uniformly to $g_0$ outside the singular point. Define $$Y_l:=Y\setminus\{r^2< 4\sqrt{s_l}\}\hookrightarrow M.$$ 
Here $\{s_l\}_{l\in\N}$ is a positive sequence which tends to 0 such that $g_l(t):=g_{s_l}(t)$ converges to $g(t)$ locally smoothly on $M\times (0,T_0]$.

We argue that for fixed $l_0>0$, $g(t)$ converges smoothly uniformly to $g_0$ on $Y_{l_0}$ as $t$ tends to $0$. For any $k\in\N$, for any fixed $t>0$, we have seen that $g_l(t)$ converges uniformly smoothly to $g(t)$ on $M$. It follows that on $Y_{l_0}$, for any $\varepsilon>0$, there exists a $L=L(k,\varepsilon,t)>0$ such that for all $l\ge L,$
\begin{equation*}
    |g_l(t)-g(t)|_{C^k(g_0)}\le \varepsilon.
\end{equation*}
Moreover, the curvature bound \eqref{radial decay for curvature} implies that for all $j\in\N,$ and $l\ge l_0$, there exists a constant $C_{j,l_0}$ such that on $Y_{l_0}$,
\begin{equation*}
    |(\nabla^{g_l(t)})^j\Rm(g_l(t))|_{g_l(t)}\le C_{j,l_0}
\end{equation*}
holds for all $t\in [0,T_0]$. In particular, from the Ricci flow equation we get that there exists a constant $C=C(l_0,k)>0$ such that on $Y_{l_0},$ 
\begin{equation*}
    |g_l(t)-g_l(0)|_{C^k(g_0)}\le Ct,
\end{equation*}
for all $t\in [0,T_0].$ Since $g_l:=g_l(0)$ converges smoothly uniformly to $g_0$ on $Y_{l_0}$ as $l$ tends to $\infty$, there also exists an $L'=L'(l_0,\varepsilon,k)>0$ such that for all $l\ge L'$, we have 
\begin{equation*}
    |g_l(0)-g_0|_{C^k(g_0)}\le \varepsilon
\end{equation*}
on $Y_{l_0}.$ It then follows from the triangle inequality that
\begin{equation*}
    |g(t)-g_0|_{C^k(g_0)}\le 2\varepsilon+C(k,l_0)t
\end{equation*}
holds on $Y_{l_0}$ for all $\varepsilon>0$, which is enough to obtain the uniform smooth convergence of $g(t)$ to $g_0$ on $Y_{l_0}$ as $t$ tends to $0$. 
Similarly, one can prove that the K\"ahler potential $\varphi(t)$ converges to $u_1$ locally smoothly on $\{0<r^2\le r_0^2\}$ as $t\to 0^+.$

By Lemma \ref{closeness to expander improves in small scales} and the argument of \cite[Section 5]{Gianniotis-Schulze}, we get that for every $\varepsilon>0,$ the resolution map $\pi: M\to Y$ is an $\varepsilon-$isometry between $(Y,d_Y)$ and $(M,d_{g(t)})$ for small $t$. In particular, $(M,d_{g(t)})$ converges to $(Y,d_Y)$ in the Gromov--Hausdorff topology as $t\to 0^+$.

\subsection{Tangent flow at singular points}\label{subsection: CG}
In this section we prove that, for any $p\in  \pi^{-1}(y_1)$, the rescaled pointed K\"ahler--Ricci flow $(M,\tau^{-1}g(\tau t),p)_{t\in (0,\tau^{-1}T_0]}$ converges to $(E,g_E(t),q)_{t\in (0,\infty)}$ in the smooth pointed Cheeger--Gromov topology. First, we need the following result. By taking $s\searrow 0$ on the estimates from Lemma \ref{closeness to expander improves in small scales}, we obtain
\begin{corollary}\label{closeness to expander improves in small scales for the limit}
    For every $\varepsilon>0$ and integer $k\ge 0$, there exist positive parameters $R(\varepsilon,k)$ small and $\lambda(\varepsilon,k)>0$ large such that for all $R\le R(\varepsilon,k),\lambda\ge \lambda(\varepsilon,k)$, we have
    \begin{equation*}
        t^{\frac{j}{2}}\left|(\nabla^{g_E(t)})^j(g(t)-g_E(t))\right|_{g_E(t)}\le \varepsilon\quad\textnormal{for all $j\le k$}
    \end{equation*}
    on $\mathcal{B}(R) \times (0,\frac{R^2}{4\lambda}]$.
\end{corollary}    

Moreover, Corollary~\ref{coro: decay of varphi_s} yields the following estimates.
\begin{prop} \label{prop:citablepotentialbound}
    For every $\varepsilon>0$ and integer $k\ge 0$, there exist positive parameters $R(\varepsilon,k)$ small and $\lambda(\varepsilon,k)>0$ large such that for all $R\le R(\varepsilon,k),\lambda\ge \lambda(\varepsilon,k)$, we have
    \begin{equation*}
        (r_M+\sqrt{t})^{j-2}\left|(\nabla^{g_E(t)})^j\varphi(t)\right|_{g_E(t)}\le \varepsilon\quad\textnormal{for all $j\le k$}
    \end{equation*}
    on $\mathcal{B}(\frac{R}{\sqrt{2}}) \times (0,\frac{R^2}{4\lambda}]$. Moreover, $\varphi(0)=u_1$ on $\{r^2 \leq \frac{R^2}{2}\}$.
\end{prop}

By the $\frac{C}{t}$ curvature bound, and by the forwards uniqueness \cite{ChenZhu} and backward uniqueness \cite{Kotschwar1} of Ricci flow, it suffices to show that $(M,t^{-1}g(t),p)$ converges to $(E,g_E,q)$ in the pointed Cheeger--Gromov sense as $t\to 0^+$.  
Since $\Phi_t(p)\in \pi^{-1}(y_1)$ and $\pi^{-1}(y_1)$ is compact and invariant under the flow of $X$, there exist $t_m\to 0^+$ and $q\in \pi^{-1}(y_1)$ such that $\Phi_{t_m}(p)\to q.$ Moreover, $(E,g_E,X)$ is real analytic in normal coordinates \cite[Corollary 1.3]{Kotschwar2}, and so is the potential function $f.$ Since $(\Phi_{-e^{-\tau}})_{\tau \in \mathbb{R}}$ is the flow of $\frac{1}{2}\nabla^{g_E}f,$ it follows from the argument in \cite[Theorem 2.2]{AMA-DescentMethods} applied to $-f$ in real analytic coordinates near $q$ that $\Phi_t(p)\to q$ as $t\to 0^+.$

For any $\varepsilon>0,k\in\N$, Corollary \ref{closeness to expander improves in small scales for the limit} tells us that there exists $R(\varepsilon,k)>0$ small and $\lambda(\varepsilon,k)>0$ large such that for all $R\le R(\varepsilon,k)$, $\lambda>\lambda(\varepsilon,k)$, we have
\begin{equation*}
    t^{\frac{j}{2}}|(\nabla^{g_E(t)})^j(g_E(t)-g(t))|\le \varepsilon.\quad\textnormal{for all $j\le k$}
\end{equation*}
on $\{r_M^2\le R^2\}$, for all $0<t\le\frac{R^2}{4\lambda}.$ For any $\lambda_0>0$, define the rescaling
\[
\Phi_{t^{-1}}:\{r^2\le\lambda_0\}\subset E\to\{r_M^2\le \lambda_0t\}\subset  (M,t^{-1}g(t)),
\]
where, as before, $\Phi_t$ is the flow of $-\frac{X}{2t}$ for all $t>0$, and we did not distinguish the radial function defined on $E$ and the radial function defined on $M$. In this case, for all $\varepsilon>0,k\in\N$, and for all $t$ such that $t\le\min\{\frac{R^2(\varepsilon,k)}{4\lambda(\varepsilon,k)},\frac{R^2(\varepsilon,k)}{\lambda_0}\}$, it follows that $\{r_M^2\le \lambda_0t\}\subset\{r_M^2\le R^2(\varepsilon,k)\}$, and also we have 
for all $j\le k$, on the region $\{r^2\le\lambda_0\}$ on $E$,
\begin{equation*}
    |(\nabla^{g_E})^j(g_E-t^{-1}\Phi_{t^{-1}}^*g(t))|_{g_E}\le\varepsilon.
\end{equation*}

This implies that $(M,t^{-1}g(t),p)$ converges under the smooth pointed Cheeger--Gromov topology to $(E,g_E,q)$, with $q$ on the exceptional set. 

\subsection{Uniqueness and relation to Song--Tian solutions}\label{subsection: uniqueness}

In this section, we show that the solutions constructed in Theorem \ref{main theorem} are uniquely determined by their initial data, and coincide with previously constructed examples. For simplicity, we denote $\omega(t)$ as $\omega_t$.

Let $(\widetilde{M},\tilde g(t))_{t\in (0,T]}$ be the K\"ahler-Ricci flow on a compact K\"ahler manifold $\widetilde{M}$ as in Theorem \ref{main theorem} \ref{thm:mainthm5}, and suppose there is a modification $\widetilde{\pi}:\widetilde{M}\to Y$  with $\widetilde{\pi}^{\ast}[\omega_0]=\lim_{t\to 0^+}[\widetilde{\omega}(t)]$ and such that $\widetilde{\pi}_*\tilde g(t)$ converges to $g_0$ locally smoothly away from the singular set. Then, for sufficiently small positive $t$, $\widetilde\pi^*[\omega_0]+t[K_{\widetilde{M}}]$ is ample. By \cite[Theorem 1.1]{Collins-Tosatti}, it follows that $K_{\widetilde M }$ is $\widetilde \pi-$ample. Consequently, $\widetilde{M}$ is the smooth canonical model of $Y$. By the uniqueness of smooth canonical models, there exists a biholomorphism $H:M\mapsto \widetilde{M}$ such that $\pi=\widetilde{\pi}\circ H$. We then identify $M$ and $\widetilde{M}$ via the biholomorphism $H$.

Because the only singular points of $Y$ are biholomorphic to K\"ahler cones with smooth links, $Y$ is a normal analytic space \cite[Theorem 1.8]{ConlonHeinI}. We now recall some basic definitions concerning K\"ahler geometry on such spaces. For more details, the reader is referred to \cite[Section 16.3]{CMA}.

\begin{definition} (c.f. \cite[Definitions 16.35-16.38]{CMA})
    \begin{enumerate}
        \item A plurisubharmonic function $\varphi:U\to \mathbb{R} \cup \{-\infty\}$ on an open subset $U\subseteq Y$ is an upper semi-continuous function which is not identically equal to $-\infty$, and which extends to a plurisubharmonic function under a local embedding $U \hookrightarrow \mathbb{C}^N$. It is strongly plurisubharmonic, resp. $C^0$, resp. $C^{\infty}$ if it extends to a strongly plurisubharmonic, resp. $C^0$, resp. $C^{\infty}$ function in a local embedding. We say $\varphi$ is pluriharmonic if $\varphi$ is a continuous plurisubharmonic function which extends under a local embedding to a pluriharmonic function.

        \item A K\"ahler potential on $Y$ is a family $(U_i,(\varphi_i)_{i\in I})$, where $(U_i)_{i\in I}$ is an open cover of $Y$ and $\varphi_i$ are smooth strictly plurisubharmonic functions such that $\varphi_i-\varphi_j$ is pluriharmonic. We define an equivalence relation on K\"ahler potentials by 
        \begin{equation*}
            (U_i,\varphi_i)_{i\in I} \sim (V_j,\psi_j)_{j\in J} \iff \varphi_i-\psi_j \text{ is pluriharmonic on } U_i \cap V_j \text{ for all } i\in I, \ j\in J.
        \end{equation*}
        A K\"ahler metric on $Y$ is an equivalence class of K\"ahler potentials. 

        \item A positive current on $Y$ is an equivalence class of plurisubharmonic potentials. A positive current $(U_i,\varphi_i)_{i\in I}/\sim$ is said to have locally bounded potentials, resp. continuous local potentials if $\varphi_i$ is locally bounded, resp. continuous. We say a positive current on $Y$ is a K\"ahler current if in addition $\varphi_i$ extend under local embeddings to plurisubharmonic functions satisfying $\sqrt{-1}\partial \overline{\partial} \varphi_i \geq C^{-1}\omega_{\mathbb{C}^n}$ for some $C>0$.

        \item If $\omega_Y$ is a smooth K\"ahler metric on $Y$ with K\"ahler potential $(U_i,\varphi_i)_{i\in I}$, then an upper-semicontinuous function $\varphi:Y\to \mathbb{R} \cup \{-\infty\}$ is called $\omega_Y$-plurisubharmonic if $\varphi_i+\varphi$ is plurisubharmonic on $U_i$ for each $i\in I$. In this case, we let $\omega_Y + \sqrt{-1}\partial \overline{\partial}\varphi$ denote the positive current corresponding to $(U_i,\varphi_i+\varphi)_{i\in I}$. 
        
    \end{enumerate}
\end{definition}

\begin{lemma} \label{lem:Kahlercurrent}
\begin{enumerate}
    \item \label{Kahlercurrent0} Each open embedding $\phi_i:\mathcal{C}(S_i) \supseteq B_{g_{\mathcal{C}(S_i)}}(o,r_0)\setminus \{o\}\to Y$ extends to an open holomorphic embedding of $B_{g_{\mathcal{C}(S_i)}}(o,r_0)$. 

    \item \label{Kahlercurrent1} The smooth K\"ahler metric $\omega_0$ naturally extends to a K\"ahler current on $Y$, which we also denote $\omega_0$.

    \item \label{Kahlercurrent2} There is a smooth K\"ahler form $\omega_Y$ on $Y$ and a continuous $\omega_Y$-plurisubharmonic function $w:Y\to \mathbb{R} \cup \{-\infty\}$ such that $\omega_0 = \omega_Y+\sqrt{-1}\partial \overline{\partial}w$. 
\end{enumerate}
\end{lemma}
\begin{proof} \ref{Kahlercurrent0} By \eqref{eq:estimatesforu1}, $\phi_i$ extends uniquely to an open topological embedding $B_{g_{\mathcal{C}(S_i)}}(o,r_0) \hookrightarrow Y$. The claim then follows by applying the Riemann extension theorem to $\phi_i,\phi_i^{-1}$ composed with local embeddings of neighborhoods of $y_i$, $o$ into $\mathbb{C}^N$, respectively.

\ref{Kahlercurrent1} Because $(Y\setminus \{y_1,...,y_Q\},\omega_0)$ is a smooth K\"ahler manifold, we can use the local $\partial \overline{\partial}$-lemma to find a cover $(U_i,\varphi_i)_{i\in I}$ of $Y\setminus \{y_1,...,y_Q\}$ of K\"ahler potentials $\varphi_i$ for $\omega_0|_{U_i}$. On the other hand, by using $\phi_i$ to identify a small neighborhood of each $y_i$ with a subset of $\mathcal{C}(S_i)$, we can write $\omega_0 = \sqrt{-1}\partial \overline{\partial} (\frac{1}{2}r^2+u_i)$ on a punctured neighborhood of $y_i$. By \cite[Th\'eor\'eme 2]{GraRem}, $\frac{1}{2}r^2+u_i$ extends to a plurisubharmonic function over a neighborhood $V_i$ of $y_i$, so that adding $(V_i,\frac{1}{2}r^2+u_i)$ to the collection of K\"ahler potentials gives the positive current property of $\omega_0$. It therefore suffices to note that any K\"ahler cone metric $(\mathcal{C}(L),g_{\mathcal{C}(L)})$ (where $C$ is a smooth Sasaki manifold) is naturally a K\"ahler current. In fact, any such cone can be embedded in $\mathbb{C}^N$ in a way which is equivariant with respect to the torus generated by its Reeb vector field \cite[Theorem 3.1]{VanCoe} and a linear action with positive weights on $\mathbb{C}^N$. As a consequence, the components of this embedding are Lipschitz with respect to the cone metric, hence with respect to this embedding, $\omega_{\mathcal{C}(L)}\geq C^{-1}\omega_{\mathbb{C}^N}$ near the vertex, for some $C>0$.

\ref{Kahlercurrent2} Given \ref{Kahlercurrent1}, this follows from \cite[Theorem 3.1]{smith}.
\end{proof}

Letting $\omega_Y$ be as in Lemma \ref{lem:Kahlercurrent}, it follows that $\pi^{\ast}\omega_Y$ is a smooth, big, and semi-ample $(1,1)$-form. Moreover, $\pi^{\ast}\omega_0 = \pi^{\ast}\omega_Y + \sqrt{-1}\partial\bar\partial\varphi_0$, where $\varphi_0:=\pi^{\ast}w$ is a continuous $\pi^{\ast}\omega_Y$-plurisubharmonic function on $M$. Fix $t_0 \in (0,T_0)$, and set $\eta:= \frac{1}{t_0}(\omega_{t_0}-\pi^{\ast}\omega_Y) \in -c_1(M)$, so that 
\begin{equation*}
    \theta_t := \pi^{\ast}\omega_Y + t\eta = \frac{t_0-t}{t_0}\pi^{\ast}\omega_Y +\frac{t}{t_0}\omega_{t_0}
\end{equation*}
are nonnegative smooth $(1,1)$-forms for $t\in [0,t_0]$, which are K\"ahler for $t\in (0,t_0]$, and satisfy $\theta_t \geq \frac{1}{2}\pi^{\ast}\omega_Y$ for $t\in [0,\frac{t_0}{2}]$. Choose a smooth volume form $\Omega$ on $M$ satisfying $\Ric(\Omega)=-\eta$. 

The following is a consequence of an existence/uniqueness theorem proved in \cite{BouGue}, which generalized results from \cite{songtian}. 

\begin{theorem}(\cite[Theorem 4.3.3]{BouGue}) \label{thm:songtian}
There is a unique family $(\widetilde{\omega}_t)_{t\in [0,t_0]}$ of positive currents on $M$ such that the following hold:
\begin{enumerate}
    \item \label{conditionforsongtian1} $\widetilde{\omega}_0 = \pi^{\ast}\omega_0$,
    \item \label{conditionforsongtian2} $(\widetilde{\omega}_t)_{t\in (0,t_0]}$ is a smooth K\"ahler--Ricci flow,
    \item \label{conditionforsongtian4} there exists a smooth and bounded function $\varphi:(M\setminus \pi^{-1}(\{y_1,...,y_Q\})\times [0,t_0] \to \mathbb{R}$ satisfying $\varphi(0)=\varphi_0$, $\widetilde \omega_t = \theta_t + \sqrt{-1}\partial \overline{\partial} \varphi(t)$, and solving the parabolic complex Monge-Amp\`ere equation
    \begin{equation} \label{eq:CMAforuniqueness}
        \frac{\partial \varphi}{\partial t} = \log \left( \frac{(\theta_t+\sqrt{-1}\partial \overline{\partial}\varphi(t))^n}{\Omega}\right).
    \end{equation}
\end{enumerate}
\end{theorem}

Using Theorem \ref{thm:songtian}, we now identify the flow constructed in Theorem \ref{main theorem} with that constructed in \cite{songtian}.

\begin{prop} \label{prop:uniqueness} The flow $(\omega(t))_{t\in [0,t_0]}$ constructed in Theorem \ref{main theorem} coincides with the flow $(\widetilde{\omega}_t)_{t\in [0,t_0]}$ of Theorem \ref{thm:songtian}.
\end{prop}
\begin{proof} Observe that the flow $(\omega_t)_{t\in [0,t_0]}$ from Theorem \ref{main theorem} satisfies conditions \ref{conditionforsongtian1} and \ref{conditionforsongtian2}, it suffices to verify that \ref{conditionforsongtian4} holds for some $\varphi \in C^{\infty}(M \setminus \pi^{-1}(y_1,...,y_Q\})$. We define 
\[\varphi_t:=\varphi_0+\int_0^t\log\frac{\omega^n(\tau)}{\Omega}d\tau,\]
for all $t\in (0,t_0]$. Then $\omega_0 = \theta_0 + \sqrt{-1}\partial \overline{\partial}\varphi_0$, and we have
\begin{equation*}
    \frac{\partial}{\partial t}(\omega(t) - \theta_t - \sqrt{-1}\partial \overline{\partial}\varphi(t)) = -\Ric(\omega(t)) - \eta -\sqrt{-1}\partial \overline{\partial}\log\frac{\omega(t)^n}{\Omega}=0,
\end{equation*}
so that $\omega(t) = \theta_t + \sqrt{-1}\partial \overline{\partial} \varphi(t)$ on $(M\setminus \pi^{-1}(\{y_1,...,y_Q\})\times [0,t_0]$. It therefore suffices to show that $\int_0^t\log\frac{\omega_\tau^n}{\Omega}d\tau$ is uniformly bounded on $M\setminus\pi^{-1}(y_1)\times (0,t_0]$. We compute for all $t\in (0,t_0]$
\begin{equation*}
    \int_0^t \log\frac{\omega^n(\tau)}{\omega^n(t_0)}d\tau=-\int_0^t \int_\tau^{t_0}\frac{\partial}{\partial\mu}(\log\omega^n(\mu))d\mu d\tau=\int_0^t\int_\tau^{t_0}R_{\omega(\mu)}d\mu d\tau.
\end{equation*}
Hence
\begin{equation*}
     \left|\int_0^t\log\frac{\omega^n(\tau)}{\omega^n(t_0)}d\tau \right|\le C\int_0^t (|\log t_0| +|\log\tau|)d\tau.
\end{equation*}
Therefore, $\int_0^t\log\frac{\omega^n(\tau)}{\Omega}d\tau$ is uniformly bounded.
\end{proof}
\begin{remark}
   For any K\"ahler–Ricci flow appearing in Theorem \ref{main theorem} \ref{thm:mainthm5}, because $R_{\widetilde{\omega}_t}\geq -\frac{n}{t}$ automatically, the same argument as in the proof of Theorem \ref{prop:uniqueness} shows that $(\widetilde{\omega}_t)$ coincides with the flow constructed in Proposition \ref{thm:songtian}. Consequently, it agrees with the flow constructed in Theorem \ref{main theorem}.
\end{remark}

\subsection{Examples}

In this subsection, we construct examples of smooth K\"ahler Ricci flows with singular initial data and prescribed asymptotics near a given point in spacetime.

\begin{proof}[Proof of Corollary \ref{cor:examples}] Let $(\mathcal{C},g_{\mathcal{C}})$ be the asymptotic cone of $(E,g_E)$. Let $o \in \mathcal{C}$ and $r$ denote the radius function. By \cite[Proof of Theorem 3.1]{VanCoe}, there exist holomorphic functions $u_1,...,u_N \in \mathcal{O}_{\mathcal{C}}(\mathcal{C})$ such that $F:=(u_1,...,u_N):\mathcal{C}\to \mathbb{C}^N$ is a proper holomorphic embedding, and for each $i\in \{1,...,N\}$, there exists $\lambda_i \in (0,\infty)$ such that
\begin{equation*}
    \mathcal{L}_{r\partial_r} u_i=\lambda_i u_i,
\end{equation*}
where $\partial_r$ is the radial vector field of the cone. 
Let 
\begin{equation*}
    \Phi:\mathbb{C}^N \hookrightarrow \mathbb{P}^N, (z_1,...,z_N)\mapsto [1:z_1:\cdots :z_N]
\end{equation*}
be the standard embedding, and let $\omega_{FS}$ be the Fubini-Study metric on $\mathbb{P}^N$, so that 
\begin{equation*}
    (\Phi\circ F)^{\ast}\omega_{FS} = \sqrt{-1}\partial \overline{\partial}\varphi_{FS}, \qquad \varphi_{FS}:=\log \left( 1+\sum_{i=1}^N |u_i|^2 \right).
\end{equation*}
We now fix a nonnegative function $\theta \in C_c^{\infty}(\mathbb{R})$ with $-1\leq \theta \leq 1$, $\int_{\mathbb{R}}\theta(t)dt=1$ and $\int_{\mathbb{R}}t\theta(t)dt=0$. We then recall the definition of the regularized maximum function from \cite[Lemma I.5.18]{Demailly2012}, with parameter $\eta=(\frac{1}{2},\frac{1}{2})$:
\begin{equation*}
M_{\eta}(t_1,t_2):=\frac{1}{\eta_1 \eta_2}\int_{\mathbb{R}^2} \max\{t_1+h_1,t_2+h_2\}\theta\left(\frac{h_1}{\eta_1}\right) \theta \left( \frac{h_2}{\eta_2}\right) dh_1dh_2.
\end{equation*}
By \cite[Lemma I.5.18(e)]{Demailly2012}, it follows that for any $A\in (0,\infty)$,
\begin{equation*}
    \varphi := M_{\eta}(A\varphi_{FS}-2,\frac{1}{2}r^2) : B_{g_{\mathcal{C}}}(o,2)\to \mathbb{R}
\end{equation*}
is strictly plurisubharmonic. Since $\varphi_{FS}$ is a continuous function satisfying $\varphi_{FS}(o)=0$ and $\min_{\frac{1}{2}\leq r \leq 2}\varphi_{FS}>0$, we can choose $A \in (0,\infty)$ such that $A\varphi_{FS}-2\geq 4$ on $\{\frac{1}{2} \leq r \leq 2\}$. 
It follows from \cite[Lemma I.5.18(c)]{Demailly2012} that $\varphi=\frac{1}{2}r^2$ on a neighborhood of $o$ while $\varphi=A\varphi_{FS}-2$ on $\{\frac{1}{2}<r<2\}$. Thus
\begin{equation*}
    \widetilde{\varphi}:= \begin{cases}
        \varphi, &  r<2\\
        A\varphi_{FS}-2, & r>\frac{1}{2},
    \end{cases}
\end{equation*}
is a strictly plurisubharmonic function on $\mathcal{C}$, smooth on $\mathcal{C}\setminus \{o\}$, hence $\widetilde{\omega}:= \sqrt{-1}\partial \overline{\partial}\widetilde{\varphi}$ defines a K\"ahler current on $\mathcal{C}$ which agrees with $A(\Phi\circ F)^{\ast}\omega_{FS}$ outside a compact set. In particular, $(\Phi \circ F)_{\ast}\widetilde{\omega}$ extends to a K\"ahler current $\widetilde{\omega}_0$ on the projective closure $\overline{\mathcal{C}}$ of $\Phi(\mathcal{C})$, which is smooth on $\overline{\mathcal{C}}\setminus \{o\}$, and equal to $\omega_{\mathcal{C}}$ on a neighborhood of $o$. 

Finally, by \cite[Proof of Lemma 2.2]{CMM}, there is a normal projective variety $Y$ and a map $\sigma:Y\to \overline{\mathcal{C}}$ satisfying the following:
\begin{enumerate}
    \item $\sigma$ is a composition of blowups along smooth centers,

    \item there is a neighborhood $U \subseteq \overline{\mathcal{C}}$ of $o$ such that $\sigma|_{\sigma^{-1}(U)}:\sigma^{-1}(U)\to U$ is a biholomorphism,

    \item $Y\setminus \sigma^{-1}(o)$ is smooth,
    
    \item there is a K\"ahler current $\omega_0$ on $Y$ which is smooth on $Y\setminus \sigma^{-1}(o)$ and is equal to $\sigma^{\ast}\widetilde{\omega}_0$ on a neighborhood of $\sigma^{-1}(o)$. 
\end{enumerate}

It follows that $(Y,\omega_0)$ is a compact analytic space with a single isolated conical singularity at $ \sigma^{-1}(o)$, modeled on the K\"ahler cone $(\mathcal{C},\omega_{\mathcal{C}})$. The existence of a flow $(M,(g(t))_{t\in (0,T]})$ with the desired properties is therefore a consequence of Theorem \ref{main theorem}.
\end{proof}

\bibliographystyle{alpha}
\bibliography{references}
\end{document}

%% file: normalised_spacetime.tex
\tikzset{every picture/.style={line width=0.75pt}} 

\begin{tikzpicture}[x=0.75pt,y=0.75pt,yscale=-1,xscale=1]

\draw  [color={rgb, 255:red, 255; green, 255; blue, 255 }  ,draw opacity=1 ][fill={rgb, 255:red, 80; green, 227; blue, 194 }  ,fill opacity=1 ] (56.6,355.2) -- (271,261) -- (56.6,261) -- cycle ;
\draw  (25,355.2) -- (341,355.2)(56.6,127.5) -- (56.6,380.5) (334,350.2) -- (341,355.2) -- (334,360.2) (51.6,134.5) -- (56.6,127.5) -- (61.6,134.5)  ;
\draw [line width=0.75]    (57,182) -- (307,182) ;
\draw [line width=0.75]    (57,261) -- (307,261) ;
\draw [color={rgb, 255:red, 208; green, 2; blue, 27 }  ,draw opacity=1 ][line width=1.5]    (56.6,355.2) -- (271,261) ;
\draw  [dash pattern={on 4.5pt off 4.5pt}]  (271,261) -- (271,355.5) ;
\draw  [color={rgb, 255:red, 255; green, 255; blue, 255 }  ,draw opacity=1 ][fill={rgb, 255:red, 80; green, 227; blue, 194 }  ,fill opacity=1 ] (409.6,357.2) -- (624,263) -- (409.6,263) -- cycle ;
\draw  (378,357.2) -- (694,357.2)(409.6,129.5) -- (409.6,382.5) (687,352.2) -- (694,357.2) -- (687,362.2) (404.6,136.5) -- (409.6,129.5) -- (414.6,136.5)  ;
\draw [line width=0.75]    (410,184) -- (660,184) ;
\draw [line width=0.75]    (410,263) -- (660,263) ;
\draw [color={rgb, 255:red, 208; green, 2; blue, 27 }  ,draw opacity=1 ][line width=1.5]    (409.6,357.2) -- (624,263) ;
\draw  [dash pattern={on 4.5pt off 4.5pt}]  (624,263) -- (624,357.5) ;

\draw (147,153.4) node [anchor=north west][inner sep=0.75pt]    {$r_M^{2} =r_{0}^{2}$};
\draw (145,234.4) node [anchor=north west][inner sep=0.75pt]    {$r_M^{2} =R^{2}$};
\draw (165.8,311.5) node [anchor=north west][inner sep=0.75pt]    {$r_M^{2} =\lambda t$};
\draw (70,285) node [anchor=north west][inner sep=0.75pt]   [align=left] {Conical region};
\draw (-17,144) node [anchor=north west][inner sep=0.75pt]   [align=left] {space $\displaystyle r^{2}$};
\draw (302,366) node [anchor=north west][inner sep=0.75pt]   [align=left] {time $\displaystyle t$};
\draw (233,356.4) node [anchor=north west][inner sep=0.75pt]    {$t=\frac{R^{2}}{\lambda }$};
\draw (501,147.4) node [anchor=north west][inner sep=0.75pt]    {$r^{2} =\frac{r_{0}^{2}}{s}$};
\draw (500,230.4) node [anchor=north west][inner sep=0.75pt]    {$r^{2} =\frac{R^{2}}{s}$};
\draw (518.8,313.5) node [anchor=north west][inner sep=0.75pt]    {$r^{2} =\lambda t$};
\draw (416,275) node [anchor=north west][inner sep=0.75pt]   [align=left] {Normalised\\conical region};
\draw (336,146) node [anchor=north west][inner sep=0.75pt]   [align=left] {space $\displaystyle r^{2}$};
\draw (650,367) node [anchor=north west][inner sep=0.75pt]   [align=left] {time $\displaystyle t$};
\draw (585,357.4) node [anchor=north west][inner sep=0.75pt]    {$t=\frac{R^{2}}{\lambda s}$};

\end{tikzpicture}

%% file: modified_spacetime.tex
 
\tikzset{
pattern size/.store in=\mcSize, 
pattern size = 5pt,
pattern thickness/.store in=\mcThickness, 
pattern thickness = 0.3pt,
pattern radius/.store in=\mcRadius, 
pattern radius = 1pt}
\makeatletter
\pgfutil@ifundefined{pgf@pattern@name@_g0w2vfdc5}{
\pgfdeclarepatternformonly[\mcThickness,\mcSize]{_g0w2vfdc5}
{\pgfqpoint{0pt}{0pt}}
{\pgfpoint{\mcSize+\mcThickness}{\mcSize+\mcThickness}}
{\pgfpoint{\mcSize}{\mcSize}}
{
\pgfsetcolor{\tikz@pattern@color}
\pgfsetlinewidth{\mcThickness}
\pgfpathmoveto{\pgfqpoint{0pt}{0pt}}
\pgfpathlineto{\pgfpoint{\mcSize+\mcThickness}{\mcSize+\mcThickness}}
\pgfusepath{stroke}
}}
\makeatother
\tikzset{every picture/.style={line width=0.75pt}} 

\begin{tikzpicture}[x=0.75pt,y=0.75pt,yscale=-1,xscale=1]

\draw  [color={rgb, 255:red, 255; green, 255; blue, 255 }  ,draw opacity=1 ][fill={rgb, 255:red, 248; green, 231; blue, 28 }  ,fill opacity=1 ] (153.1,274.53) -- (542,274.53) -- (542,417.85) -- (153.1,417.85) -- cycle ;
\draw  (105,417.85) -- (586,417.85)(153.1,61) -- (153.1,457.5) (579,412.85) -- (586,417.85) -- (579,422.85) (148.1,68) -- (153.1,61) -- (158.1,68)  ;
\draw    (153,109) -- (542,207.5) ;
\draw [color={rgb, 255:red, 208; green, 2; blue, 27 }  ,draw opacity=1 ][line width=1.5]    (153.02,274.53) -- (541.22,274.53) ;
\draw    (541.22,274.53) -- (542,417.85) ;
\draw  [color={rgb, 255:red, 0; green, 0; blue, 0 }  ,draw opacity=1 ][pattern=_g0w2vfdc5,pattern size=3.75pt,pattern thickness=0.75pt,pattern radius=0pt, pattern color={rgb, 255:red, 155; green, 155; blue, 155}] (153.1,172) -- (541.22,274.53) -- (541.22,315.32) -- (153.1,417.85) -- cycle ;

\draw (471,141.4) node [anchor=north west][inner sep=0.75pt]    {$r^{2} =\frac{r_{0}^{2}}{se^{\tau }}$};
\draw (471,212.4) node [anchor=north west][inner sep=0.75pt]    {$r^{2} =\frac{R^{2}}{se^{\tau }}$};
\draw (86,78) node [anchor=north west][inner sep=0.75pt]   [align=left] {space $\displaystyle r^{2}$};
\draw (594,408) node [anchor=north west][inner sep=0.75pt]   [align=left] {time $\displaystyle \tau $};
\draw (101,263.4) node [anchor=north west][inner sep=0.75pt]    {$r^{2} =\lambda $};
\draw (432,360.4) node [anchor=north west][inner sep=0.75pt]    {$r^{2} =\frac{\lambda \left( e^{\tau } -1\right)}{e^{\tau }}$};
\draw (476,424.4) node [anchor=north west][inner sep=0.75pt]    {$\tau =\log\left(\frac{R^{2}}{\lambda s}\right)$};
\draw (293,393) node [anchor=north west][inner sep=0.75pt]   [align=left] {\textit{Expanding region}};

\end{tikzpicture}